\documentclass[a4paper, 9pt]{article}
\usepackage{graphics}
\usepackage{hyperref}
\usepackage{color}

\usepackage{lscape}
\usepackage{latexsym}
\usepackage{amsmath,verbatim}
\usepackage{graphics}
\usepackage{amsthm}
\usepackage{amssymb}
\usepackage[mathscr]{euscript}
\usepackage{yfonts}
\usepackage{makeidx}
\usepackage{stmaryrd}
\usepackage{multicol}
\usepackage{bm}
\usepackage[all]{xy}
\usepackage[normalem]{ulem}
\numberwithin{equation}{section}
\newtheorem{thm}{Theorem}[section]
\newtheorem{prop}[thm]{Proposition}
\newtheorem{lem}[thm]{Lemma}
\newtheorem{cor}[thm]{Corollary}
\newtheorem{claim}{Claim}{\bf}{\it}

\newtheorem{fthm}{Theorem}{\bf}{\it}
{\bf}{\it}
\newtheorem{fcor}[fthm]{Corollary}{\bf}{\it}
{\bf}{\it}
\newtheorem{fques}[fthm]{Open Question}{\bf}{\it}

\theoremstyle{definition}
\newtheorem{defn}[thm]{Definition}

{\bf}{\rm}

\theoremstyle{remark}

\newtheorem{rem}[thm]{Remark}
{\bf}{\it}

\newtheorem{definition and corollary}[thm]{Definition and Corollary}

{\it}{\rm}

\newcommand{\A}{{\mathbb A}}

\newcommand{\al}{\alpha}
\newcommand{\af}{\mathrm{af}}
\newcommand{\op}{\mathrm{op}}

\newcommand{\C}{{\mathbb C}}

\newcommand{\cO}{{\mathcal O}}

\newcommand{\GW}{\mathtt{GW}}

\newcommand{\Hom}{\mbox{\rm Hom}}

\newcommand{\End}{\mbox{\rm End}}

\newcommand{\bI}{{\mathbf I}}
\newcommand{\bh}{{\mathbf h}}

\newcommand{\ch}{\mathrm{ch}}

\newcommand{\gch}{\mathrm{gch}}
\newcommand{\gdim}{\mathrm{gdim}}

\newcommand{\la}{\lambda}
\newcommand{\lo}{\mathrm{loc}}
\newcommand{\ra}{\mathrm{rat}}

\newcommand{\Fl}{\mathrm{Fl}}
\newcommand{\Gr}{\mathrm{Gr}}

\newcommand{\Spec}{\mbox{\rm Spec}}

\newcommand{\MID}{\! \! \mid}

\newcommand{\g}{\mathfrak{g}}
\newcommand{\gb}{\mathfrak{b}}
\newcommand{\sB}{\mathscr B}
\newcommand{\sGB}{\mathscr{GB}}

\newcommand{\tI}{\mathtt{I}}

\newcommand{\bO}{\mathbb{O}}

\renewcommand{\P}{\mathbb{P}}

\newcommand{\bQ}{\mathbf{Q}}

\newcommand{\sA}{\mathscr{A}}
\newcommand{\sC}{\mathscr{C}}
\newcommand{\sS}{\mathscr{S}}

\newcommand{\sQ}{\mathscr{Q}}
\newcommand{\sH}{\mathscr{H}}
\newcommand{\bW}{\mathbb{W}}
\newcommand{\sX}{\mathscr{X}}

\newcommand{\sZ}{\mathscr{Z}}
\newcommand{\Q}{\mathbb{Q}}
\newcommand{\R}{\mathbb{R}}

\newcommand{\Z}{\mathbb{Z}}
\newcommand{\si}{\frac{\infty}{2}}

\newcommand{\Gm}{\mathbb G_m}

\newcommand\ringring[1]{%
  {
   \mathop{\kern0pt #1}\limits^{
     \vbox to-1.85ex{
       \kern-2ex 
       \hbox to 0pt{\hss\normalfont\kern.1em \r{}\kern-.45em \r{}\hss}%
       \vss 
     }
   }
  }
}

\title{Loop structure on equivariant $K$-theory of semi-infinite flag manifolds}
\author{Syu \textsc{Kato}\footnote{Department of Mathematics, Kyoto University, Oiwake Kita-Shirakawa Sakyo Kyoto 606-8502 JAPAN \tt{E-mail:syuchan@math.kyoto-u.ac.jp}}}

\begin{document}
\maketitle

\begin{abstract}
We explain that the Pontryagin product structure on the equivariant $K$-group of an affine Grassmannian considered in [Lam-Schilling-Shimozono, Compos. Math. {\bf 146} (2010)] coincides with the tensor structure on the equivariant $K$-group of a semi-infinite flag manifold considered in [K-Naito-Sagaki, Duke Math. {\bf 169} (2020)]. Then, we construct an explicit isomorphism between the equivariant $K$-group of a semi-infinite flag manifold and a suitably localized equivariant quantum $K$-group of the corresponding flag manifold. These exhibit a new framework to understand the ring structure of equivariant quantum $K$-groups and the Peterson isomorphism.
\end{abstract}

\section*{Introduction}

Let $G$ be a simply connected simple algebraic group over $\C$ with a maximal torus $H$. Let $\Gr$ denote its affine Grassmannian and let $\sB$ be its flag variety.

Following the seminal work of Peterson \cite{Pet97} (on the quantum cohomology, see also Lam-Shimozono \cite{LS10}), there were many efforts to understand the (small) quantum $K$-group $qK ( \sB )$ of $\sB$ in terms of the $K$-group $K ( \Gr )$ of affine Grassmannians (see \cite{LSS10,LLMS17} and the references therein). One of its forms, borrowed from Lam-Li-Mihalcea-Shimozono \cite{LLMS17}, is a (conjectural) ring isomorphism:
\begin{equation}
K _H ( \Gr ) _\lo \cong qK_H ( \sB )_\lo,\label{KPeterson}
\end{equation}
where subscript $H$ indicate the $H$-equivariant version and the subscript $\lo$ denotes certain localizations. Here the multiplication in $K _H ( \Gr ) _\lo$ is the {\it Pontryagin product}, that differs from the usual action of the $K$-group of the thick affine Grassmannian (that one may internalize using the perfect pairing \cite{Kum17} or the identification with the topological $K$-group \cite{KK90}), while the multiplication of $qK_H ( \sB )_\lo$ is standard in quantum $K$-theory \cite{Giv00,Lee04}.

On the other hand, we have another version $\bQ_G^{\ra}$ of affine flag variety of $G$, called the semi-infinite flag variety (\cite{FF,FM99,FFKM}). Almost from the beginning \cite{Giv94}, it was expected that $\bQ_G^{\ra}$ have some relation with the quantum cohomology of $\sB$. In fact, we can calculate the equivariant $K$-theoretic $J$-function of $\sB$ using $\bQ_G^{\ra}$ (\cite{GL03,BF14a}), and the reconstruction theorem \cite{LP04,IMT15} tells us that this essentially recovers the ring structure of the (big) quantum $K$-group of $\sB$.

In \cite{KNS17}, we have defined and calculated the equivariant $K$-group of $\bQ_G^{\ra}$, that is also expected to have some relation to $qK_H ( \sB )$, and hence also to $K_H ( \Gr )$. The goal of this paper is to tell the exact relations as follows:

\begin{fthm}[$\doteq$ Theorem \ref{main}]\label{fmain}
Each of $K _H ( \Gr  )_{\lo}$ and $K_{H} ( \bQ_G^{\ra} )$ admits an action of a variant $\sH$ of the double affine Hecke algebra and the coroot lattice $Q^{\vee}$ of $G$. It gives rise to a dense embedding
$$\Phi : K _H ( \Gr )_{\lo} \hookrightarrow K_{H} ( \bQ_G^{\ra} )$$
of $(\sH, Q^{\vee})$-bimodules that sends the Pontryagin product on the LHS to the tensor product on the RHS.
\end{fthm}

Here we note that the topology of $K_{H} ( \bQ_G^{\ra} )$ arises from the Schubert stratification of $\bQ_G^\ra$, and its role in Theorem \ref{fmain} is minor. By transplanting the path model of $K_{H} ( \bQ_G^{\ra} )$, Theorem \ref{fmain} yields multiplication formulae of the classes in $K _H ( \Gr )_{\lo}$ (\cite{KNS17,NOS18}).

Our strategy to prove Theorem \ref{fmain} is as follows: the $\sH \otimes \C Q^{\vee}$-module $K _H ( \mathrm{Gr} _G )_{\mathrm{loc}}$ is cyclic. Hence, its $\sH \otimes \C Q^{\vee}$-endomorphism is determined by the image of a cyclic vector. Moreover, the tensor product action of an equivariant line bundle on $K_{H} ( \bQ_G^{\mathrm{rat}} )$ yields a $\sH \otimes \C Q^{\vee}$-endomorphism. These make it possible to identify important parts of the Pontryagin action on the LHS that gives a $\sH \otimes \C Q^{\vee}$-endomorphism with the tensor product action on the RHS.

The other part of the exact relation we exhibit is:

\begin{fthm}[$\doteq$ Corollary \ref{Jcomp} and Theorem \ref{b-trans}]\label{fJcomp}
We have a $K_H ( \mathrm{pt})$-module isomorphism
$$\Psi : qK _H ( \sB )_\lo \stackrel{\cong}{\longrightarrow} K_{H} ( \bQ_G^{\ra} )$$
that sends the quantum product of a primitive anti-nef line bundle to the tensor product of the corresponding line bundle.\\
Moreover, $\Psi$ sends a Schubert class of the LHS to a Schubert class in the RHS, and intertwines the Novikov variable twist in the LHS to the right translation of the Schubert classes in the RHS. In particular, the topology of $qK _H ( \sB )_\lo$ with respect to the Novikov variables is compatible with the topology of $K_{H} ( \bQ_G^{\ra} )$ through $\Psi$.
\end{fthm}

Here we remark that a priori $K_{H} ( \bQ_G^{\ra} )$ is not a ring (see Remark \ref{Kra}).

Unlike Theorem \ref{fmain}, Theorem \ref{fJcomp} is best understood by its completed topological form as the inverse quantum multiplication of an anti-nef line bundle corresponds to the tensor product of a nef line bundle through $\Psi$. The latter tensor product action on $K_H ( \bQ_G^\ra )$ is quite natural, and its structure constants with respect to the Schubert classes are positive (\cite[Theorem 5.11]{KNS17}). However, it lives genuinely in the completions in general (as the sum is infinite; see \S \ref{subsec: sl(2)}).

Such tensor products (with infinitely many nonzero structure constants) play a central role in our proof of Theorem \ref{fJcomp}. To analyze them, we need to include an extra $q$-variable that is responsible for the $\Gm$-action on a curve $\P^1$ in both sides. In particular, our proof of Theorem \ref{fJcomp} is in fact the $q = 1$ specialization of an isomorphism
$$\Psi_q : \C [q^{\pm 1}] \otimes qK _H ( \sB )_\lo \cong K_{\Gm \times H} ( \bQ_G^{\mathrm{rat}} ),$$
that intertwines shift operators (of line bundles on $\sB$) and line bundle twists (on $\bQ_G^{\ra}$). Combining Theorems \ref{fmain} and \ref{fJcomp}, we conclude:

\begin{fcor}[$\doteq$ Corollary \ref{triangle}]\label{fLLMS}
We have a commutative diagram, whose bottom arrow is a natural embedding of rings:
$$
\xymatrix{
& K_{H} ( \bQ_G^{\ra} ) & \\
K _H ( \Gr ) _\lo \ar@{^{(}->}[rr]^{(\ref{KPeterson})} \ar@{^{(}->}[ru]^{\Phi} & & qK_H ( \sB )_\lo \ar[lu]_{\Psi}^{\cong}
}.$$
Here the uncompleted version of $qK_H ( \sB )_\lo$ is isomorphic to $K _H ( \Gr ) _\lo$, the map $\Phi$ is an injective $K_H ( \mathrm{pt} ) \otimes \C Q^{\vee}$-module homomorphism, and $K_{H} ( \bQ_G^{\ra} )$ acquires the structure of a ring from $K_H ( \Gr )$ or $qK_H ( \sB )$.
\end{fcor}

The explicit nature of Corollary \ref{fLLMS} verifies conjectures in \cite{LLMS17} (Corollary \ref{LLMSc}). In the same vein, we find that the structure constants of quantum multiplications, as well as the shift operator actions, on $qK_H ( \sB )$ are finite (Corollary \ref{LLMSfin} and Corollary \ref{qfin}; see also Anderson-Chen-Tseng \cite{ACT18}). Therefore, this paper also provides an indispensable step in the proof \cite{ACT18} of the finiteness of the multiplication of quantum $K$-groups of partial flag manifolds, that stood as a fundamental problem in the quantum $K$-theoretic Schubert calculus from the beginning.

Theorem \ref{fmain}, and hence Corollary \ref{fLLMS}, also have $\Gm$-equivariant versions by supplementing cosmetic arguments to the results presented in this paper that we exhibit in \cite{Kat20} together with its representation-theoretic consequences.

The idea of the construction of $\Psi$ in Theorem \ref{fJcomp} is to compare the structure of the both sides via the asymptotic behavior of the cohomology of quasi-map spaces with respect to the degree of curves. Adapting the technicality on the topology and the $q$-variables discussed above, it is rather natural to consider such a thing if we know the ``cohomological invariance" between two models of semi-infinite flag manifolds proved in \cite{BF14b,KNS17}, the reconstruction theorem in the form of \cite{IMT15}, and the $J$-function calculations in \cite{GL03, BF14a}. In order to show that $\Psi$ respects products (Theorem \ref{b-trans}), we need to analyze the geometry of graph spaces and quasi-map spaces. Such an analysis is based on the identification of natural subvarieties of quasi-map spaces with the scheme-theoretic intersection of orbit closures in $\bQ_G^\ra$ with respect to two mutually opposite (Iwahori) subgroups $\bI$ and $\bI^-$ of $G(\!(z)\!)$, that we call Richardson varieties of $\bQ_G^\ra$ (\cite{Kat18d}). With this in hands, the core of the proof of our assertion reduces to a generalization of the following fact from the case of $\sB$ to the case of $\bQ_G^\ra$: the singularity types of Richardson varieties come from the singularity types between two strata in the Bruhat stratification. However, there are some pit-falls to carry this out (see Remark \ref{oppSchubert} and \S \ref{subsec:fact} for related accounts), and our proof employs the factorizaton property to identify the transversal slices arising from Richardson varieties of $\bQ_G^\ra$ and the Bruhat stratification of $\bQ_G^\ra$. The outcome of our analysis includes:

\begin{fthm}[$\doteq$ Theorem \ref{Q-rat-sing}]
Richardson varieties of $\bQ_G^\ra$ have rational singularities and are Cohen-Macaulay\footnote{Previous versions of this paper contained proofs of Theorem \ref{Q-rat-sing} with gaps. To clarify the whole point, the author decided to separate out the proof of the normality and other related technical results into \cite{Kat18d} (see Theorem \ref{Qnorm}).}.
\end{fthm}

Note that $\bQ_G^{\mathrm{rat}}$ is the universal indscheme associated to the formal loop space of $\sB$ (\cite{Kat18d} see also \cite{BF14a,KNS17}). Hence, it is tempting to spell out the following, that unifies the proposals by Givental \cite[\S 4]{Giv94} (cf. Iritani \cite{Iri06}), Peterson \cite{Pet97} (cf. \cite{LLMS17}), and Arkhipov-Kapranov \cite[\S 6.2]{AK06}:

\begin{fques}\label{lconj}
Let $X$ be a smooth projective Fano variety with an action of an algebraic group $H$. Let $\mathcal L X$ be the formal loop space of $X$ $($see $\cite{AK06})$. Then, we have an inclusion that intertwines the quantum product and tensor product of primitive anti-nef line bundles:
$$\Psi_X : qK _H ( X ) \hookrightarrow K_H ( \mathcal L X ),$$
where we define $K_H ( \mathcal L X )$ as the $q=1$ specialization of a subspace of $K_{\Gm \times H} ( \mathcal L X )$ $($cf. $\S \ref{subsec:eK}$ and $\cite{KNS17})$.
\end{fques}


The organization of this paper is as follows: In section one, we recall some basic results from previous works (needed to formulate Theorems \ref{fmain} and \ref{fJcomp}), and prove some complementary results including the definition of $K_H (\bQ_G^\ra)$. In section two, we formulate and prove the precise version of Theorem \ref{fmain} and exhibit its $\mathop{SL} ( 2 )$-example. In section three, we make recollections on quasi-map spaces and $J$-functions, and construct the map $\Psi$ following ideas of \cite{GL03,BF14a,BF14b,IMT15} using results from \cite{KNS17}. At the same time, we prove Theorem \ref{fJcomp} modulo the behavior of Schubert basis (Corollary \ref{Jcomp}). In section four, we first prove that Richardson varieties of $\bQ_G^\ra$ have only rational singularities (Theorem \ref{Q-rat-sing}), using a detailed analysis of the transversal slices. Then, we prove Theorem \ref{b-trans} about the behavior of bases under $\Psi$. In conjunction with Corollary \ref{Jcomp}, this completes the proof of Theorem \ref{fJcomp} in its precise from. After that, we exhibit the consequences of our results in \S \ref{subsec:conseq}, including the proof of the Lam-Li-Mihalcea-Shimozono conjecture, the finiteness of the quantum multiplications for $\sB$, and the finiteness of the shift operators.

The results of this paper is supported by our previous works on semi-infinite flag manifolds and representation theory of current algebras, including \cite{KL17,Kat18,KNS17,Kat18d}. An overview of the whole project, as well as brief accounts on this paper and \cite{Kat19a,Kat20}, can be found in the survey \cite{Kat21}.

Finally, a word of caution is in order. The equivariant $K$-groups dealt in this paper are {\it not} identical to these dealt in \cite{LSS10} and \cite{KNS17} in the sense that both groups are just dense subset (or intersects with a dense subset) in the original $K$-groups (the both groups are suitably topologized). The author does not try to complete this point as he believes it not essential.

\section{Preliminaries}\label{sec:prelim}

A vector space is always a $\C$-vector space, and a graded vector space refers to a $\Z$-graded vector space whose graded pieces are finite-dimensional and its grading is bounded from the above or bounded from the below. Tensor products are taken over $\C$ unless stated otherwise. We define the graded dimension of a graded vector space as the formal sum
$$\mathrm{gdim} \, M := \sum_{i\in \Z} q^i \dim _{\C} M_i.$$
For a (possibly operator-valued) rational function $f ( q )$ on $q$, we set $\overline{f ( q )} := f ( q^{-1} )$.
In this paper, a variety is a separated integral scheme of finite type over $\C$. Let $\C[\![z]\!]$ be the formal power series ring with its variable $z$, and let $\C(\!(z)\!)$ be the fraction field of $\C[\![z]\!]$, the field of formal Laurent series.

\subsection{Groups, root systems, and Weyl groups}\label{subsec:prelim}
Basically, material presented in this subsection can be found in \cite{CG97, Kum02}.

Let $G$ be a connected, simply connected simple algebraic group of rank $r$ over $\C$, and let $B$ and $H$ be a Borel subgroup and a maximal torus of $G$ such that $H \subset B$. We set $N$ $(= [B,B])$ to be the unipotent radical of $B$ and let $N^-$ be the opposite unipotent subgroup of $N$ with respect to $H$. We set $B^- := H N^-$. We denote the Lie algebra of an algebraic group by the corresponding German small letter. We have a (finite) Weyl group $W := N_G ( H ) / H$. For an algebraic group $E$, we denote its set of $\C [z]$-valued points by $E [z]$, its set of $\C [\![z]\!]$-valued points by $E [\![z]\!]$, and its set of $\C (\!(z)\!)$-valued points by $E (\!(z)\!)$. Let $\mathbf I \subset G [\![z]\!]$ be the preimage of $B \subset G$ via the evaluation of $G[\![z]\!]$ at $z = 0$ (the Iwahori subgroup of $G [\![z]\!]$). We also define a subgroup $\bI^{-} \subset G [z^{-1}]$ as the preimage of $B^-$ via the evaluation of $G [z^{-1}]$ at $z = \infty$. Here we warn that while $E [\![z]\!]$ can be understood as a(n infinite type) group scheme, the group $E[z]$ is a group ind-scheme in general (we refer \cite[Chap. I\!V]{Kum02} for basics on ind-schemes).

Let $P := \mathrm{Hom} _{gr} ( H, \Gm )$ be the weight lattice of $H$, let $\Delta \subset P$ be the set of roots, let $\Delta_+ \subset \Delta$ be the set of roots that yield root subspaces in $\gb$, and let $\Pi \subset \Delta _+$ be the set of simple roots. We set $\Delta_- := - \Delta_+$ and $Q_+ := \sum_{\al \in \Pi}\Z_{\ge 0} \al \subset P$. Let $Q^{\vee}$ be the dual lattice of $P$ with a natural pairing $\left< \bullet, \bullet \right> : Q^{\vee} \times P \rightarrow \Z$. We define $\Pi^{\vee} \subset Q ^{\vee}$ to be the set of positive simple coroots, and let $Q_+^{\vee} \subset Q ^{\vee}$ be the set of non-negative integer span of $\Pi^{\vee}$. For $\beta, \gamma \in Q^{\vee}$, we define $\beta \ge \gamma$ if and only if $\beta - \gamma \in Q^{\vee}_+$. For $\lambda, \mu \in P$, we define $\la \ge \mu$ if and only if $\lambda - \mu \in Q_+$. We set $P_+ := \{ \lambda \in P \mid \left< \alpha^{\vee}, \lambda \right> \ge 0, \hskip 2mm \forall \alpha^{\vee} \in \Pi^{\vee} \}$ and $P_{++} := \{ \lambda \in P \mid \left< \alpha^{\vee}, \lambda \right> > 0, \hskip 2mm \forall \alpha^{\vee} \in \Pi^{\vee} \}$. Let $\tI := \{1,2,\ldots,r\}$. We fix bijections $\tI \cong \Pi \cong \Pi^{\vee}$ such that $i \in \tI$ corresponds to $\alpha_i \in \Pi$, its coroot $\alpha_i^{\vee} \in \Pi ^{\vee}$, and a simple reflection $s_i \in W$ corresponding to $\alpha_i$. We also have a reflection $s_{\alpha} \in W$ corresponding to $\alpha \in \Delta_+$. Let $\{\varpi_i\}_{i \in \tI} \subset P_+$ be the set of fundamental weights (i.e. $\left< \al_i^{\vee}, \varpi_j \right> = \delta_{i,j}$) and we set $\rho := \sum_{i \in \tI} \varpi_i = \frac{1}{2}\sum_{\al \in \Delta^+} \al \in P_+$. 

Let $\Delta_{\af} := \Delta \times \Z \delta \cup \{m \delta\}_{m \neq 0}$ be the untwisted affine root system of $\Delta$ with its positive part $\Delta_+ \subset \Delta_{\af, +}$. We set $\alpha_0 := - \vartheta + \delta$, $\Pi_{\af} := \Pi \cup \{ \alpha_0 \}$, and $\tI_{\af} := \tI \cup \{ 0 \}$, where $\vartheta$ is the highest root of $\Delta_+$. We set $W _{\af} := W \ltimes Q^{\vee}$ and call it the affine Weyl group. It is a reflection group generated by $\{s_i \mid i \in \mathtt I_{\af} \}$, where $s_0$ is the reflection with respect to $\alpha_0$. Let $\ell : W_\af \rightarrow \Z_{\ge 0}$ be the length function and let $w_0 \in W$ be the longest element in $W \subset W_\af$. Together with the normalization $t_{- \vartheta^{\vee}} := s_{\vartheta} s_0$ (for the coroot $\vartheta^{\vee}$ of $\vartheta$), we introduce the translation element $t_{\beta} \in W _{\af}$ for each $\beta \in Q^{\vee}$.

For each $i \in \tI_{\af}$, we have a subgroup $\mathop{SL} ( 2, i ) \subset G (\!(z)\!)$ that is isomorphic to $\mathop{SL} ( 2, \C )$ corresponding to $\al_i \in \tI_{\af}$. We set $B_i := \mathop{SL} ( 2, i ) \cap \mathbf I$, that is a Borel subgroup of $\mathop{SL} ( 2, i )$. For each $i \in \tI$, we denote the minimal parabolic subgroup $\mathop{SL} ( 2, i ) B$ of $G$ corresponding to $i \in \tI$ by $P_i$. For each $w \in W$ or $w \in W_\af$, we find a representative $\dot{w}$ in $N_G ( H )$ or $N_{G(\!(z)\!)} ( H (\!(z)\!) )$, respectively.

Let $W_\af^-$ denote the set of minimal length representatives of $W_\af / W$ in $W_\af$. We set
$$Q^{\vee}_< := \{\beta \in Q^{\vee} \mid \left< \beta, \al_i \right> < 0, \forall i \in \tI \}.$$

Let $\le$ be the Bruhat order of $W_\af$. In other words, $w \le v$ holds if and only if a subexpression of a reduced decomposition of $v$ yields a reduced decomposition of $w$ (see \cite[Theorem 2.2.2]{BB05}). We define the generic (semi-infinite) Bruhat order $\le_\si$ as:
\begin{equation}
w \le_\si v \Leftrightarrow w t_{\beta} \le v t_{\beta} \hskip 5mm \text{for every } \beta \in Q^{\vee} \text{ such that } \left< \beta, \al_i \right> \ll 0 \text{ for } i \in \tI. \label{si-ord}
\end{equation}
By \cite[\S 1.5]{Lus80} (see also \cite[\S 2.2]{KNS17}), this defines a preorder on $W_{\af}$. Here we remark that $w \le v$ if and only if $w \ge_\si v$ for $w,v \in W$.

For each $\la \in P_+$, we denote a finite-dimensional simple $G$-module with a $B$-eigenvector with its $H$-weight $\la$ by $L ( \la )$. We understand that $L ( \la ) = \{0\}$ for $\la \not\in P_+$. Let $R ( G )$ be the (complexified) representation ring of $G$. We have an identification $R ( G ) = ( \C P )^W \subset \C P$ by taking characters. For a semi-simple $H$-module $V$, we set
$$\ch \, V := \sum_{\la \in P} e^\la \cdot \dim _{\C} \mathrm{Hom}_H ( \C_\la, V ).$$
If $V$ is a graded $H$-module in addition, then we set
$$\gch \, V := \sum_{\la \in P, n \in \Z} q^n e^\la \cdot \dim _{\C} \mathrm{Hom}_H ( \C_\la, V_n ).$$

Let $\sB := G / B$ and call it the flag manifold of $G$. It is equipped with the Bruhat decomposition
$$\sB = \bigsqcup_{w \in W} \bO_{\sB} ( w )$$
into $B$-orbits such that $\dim \bO_{\sB} ( w ) = \ell ( w_0 ) - \ell ( w )$ for each $w \in W \subset W_\af$. Namely, we have $\bO_{\sB} ( w ) = B \dot{w} \dot{w}_0 B / B$. We set $\sB ( w ) := \overline{\mathbb O_{\sB} ( w )} \subset \sB$. We also define $\bO^\op_{\sB} ( w ) := B^- \dot{w} \dot{w}_0 B / B$ and $\sB^\op ( w ) := \overline{\bO^\op_{\sB} ( w )}$.

For each $\la \in P$, we have a line bundle $\cO _{\sB} ( \la )$ on $\sB$ such that
$$\ch \, H ^0 ( \sB, \cO_{\sB} ( \la ) ) = \ch \, L ( \la ), \hskip 3mm \cO_{\sB} ( \la ) \otimes_{\cO_\sB} \cO _\sB ( - \mu ) \cong \cO_\sB ( \la - \mu ) \hskip 5mm \la, \mu \in P_+.$$
The line bundle $\cO_{\sB} ( \la )$ is defined as $G \times ^B \C_{w_0 \la}$, where $\C_{w_0\la}$ is the one-dimensional representation of $H$ corresponding to $w_0 \la$ (lifted to $B$).

We have a notion of $H$-equivariant $K$-group $K_H ( \sB )$ of $\sB$ with coefficients in $\C$ (see e.g. \cite[\S 4]{KK90}). Explicitly, we have
\begin{equation}
K_H ( \sB ) = \bigoplus_{w \in W} \C P \, [\cO _{\sB (w)}]  = \C P \otimes_{R ( G )} \bigoplus_{\la \in P} \C [\cO _{\sB} ( \la )].\label{KB}
\end{equation}
The map $\ch$ extends to a $\C P$-linear map
$$\chi : K _H ( \sB ) \rightarrow \C P,$$
that we call the $H$-equivariant Euler-Poincar\'e characteristic. The group $K_H ( \sB )$ is equipped with the product structure $\cdot$ induced by the tensor product of line bundles. For each $i \in \tI$, we have
\begin{equation}
[\cO _{\sB (s_i)}] = [\cO_{\sB}] - e^{\varpi_i} [\cO _{\sB} ( -\varpi_i )] \in K_H ( \sB ) \label{LSrem}
\end{equation}
coming from the $B$-equivariant short exact sequence
\begin{equation}
0 \rightarrow \C_{\varpi_i} \otimes \cO_{\sB} ( -\varpi_i ) \rightarrow \cO _{\sB} \rightarrow \cO_{\sB ( s_i )} \rightarrow 0.\label{eLSrem}
\end{equation}
Here the $B$-equivariant map $\C_{\varpi_i} \otimes \cO_{\sB} ( -\varpi_i ) \rightarrow \cO _{\sB}$ is unique up to scalar.

\subsection{Level zero nil-DAHA}\label{subsec:nDAHA}

\begin{defn}\label{def-DAHA}
The level zero nil-DAHA $\sH$ of type $G$ is a $\C$-algebra generated by $\{ e^{\la} \}_{\la \in P} \cup \{ D_i \} _{i \in \tI_\af}$ subject to the following relations:
\begin{enumerate}
\item $e^{\la + \mu} = e ^{\la} \cdot e ^{\mu}$ for $\la, \mu \in P$;
\item $D_i ^2 = D_i$ for each $i \in \tI_\af$;
\item For each distinct $i, j \in \tI_{\af}$, we set $m_{i,j} \in \Z_{> 0}$ as the minimum number such that $(s_is_j)^{m_{i,j}} = 1$. Then, we have
$$\overbrace{D_i D_j \cdots}^{m_{i,j}\text{-terms}} = \overbrace{D_j D_i \cdots}^{m_{i,j}\text{-terms}};$$
\item For each $\la \in P$ and $i \in \tI$, we have
$$D_i e^{\la} - e^{s_i \la} D_i = \frac{e^{\la} - e^{s_i \la}}{1 - e ^{\al_i}};$$
\item For each $\la \in P$, we have
$$D_0 e^{\la} - e^{s_\vartheta \la} D_0 = \frac{e^{\la} - e^{s_\vartheta \la}}{1 - e ^{- \vartheta}}.$$
\end{enumerate}
\end{defn}

Let $\sS := \C P \otimes \C W_\af$ be the smash product algebra, whose multiplication reads as:
$$( e^\la \otimes w ) (e^{\mu} \otimes v) = e^{\la + w \mu} \otimes w v \hskip 5mm \la,\mu \in P, w,v \in W_\af,$$
where $s_0$ acts on $P$ as $s_{\vartheta}$. Let $\C ( P )$ denote the fraction field of (the Laurent polynomial algebra) $\C P$. We have a scalar extension
$$\sA := \C ( P ) \otimes_{\C P} \sS = \C ( P ) \otimes \C W_\af.$$

\begin{thm}[\cite{LSS10} \S 2.2]\label{embtoA}
We have an embedding of algebras $\imath^* : \sH \hookrightarrow \sA$:
\begin{align*}
e^{\la} \mapsto e ^{\la} \otimes 1, \,\, & D_i \mapsto  \frac{1}{1 - e ^{\al_i}} \otimes 1 - \frac{e^{\al_i}}{1 - e ^{\al_i}} \otimes s_i, \hskip 5mm \la \in P, i \in \tI\\
& D_0 \mapsto \frac{1}{1 - e ^{- \vartheta}} \otimes 1 - \frac{e^{- \vartheta}}{1 - e ^{- \vartheta}} \otimes s_0. 
\end{align*}

\end{thm}

Since we have a natural action of $\sA$ on $\C ( P )$, we obtain an action of $\sH$ on $\C ( P )$, that we call the polynomial representation.

For $w \in W_\af$, we find a reduced expression $w = s_{i_1} \cdots s_{i_{\ell}}$ ($i_1,\ldots,i_{\ell} \in \tI_\af$) and set
$$D_w := D_{s_{i_1}} D_{s_{i_2}} \cdots D_{s_{i_{\ell}}} \in \sH.$$
By Definition \ref{def-DAHA} 3), the element $D_w$ is independent of the choice of a reduced expression. By Definition \ref{def-DAHA} 2), we have
$D_i D_{w_0} = D_{w_0}$ for each $i \in \tI$, and hence $D_{w_0}^2 = D_{w_0}$. We have an explicit form
\begin{equation}
D_{w_0} = \Bigl( 1 \otimes \sum_{w \in W} w \Bigr) \cdot \Bigl( \frac{e^{-\rho}}{\prod_{\al \in \Delta^+} ( e^{- \al / 2} - e^{\al/2})} \otimes 1 \Bigr) \in \sA \label{WCF}
\end{equation}
obtained from the (left $W$-invariance of the) Weyl character formula.

\subsection{Affine Grassmannians}\label{subsec:Gr}
We define the (thin) affine Grassmannian and (thin) affine flag variety by
$$\Gr := G (\!(z)\!) / G [\![z]\!] \hskip 3mm \text{and} \hskip 3mm \Fl := G (\!(z)\!) / \bI,$$
respectively. We have a natural fibration map $\pi : \Fl \rightarrow \Gr$ whose fiber is isomorphic to $\sB$. For each $w \in W_\af$, we set $\mathbb O^{\Fl}_w = \bI \dot{w} \bI / \bI$. For each $\beta \in Q^{\vee}$, we find $w \in t_{\beta}W$ and set $\mathbb O^\Gr_\beta := \pi ( \mathbb O^{\Fl}_w )$. The sets $\mathbb O^{\Fl}_w$ and $\mathbb O^\Gr_\beta$ do not depend on the choices involved.

\begin{thm}[Bruhat decomposition, \cite{Kum02} Corollary 6.1.20] We have $\bI$-orbit decompositions
$$\Gr = \bigsqcup_{\beta \in Q^{\vee}} \mathbb O^\Gr_\beta \hskip 3mm \text{and} \hskip 3mm \Fl = \bigsqcup_{w \in W_\af} \mathbb O _w^{\Fl}$$
such that we have $\mathbb O _v^{\Fl} \subset \overline{\mathbb O _w^{\Fl}}$ if and only if $v \le w$.
\end{thm}

Let us set $\Gr_{\beta} := \overline{\mathbb O^\Gr_\beta}$ and $\Fl_w := \overline{\mathbb O^{\Fl}_w}$ for $\beta \in Q^{\vee}$ and $w \in W_\af$. For $w \in W_\af^-$, we also set $\Gr_{w} := \Gr_{\beta}$ for a unique $\beta \in Q^{\vee}$ such that $w \in t_{\beta} W$.

We set
$$K_H ( \Gr ) := \bigoplus_{\beta \in Q^{\vee}} \C P \, [\cO_{\Gr_{\beta}}] \hskip 3mm \text{and} \hskip 3mm K_H ( \Fl ) := \bigoplus_{w \in W_\af} \C P \, [\cO_{\Fl_w}].$$

The following two results of Kostant-Kumar \cite{KK90} are about the ring $Y$ in \cite[(2.9)]{KK90}, and is dual to their equivariant $K$-group (see also Remark \ref{rem:prod}). We can see that $Y$ actually corresponds to our equivariant $K$-group through \cite[(3.39)]{KK90}, or by later reference like \cite[\S 3]{Kum17}. The original treatment comes with extra $\Gm$-actions that are responsible for the variable $q$. We can forget the $\Gm$-actions (or specialize to $q = 1$) since the construction involves only Laurent polynomials with respect to the additional variables, together with the denominators that admit the specializations:

\begin{thm}[Kostant-Kumar]\label{sH-reg}
The vector space $K_H ( \Fl )$ affords a regular representation of $\sH$ such that:
\begin{enumerate}
\item the subalgebra $\C P \subset \sH$ acts by the multiplication as $\C P$-modules;
\item we have $D_i [\cO_{\Fl_w}] = [\cO_{\Fl_{s_i w}}]$ $(s_i w > w)$ or $[\cO_{\Fl_{w}}]$ $(s_i w < w)$.\hfill $\Box$
\end{enumerate}
\end{thm}

Being a regular representation, we sometimes identify $K_H ( \Fl )$ with $\sH$ (through $e^{\la} [\cO_{\Fl_w}] \leftrightarrow e^{\la}  D_w$ for $\la \in P, w \in W_\af$) and consider the product of two elements in $\sH \cup K _H ( \Fl )$, that results in an element of $K_H ( \Fl ) \cong \sH \subset \sA$.

\begin{thm}[Kostant-Kumar]\label{sH-sph}
The pullback defines a map $\pi^* : K_H ( \Gr ) \hookrightarrow K _H ( \Fl )$ such that
$$\pi^* [\cO_{\Gr_{\beta}}] = [\cO_{\Fl_{t_{\beta}}}] D_{w_0} = D_{t_{\beta}} D_{w_0}\hskip 5mm \beta \in Q^{\vee}.$$
In particular, $\mathrm{Im} \, \pi^* = \sH D_{w_0}$ is a $\sH$-submodule, that can be regarded as a left ideal of $\sH$.\hfill $\Box$
\end{thm}

Let $\sC := \C ( P ) \otimes \C Q^{\vee} \subset \sA$ be the subalgebra generated by elements of the form $f \otimes t_{\beta}$ ($f \in \C ( P )$, $\beta \in Q^{\vee}$). By our convention on the $W_\af$-action on $P$, we have
$$( 1 \otimes t_{-\vartheta^{\vee}} ) ( f \otimes 1 ) \equiv ( 1 \otimes s_{\vartheta} s_0 ) ( f \otimes 1 ) = ( f \otimes 1 ) ( 1 \otimes s_{\vartheta} s_0 ) \equiv ( f \otimes 1 ) ( 1 \otimes  t_{-\vartheta^{\vee}} )$$
for each $f \in \C P$. From this, we deduce that $\sC$ is commutative. We have a projection map
$$\mathsf{pr} : \sA = \C ( P ) \otimes \C W_\af \longrightarrow \C ( P ) \otimes \C Q^{\vee} = \sC$$
defined as $\mathsf{pr} ( f \otimes t_{\beta} w ) = f \otimes t_{\beta}$ for each $f \in \C ( P ), w \in W$, and $\beta \in Q^{\vee}$.

\begin{thm}[Lam-Schilling-Shimozono]\label{LSS10}
The composition map $\mathsf{pr} \circ \imath^* \circ \pi^*$ defines an embedding
$$K_H ( \Gr ) \hookrightarrow K_H ( \Fl ) \rightarrow \sC \hskip 3mm ( \subset \sA )$$
whose image is equal to $K _H ( \Fl ) \cap \sC$. It descends to a $\C P$-module isomorphism
$$r^* : K_H ( \Gr ) \hookrightarrow K_H ( \Fl ) \cap \sC \hskip 3mm ( \subset \sA ).$$
This equips $K_H ( \Gr )$ with a subalgebra structure of a commutative algebra $\sC$.
\end{thm}

\begin{proof}
By \cite[Proposition 2]{LLMS17}, we deduce that the image of $D_{v}$ under the map $\mathsf{pr}$ is the same for each $v \in t_{\beta}W$. Therefore, the assertion follows from the description of \cite[\S 5.2]{LSS10}.
\end{proof}

Thanks to Theorem \ref{LSS10}, we obtain a commutative product structure of $K_H ( \Gr )$ inherited from $\sC$, that we denote by $\odot$. We call it the {\it Pontryagin product}. This is the same product as in \cite[\S 5.2]{LSS10}, and its relation with the Pontryagin product (\cite{Pon39}) in the topological $K$-group of the based loop space of the maximal compact subgroup of $G$ is read-out from \cite[\S 5.1]{LSS10} and the following:

\begin{rem}\label{rem:prod}
The Pontrjyagin product on $K_H ( \Gr )$ is the product structure coming from the homotopy equivalence (\cite[Proposition 8.6.6]{PS86}) of $\Gr$ and
$$\Omega K := \{f : S^1 \longrightarrow K \mid f \text{ is a $C^{0}$-map and } f ( 1 ) = e \},$$
where $K$ is the maximal compact subgroup of $G$ (viewed as a Lie group). Here $\Omega K$ is acted by ${T_\R} := ( H \cap K )$, and it acquires the group structure by the pointwise multiplication whose identity is the constant map to the identity, and whose inverse is the pointwise inverse. This induces a coproduct $\triangle$ on the ${T_\R}$-equivariant topological $K$-group $K^{T_\R} ( \Omega K )$. Note that the space of (the adjoint) $T_\R$-fixed points on $\Omega K$ is precisely the loops landing on $T_\R$ (homotopic to a discrete set isomorphic to $Q^{\vee}$), and hence this coproduct must coincide with the coproduct induced from the concatenation of loops by the injectivity of the restriction map to the $T_\R$-fixed points (see \cite[(3.19)]{KK90}).

Here $K^{T_\R} ( \Omega K )$ and $K_H ( \Gr )$ are naturally dual as discussed in \cite[(2.19) and (3.28)]{KK90}, in the sense that $K^{T_\R} ( \Omega K )$ is the equivariant $K$-group dealt there, and our $K_H ( \Gr )$ here is its dual dealt as (a subring of) the ring $Y$ in \cite[\S 2]{KK90}. In particular, one can transport $\triangle$ on $K^{T_\R} ( \Omega K )$ into the product on $K_H ( \Gr )$ coming from the multiplication of $\Omega K$, that is the Pontryagin product transported via the above interpretation. This is what carried out in \cite[Lemma 5.1]{LSS10} (where the coincidence of the two multiplications is implicitly explained as the homotopy commutativity of the product structure of $\Omega K$).

The product of $K_H ( \Gr )$ coming from Theorem \ref{sH-sph} differs from the Pontryagin product. In fact, we have ${T_\R} = ( B \cap K )$, and hence $\Omega K$ does not carry natural subvarieties corresponding to non-$G$-stable Schubert varieties in $\Gr$ (similar to the case of $\sB = K / {T_\R}$ without complex structure). This poses an effect on the Pontryagin product structure, and we only have a non-degenerate pairing
\begin{equation}
K^K ( \Omega K ) \times K_G ( \Gr ) \longrightarrow ( \C P )^W\label{KG-pair}
\end{equation}
based on larger group actions that intertwines the Pontryagin product structure and the product in Theorem \ref{sH-sph}. Here $K^{T_\R} ( \Omega K )$ is the scalar extension of $K^K ( \Omega K )$, and hence the scalar extension of (\ref{KG-pair}) describes the Pontryagin product. In this picture,
$$\sA \cong \bigoplus_{p \in ( \Fl )^{(\Gm \times H)}} \C ( P ) \otimes_{R ( H )} K_H ( p ) \supset  \bigoplus_{p \in ( \Gr )^{(\Gm \times H)}} \C ( P ) \otimes_{R ( H )} K_H ( p ) \cong \sC$$
contains $K_H ( \Fl )$ and $K_H ( \Gr )$ through the comparisons with the dualities with the thick versions (\cite[\S 3]{Kum17}) and their $H$-equivariant localizations to the $(\Gm \times H)$-fixed points (whose numerical computations trace back to \cite[\S 2]{KK90}, and presented here as Theorem \ref{embtoA}).

The inverse of $G(\!(z)\!)$ induces the following isomorphisms:
\begin{align*}
K_H ( \Gr ) & \equiv K ( \bI \backslash G(\!(z)\!) / G [\![z]\!] ) \cong  K ( G [\![z]\!] \backslash G(\!(z)\!) / \bI )\equiv K_G ( \Fl )\\
K_H ( \Fl ) & \equiv K ( \bI \backslash G(\!(z)\!) / \bI ) \cong  K ( \bI \backslash G(\!(z)\!) / \bI )\equiv K_H ( \Fl ).
\end{align*}
We denote these two maps by $\mathrm{inv}$. We have the following commutative diagram of $\C$-vector spaces:
\begin{equation}
\xymatrix{
K_G (\Gr) \ar[rd] \ar@{-->}[rrd] &  &  K_G ( \Fl ) \ar[r]& K_H ( \Fl ) \ar[ld]_{\mathrm{inv}}\\
& K_H ( \Gr) \ar[r]^{\pi^*} \ar[ru]^{\mathrm{inv}} \ar@{~>}[rrd] & K_H ( \Fl ) \ar[r]^{\hskip 5mm \imath^*} & \sA \ar[d]^{\mathsf{pr}}\\
& & & \sC
},\label{eqn:comm}
\end{equation}
where $K_G (\Gr) \to K_H (\Gr)$ and $K_G (\Fl) \to K_H (\Fl)$ are the scalar extension maps. In (\ref{eqn:comm}), the dashed arrow realizes $K_G (\Gr)$ as a commutative subalgebra of $K_H ( \Fl )$ that spans a subspace isomorphic to $K_H ( \Gr )$ by the (right) $\C P$-action, and the winding arrow respects the Pontryagin product. Theorem \ref{LSS10} asserts that the winding map is injective, and induces the ``wrong-way map" $K_H ( \Fl ) \to K_H ( \Gr )$ introduced by Peterson \cite[Lecture 8]{Pet97} for homology. We note that the above line of discussion yields the $K$-theoretic version of Peterson's famous identification
$$K_H ( \Gr ) \cong Z_{K_H ( \Fl )} ( \C P ) \hskip 5mm \text{(\cite[Theorem 5.3]{LSS10})}.$$
\end{rem}

Below, we might think of an element of $K_H ( \Gr )$ as an element of $K_H ( \Fl )$ through $\pi^*$, as an element of $\mathscr A$ through $\imath^* \circ \pi^*$, and as an element of $\mathscr C$ through $r^*$ interchangeably. The next two results are natural extensions of the results from \cite[\S 9]{LS10} (originally due to Peterson):

\begin{thm}\label{LSS-formula}
Let $w \in W_\af^-$ and let $\beta \in Q^{\vee}_<$. We have
$$[\cO_{\Gr_{w}}] \odot [\cO_{\Gr_{\beta}}] = [\cO_{\Gr_{w t_\beta}}].$$
\end{thm}

\begin{proof}
By our assumption on $\beta$, we have $\ell ( t_{\beta} ) = \ell ( w_0 ) + \ell ( w_0 t_{\beta} )$ (see \cite[(2.4.1)]{Mac03}). In particular, the element $[\cO_{\Gr_{\beta}}]$, viewed as an element of $\sA$ through $\imath^* \circ \pi^*$, is of the form $( \sum_{v \in W} v ) \xi$ for some $\xi \in \sA$ by (\ref{WCF}). Hence, it is invariant by the left action of $W$. Since the effect of the map $\mathsf{pr}$ is to twist by elements of $W$ from the right in a term by term fashion, we deduce the equality
$$[\cO_{\Gr_{w}}] [\cO_{\Gr_{\beta}}] = \mathsf{pr} ( [\cO_{\Gr_{w}}] ) [\cO_{\Gr_{\beta}}]$$
of multiplications in $\sA$ (multiplication in a non-commutative algebra). By examining the definition of $\mathsf{pr}$, we further deduce
\begin{equation}
\mathsf{pr} (  [\cO_{\Gr_{w}}] [\cO_{\Gr_{\beta}}] ) = \mathsf{pr} (  \mathsf{pr} ( [\cO_{\Gr_{w}}] ) [\cO_{\Gr_{\beta}}] ) = \mathsf{pr} ( [\cO_{\Gr_{w}}] \odot [\cO_{\Gr_{\beta}}] ).\label{lp-anti}
\end{equation}
Since $w \in W_\af^-$, we have $\ell ( w ) + \ell ( t_{\beta} ) = \ell ( w t_{\beta} )$ (see \cite[Lecture 8, page12]{Pet97}). Consequently, we have $D_{wt_{\beta}} = D_w D_{t_{\beta}}$. Therefore, (\ref{lp-anti}) and Theorem \ref{LSS10} implies that
$$[\cO_{\Gr_{w t_\beta}}] = [\cO_{\Gr_{w}}] [\cO_{\Gr_{\beta}}] = [\cO_{\Gr_{w}}] \odot [\cO_{\Gr_{\beta}}] \in K_H ( \Gr )$$
as required.
\end{proof}

Since $t_{\beta} \in W^-_\af$ for each $\beta \in Q^{\vee}_<$, Theorem \ref{LSS-formula} implies that the set
$$\{ [\cO_{\Gr_{\beta}}] \mid \beta \in Q^{\vee}_< \} \subset ( K_H ( \Gr ), \odot )$$
forms a multiplicative system. We denote by $K_H ( \Gr )_{\lo}$ its localization. The action of an element $[\cO_{\Gr_{\beta}}]$ on $K_H ( \Gr )$ in Theorem \ref{LSS-formula} is torsion-free, and hence we have an embedding $K_H ( \Gr ) \hookrightarrow K_H ( \Gr )_\lo$.

\begin{cor}\label{h-op}
Let $i \in \tI$. For $\beta \in Q^{\vee}_<$, we set
$$\bh _i := [\cO_{\Gr_{s_i t_{\beta}}}] \odot [\cO_{\Gr_{\beta}}]^{-1}.$$
Then, the element $\bh_i$ is independent of the choice of $\beta$.
\end{cor}

\begin{proof}
By Theorem \ref{LSS-formula}, we have
\begin{align*}
[\cO_{\Gr_{s_i t_{\gamma + \beta}}}] \odot [\cO_{\Gr_{\gamma + \beta}}]^{-1} & = [\cO_{\Gr_{s_i t_{\beta}}}] \odot [\cO_{\Gr_{\gamma}}] \odot  [\cO_{\Gr_{\gamma}}]^{-1} \odot [\cO_{\Gr_{\beta}}]^{-1}\\
& = [\cO_{\Gr_{s_i t_{\beta}}}] \odot [\cO_{\Gr_{\beta}}]^{-1}
\end{align*}
for $\gamma \in Q^{\vee}_<$. Hence, we conclude the assertion.
\end{proof}

For each $\gamma \in Q^{\vee}$, we can write $\gamma = \beta_1 - \beta_2$, where $\beta_1,\beta_2 \in Q^{\vee}_<$. Using this, we define an element
$$\mathtt t _{\gamma} := [\cO_{\Gr_{\beta_1}}] \odot [\cO_{\Gr_{\beta_2}}]^{-1}.$$

\begin{lem}\label{mt-op}
For each $\gamma \in Q^{\vee}$, the element $\mathtt t _{\gamma} \in K_H ( \Gr )_{\lo}$ is independent of the choices involved.
\end{lem}

\begin{proof}
Similar to the proof of Corollary \ref{h-op}. The details are left to the reader.
\end{proof}

\subsection{Semi-infinite flag manifolds}\label{subsec:sif}

The main reference of this subsection is \cite{Kat18d}. We define the semi-infinite flag manifold as the reduced scheme associated to:
$$\bQ_G^{\ra} := G (\!(z)\!) / ( H \cdot N (\!(z)\!) ).$$
This is a pure ind-scheme of ind-infinite type. Note that the group $Q^{\vee} \subset H (\!(z)\!) / H$ acts on $\bQ_G^{\ra}$ from the right. We have an embedding
\begin{equation}
\Upsilon : \bQ_G^\ra \hookrightarrow \prod_{i \in \mathtt I} \P ( L ( \varpi _i )^* \otimes \C (\!(z)\!) ),\label{formal-proj-emb}
\end{equation}
which is $\Gm \ltimes G(\!(z)\!)$-equivariant by enhancing the $\Gm$-action dilating $z$ prolonged trivially along the component $L ( \varpi_i )^*$ for each $i \in \tI$, and the $G$-action on $L ( \varpi_i )^*$ prolonged trivially along the component $\C (\!(z)\!)$ (\cite[Theorem 4.18]{Kat18d}). Note that the RHS of (\ref{formal-proj-emb}) is not a scheme by itself, but it acquires the structure of a scheme if we additionally impose the $z$-degree bound from the below on each factor. For $w \in W_\af$, we set $\mathbb O ( w ) := \bI \dot{w}\dot{w}_0 H N (\!(z)\!) / H N (\!(z)\!)$ and $\bQ _G ( w ) := \overline{\mathbb O ( w )}$. Note that we can take the closure either in $\bQ_G^\ra$ or the RHS of (\ref{formal-proj-emb}) since $\Upsilon$ restricts to a closed embedding of schemes
\begin{equation}
\Upsilon_m : \bQ_G ( w ) \hookrightarrow \prod_{i \in \tI} \P ( L ( \varpi _i )^* \otimes \C [\![z]\!] z^{-m} ) \hskip 4mm \left( \subset \prod_{i \in \mathtt I} \P ( L ( \varpi _i )^* \otimes \C (\!(z)\!) ) \right) \label{UN}
\end{equation}
for each $w \in W_\af$ under a suitable choice of $m \in \Z$. We refer $\bQ _G ( w )$ as a Schubert variety of $\bQ^\ra_G$. The set-theoretic part of the following result is deduced from the Iwasawa decomposition applied to \cite[\S 4]{FM99} (or \cite[\S 11]{LusICM}) as in \cite[Proof of Corollary 4.6]{KNS17}, and their closure relations are presented in \cite[\S 4.2 P2438]{KNS17} from the corresponding claims in \cite[\S 8]{FM99} and \cite[\S 5.1]{FFKM} (stated in the language of quasi-map spaces, see \S \ref{subsec:QM}):

\begin{thm}\label{si-Bruhat}
We have an $\bI$-orbit decomposition
$$\bQ_G^{\ra} = \bigsqcup_{w \in W_\af} \mathbb O ( w )$$
with the following properties:
\begin{enumerate}
\item each $\bO ( w )$ has infinite dimension and infinite codimension in $\bQ_G^{\ra}$;
\item each $\bO( w )$ contains a unique $( \Gm \times H )$-fixed point $p_w$;
\item the right action of $\gamma \in Q^{\vee}$ on $\bQ_G^{\ra}$ yields the translation $\mathbb O ( w ) \mapsto \mathbb O ( w t_{\gamma})$;
\item we have $\mathbb O ( w ) \subset \overline{\mathbb O ( v )}$ if and only if $w \le_{\si} v$.\hfill $\Box$
\end{enumerate}
\end{thm}

Theorem \ref{si-Bruhat} and (\ref{UN}) make the embedding (\ref{formal-proj-emb}) ind-closed.

We may write $\bQ_G$ instead of $\bQ_G ( e )$ for the sake of notational simplicity.
 
The indscheme $\bQ_G^{\ra}$ is equipped with a $G (\!(z)\!)$-equivariant line bundle $\cO _{\bQ_G^{\ra}} ( \la )$ for each $\la \in P$. This line bundle is realized as 
$$\bigotimes_{i \in \tI} \Upsilon^* \left( \mathcal O _{\P ( L ( \varpi _i )^* \otimes \C (\!(z)\!) )} ( 1 )\right) ^{\otimes m_i} \hskip 3mm \text{ when } \hskip 3mm \lambda = \sum_{i \in \tI} m_i \varpi_i.$$
In particular, the restriction  $\cO _{\bQ_G(w)} ( \la )$  of $\cO _{\bQ_G^{\ra}} ( \la )$ to each $\bQ_G ( w )$ defines a line bundle. We warn that the normalization of line bundles is twisted by $-w_0$ from that of \cite{KNS17}.

\begin{rem}[opposite Schubert varieties]\label{oppSchubert}
Here we discuss about opposite Schubert varieties of $\bQ_G^\ra$. Note that (\ref{formal-proj-emb}) is apparently non-stable with respect to the involution $z \mapsto z^{-1}$. In particular, the group $G [\![z^{-1}]\!]$ does not act on $\bQ_G^\ra$. Thus, our opposite Schubert subvariety of $\bQ_G^\ra$ should be the closure of an $\bI^-$-orbit, defined as an ind-scheme. However, such opposite Schubert subvarieties are continuously many, and hence they cannot be labelled by $W_\af$. Thus, we usually refer only the $\mathbf I^-$-orbit closures 
 $$\bQ_G^- ( w ) := \overline{\mathbf I^- p_w} \subset \bQ_G^\ra \hskip 5mm w \in W_\af$$
 as the opposite Schubert varieties of $\bQ_G^\ra$.
 
 If we set $\bI^{\flat}$ to be the Zariski closure of $\bI^-$ in $G[\![z^{-1}]\!]$, then we have another version of an opposite Schubert cell, namely an $\bI^{\flat}$-orbit in $\prod_i \P ( L ( \varpi _i )^* \otimes \C (\!( z^{-1} )\!) )$ that intersects with $\bQ_G^\ra$ in the ambient space
$$\prod_{i \in \tI} \P ( L ( \varpi _i )^* \otimes \C (\!( z^{-1} )\!) ) \subset \prod_{i \in \tI} \P ( L ( \varpi _i )^* \otimes \C [\![z, z^{-1} ]\!] ) \supset \bQ_G^\ra.$$
The equivalence class of points of
$$\prod_{i \in \tI} \P ( L ( \varpi _i )^* \otimes \C (\!( z^{-1} )\!) ) \cap \bQ_G^\ra \subset \prod_{i \in \tI} \P ( L ( \varpi _i )^* \otimes \C (\!( z^{-1} )\!) )$$
that are transferred to each other by the action of $\bI^\flat$ is in bijection with $W_\af$. The Zariski closure of each equivalence class defines another version of an opposite Schubert variety in
\begin{equation}
\prod_{i \in \tI} \P ( L ( \varpi _i )^* \otimes \C (\!( z^{-1} )\!) ).\label{prod:Linv}
\end{equation}
As every $G[\![z^{-1}]\!]$-orbit in (\ref{prod:Linv}) contains a point that does not belong to $\bQ^\ra_G$, an opposite Schubert variety in (\ref{prod:Linv}) cannot be a subscheme of $\bQ_G^\ra$. Nevertheless, each $\bQ_G^- ( w )$ defines a Zariski dense subset of an $\bI^{\flat}$-orbit consisting of points whose coordinates have only finitely many $z$-degree components (without any uniform bounds on their degrees).

In view of \cite{Kat18d}, the intersection of a Schubert variety and an opposite Schubert variety (labelled by $W_\af$) does not depend on a choice of these two versions of opposite Schubert varieties, and referred to as a Richardson variety of $\bQ_G^\ra$ (see \S \ref{subsec:GQL2} for more detailed account, including its dimension formula). Unlike the case of $\sB$, our open Richardson varieties of $\bQ_G^\ra$ are not necessarily smooth.
\end{rem}

We set $\g[z] := \g \otimes \C [z]$ and $\bI' := \bI \cap G[z]$, where the latter is an ind-group whose ind-structure is induced by $G[z]$.

\begin{thm}[Chari-Ion \cite{CI15}, see also \cite{Kat18} \S 1.2]\label{CI-proj}
For each $\lambda = \sum_{i \in \tI} m_i \varpi_i \in P_+$, we have a $(\Gm \times H)$-semisimple $G[z]$-module $\bW ( \la )$ with the following properties:
\begin{enumerate}
\item It is $G$-integrable, i.e. it is a direct sum of finite-dimensional $(\Gm \times G)$-modules by restriction;
\item It is generated by the action of $\bI'$ from a unique $($cyclic$)$ vector $($up to scalar$)$ with its $(\Gm \times H)$-weight $(w_0 \la)$;
\item It is projective in the category of $(\Gm \times H)$-semisimple $G$-integrable $G[z]$-modules whose graded characters belong to $\Z [\![q]\!] \{ \ch \, L ( \mu ) \mid \mu \le \la \}$;
\item We have
\begin{equation}
\gch \, \bW ( \la ) = ( \displaystyle \prod_{i \in \tI} \displaystyle \prod_{j=1}^{m_j}\frac{1}{1-q^j} ) P _{\lambda},\label{char-fact}
\end{equation}
where $P_\lambda \in ( \Z [q] P )^W$ is the symmetric Macdonald polynomial specialized to $t = 0$. In particular, we have:
\begin{equation}
P _\la \equiv \ch \, L ( \lambda ) \mod \Z [q] \{ \ch \, L ( \mu ) \mid \mu < \la \}. \label{P-modulo}
\end{equation}
\end{enumerate}
\end{thm}

\begin{thm}[\cite{KNS17} Corollary 4.31 and Proposition D.1]\label{BWB-KNS}
For each $\la \in P_+$, we have
\begin{equation}
\Gamma ( \bQ_G, \cO_{\bQ _G} ( \lambda ))^{\vee} = \bW (  - w_0 \lambda ).\label{H0-BWB}
\end{equation}
For each $\la \in P${, $w \in W$ and $\beta \in Q^{\vee}$,} we have
\begin{align*}
q^{\left< \la, \beta \right>} \gch \, \Gamma ( \bQ_G (wt_{\beta}), \cO_{\bQ _G (wt_{\beta})} ( \lambda )) & = \gch \, \Gamma ( \bQ_G (w), \cO_{\bQ _G (w)} ( \lambda )) \in (\C[\![q^{-1}]\!])P\\ \text{and} \hskip 5mm &  H^{>0} ( \bQ_G (w), \cO_{\bQ _G (w)} ( \lambda )) = \{0\}.
\end{align*}
\end{thm}

\begin{cor}\label{1-dimW}
The $(\Gm \times H)$-weight $( u \varpi_i + m \delta )$-part of
$$\Gamma ( \bQ_G, \cO_{\bQ _G} (  \varpi_i ))$$
is one-dimensional for each $i \in \tI$, $u \in W$, and $m \in \Z_{\le 0}$.
\end{cor}
\begin{proof}
By (\ref{P-modulo}), the monomial $q^m e^{u \varpi_i}$ appears in $P_{\lambda}$ only if $m=0$ and its coefficient is $1$. Since the $q$-series appearing as the dual of the RHS of (\ref{H0-BWB}) is
$$1 + q^{-1} + q^{-2} + \cdots$$
by (\ref{char-fact}), we conclude the assertion.
\end{proof}

For each $u \in W_\af$ and $i \in \tI$, we have a $(\Gm \times H)$-eigenvector $\phi_{u,i} \in ( L ( \varpi _i )^* \otimes \C [\![z]\!] z^{-m})^{\vee}$ dual to $p_u$ in the middle term of (\ref{UN}) for $m \gg 0$. It uniquely gives a $( \Gm \times H )$-eigensection in $\Gamma ( \bQ_G(t_{\beta}),  \cO_{\bQ_G(t_{\beta})} ( \varpi_i ) )$ of weight $u \varpi_i$ for each $t_{\beta} \ge_\si u$ since this $( \Gm \times H )$-weight space is one-dimensional by Corollary \ref{1-dimW} (and Theorem \ref{si-Bruhat} 3)).

\begin{lem}\label{phi-def}
For each $u \in W_\af$, the scheme $\bQ_G ( u )$ is set-theoretically defined as
$$\{ x \in \bQ_G^\ra \mid \phi_{v,i} ( x ) = 0, \hskip 3mm \forall (v,i) \in S ( u ) \}$$
through $(\ref{formal-proj-emb})$, where
$$S ( u ) := \{ (v,i) \in W_\af \times \tI \mid \phi_{v,i} ( x ) = 0, \hskip 3mm \forall x \in \bO ( u ) \}.$$
Similarly, the ind-closed subset $\bQ_G^- (u) \subset \bQ_G^\ra$ borrowed from Remark \ref{oppSchubert} is set-theoretically defined as
$$S^- ( u ) := \{ (v,i) \in W_\af \times \tI \mid \phi_{v,i} ( x ) = 0, \hskip 3mm \forall x \in \bI^- p_u \}.$$
\end{lem}

\begin{proof}
We first fix $\beta \in Q^{\vee}$ such that $t_{\beta} \ge_\si u$ and find sections $\{ \phi_{u,i} \}_{u, i}$. We have $\phi_{u,i} ( x' ) \neq 0$ for every point $x' \in \bO( u )$, and $\phi_{u,i} (\bullet) = 0$ for some $i \in \tI$ (set-theoretically) defines $\bQ _G (u) \setminus \bO ( u )$ in $\bQ_G(u)$. By the uniqueness of these sections, this set-theoretic defining property of $\phi_{u,i}$ holds by considering it as a section of a line bundle on any of the spaces in (\ref{UN}). Therefore, we conclude the first assertion. The second assertion follows by additionally taking Remark \ref{oppSchubert} (cf. \S \ref{subsec:GQL2}) into account.
\end{proof}

Let $\mathcal E$ be a $(\Gm \ltimes \bI)$-equivariant quasi-coherent sheaf on $\bQ_G$ that satisfies the condition $(\bigstar)$ consisting of the following two:
\begin{itemize}
\item[$(\bigstar)_1$] There exists $i_0 \in \Z$ (that may depend on $\mathcal E$) such that
$$H^{i} ( \bQ_G, \mathcal E \otimes_{\cO_{\bQ_G}} \cO_{\bQ_G} ( \la )) = \{0\}\hskip 5mm \text{for each} \hskip 2mm i > i_0, \hskip 2mm \text{and} \hskip 2mm \la \in P_{+};$$
\item[$(\bigstar)_2$] We have
$$\gch \, H^i ( \bQ_G, \mathcal E\otimes_{\cO_{\bQ_G}} \cO_{\bQ_G} ( \la )) \in ( \C(\!(q^{-1})\!) ) P\hskip 5mm \text{for each} \hskip 2mm i \in \Z, \hskip 2mm \text{and} \hskip 2mm \la \in P_{+}.$$
\end{itemize}

\begin{rem}
{\bf 1)} In the condition $(\bigstar)$, we set
$$\mathcal E \otimes_{\cO_{\bQ_G^\ra}} \cO_{\bQ_G^\ra} ( \la ) := \mathcal E \otimes_{\cO_{\bQ_G ( w )}} \cO_{\bQ_G ( w )} ( \la ) = \mathcal E \otimes_{\cO_{\bQ_G}} \cO_{\bQ_G} ( \la )$$
for each $w \ge_\si e$. This makes the inclusion $K_{H} ( \bQ_G ) \subset K_H ( \bQ^\ra_G )$ (described below) compatible with the tensor products of line bundles; {\bf 2)} In view of Theorem \ref{BWB-KNS}, the sheaf $\cO_{\bQ_G(w)}$ $(w \le_\si e)$ satisfies the condition $(\bigstar)$.
\end{rem}
For the above $\mathcal E$ and $\la \in P_{+}$, we set 
$$\chi_{q} ( \bQ_G, \mathcal E ( \la )) := \sum_{i \ge 0} (-1)^i \gch \, H^i ( \bQ_G, \mathcal E \otimes_{\cO_{\bQ_G}} \cO_{\bQ_G} ( \la ) ) \in (\C(\!(q^{-1})\!)) P.$$

\begin{lem}\label{KSESpre}
Suppose that we have a short exact sequence
$$0 \rightarrow \mathcal E_1 \rightarrow \mathcal E_2 \rightarrow \mathcal E_3 \rightarrow 0$$
of $(\Gm \ltimes \bI)$-equivariant quasi-coherent sheaves on $\bQ_G$ that satisfy $(\bigstar)$. We have
$$\chi_{q} ( \bQ_G, \mathcal E_2 ( \la )) = \chi_{q} ( \bQ_G, \mathcal E_1 ( \la )) + \chi_{q} ( \bQ_G, \mathcal E_3 ( \la )) \hskip 5mm \la \in P_+.$$
\end{lem}

\begin{proof}
This follows from the long exact sequence of cohomologies since each $\chi_{q} ( \bQ_G, \mathcal E_i ( \la ))$ ($i=1,2,3$) is well-defined.
\end{proof}

\begin{lem}\label{div}
For each $i \in \tI$, we have a short exact sequence
\begin{equation}
0 \rightarrow \C_{\varpi_i} \otimes \cO_{\bQ_G} ( -\varpi_i ) \rightarrow \cO _{\bQ_G} \rightarrow \cO_{\bQ_G ( s_i )} \rightarrow 0,\label{sLSrem}
\end{equation}
that is $(\Gm \ltimes \bI)$-equivariant.
\end{lem}

\begin{proof}
Let $\mathbb O$ be the dense open $G [\![z]\!]$-orbit in $\bQ_G$. We have $\mathbb O = \bigsqcup_{w \in W} \mathbb O ( w )$ by Theorem \ref{si-Bruhat} and the Bruhat decomposition. In particular, $\mathbb O$ yields an (uncountable dimensional) $G[\![z]\!]$-equivariant affine fibration over $\sB$ by setting $z = 0$. We lift (\ref{eLSrem}) by pulling back to obtain (\ref{sLSrem}) on $\mathbb O$. Twisting by $\cO_{\mathbb O} ( \varpi_i )$, we can interpret the map $\C_{\varpi_i} \otimes \cO_{\bQ_G} \rightarrow \cO _{\bQ_G} ( \varpi_i )$ as a unique (up to scalar) $(\Gm \ltimes \bI)$-equivariant section of $(\Gm \times H)$-weight $\varpi_i$ in $\Gamma ( \bQ_G, \cO_{\bQ_G} ( \varpi_i ) )$ as it restricts to a unique map (up to scalar) in (\ref{eLSrem}) by restriction. As a consequence, the short exact sequence (\ref{eLSrem}) yields the short exact sequence (\ref{sLSrem}) if the natural ($\Gm \ltimes \bI$-equivariant) inclusion
\begin{equation}
\C_{\varpi_i} \otimes \cO_{\bQ_G} ( -\varpi_i ) \hookrightarrow \ker ( \cO _{\bQ_G} \rightarrow \cO_{\bQ_G ( s_i )} )\label{nat-incl-lem}
\end{equation}
is an isomorphism. We set $\mathcal K := \ker ( \cO _{\bQ_G} \rightarrow \cO_{\bQ_G ( s_i )} )$.

Consider the map $\pi_i : P_i \times^{B} \bQ_G(s_i) \rightarrow \bQ_G$, that is surjective. We have a line bundle $\cO (-1)$ on $\mathop{SL} ( 2, i ) \times^{B_i} \bQ_G(s_i)$ obtained as the pullback of $\cO_{\P^1} (-D)$ through
$$P_i \times^{B} \bQ_G(s_i) \rightarrow P_i / B \cong \P^1,$$
where $D$ is the point $B/B \in \P^1$. The sheaf $\cO_{\P^1} (-D)$ (and hence $\cO ( -1 )$) is $B$-equivariant and admits a $P_i$-linearization after twisted by the $H$-character $\C _{\varpi_i}$ by (\ref{eLSrem}) for $\mathop{SL} ( 2, i )$. Let $\mathsf{infl}$ be the functor that inflates a $B$-equivariant sheaf on $\bQ_G(s_i)$ to a $P_i$-equivariant sheaf on $P_i \times^{B} \bQ_G(s_i)$.

By the Demazure character formula (\cite[Theorem A]{Kat18}, transported to this setting in \cite{KNS17}), we find that the definition of $\mathcal K$ is interpreted as
$$\R^{\bullet} ( \pi_i )_* \left( \cO ( - 1 ) \otimes \mathsf{infl} \, ( \cO_{\bQ_G ( s_i )} ) \right).$$
In particular, its twist by $\C _{\varpi_i}$ acquires the $P_i$-equivariant structure. In addition, this procedure commutes with the $G(\!(z)\!)$-equivariant line bundle twist of $\bQ_G^\ra$ (as presented in \cite[\S 6]{KNS17}, cf. Theorem \ref{H-si} and Theorem \ref{HQc}). Therefore, we conclude that
\begin{equation}
\C _{-\varpi_i} \otimes H^m ( \bQ_G, \mathcal K ( \la )) \hskip 5mm \la \in P\label{twisted}
\end{equation}
admits an action of $P_i$. Since
$$P_i \times^{B} \bQ_G(s_i) \cong ( \bI \cdot P_i ) \times^{\bI} \bQ_G(s_i),$$
we deduce that (\ref{twisted}) admits a $( P_i \bI' )$-action that prolongs the $P_i$-action. 

By \cite[Corollary 4.30]{KNS17}, we have an inclusion
\begin{equation}
\Gamma ( \bQ_G ( s_i ), \cO_{\bQ _G ( s_i )} ( \lambda ))^{\vee} \hookrightarrow \Gamma ( \bQ_G, \cO_{\bQ _G} ( \lambda ))^{\vee} \hskip 5mm \lambda \in P_+\label{lemmaKNS}
\end{equation}
as $\bI'$-modules, and the RHS has a cyclic vector with $( \Gm \times H )$-weight $- \lambda$. We set
$$K ( \lambda ) := \C_{\varpi_i} \otimes \Gamma ( \cO_{\bQ _G}, \mathcal K ( \lambda ) )^{\vee} \cong \C_{\varpi_i} \otimes \frac{\Gamma ( \bQ_G, \cO_{\bQ _G} ( \lambda ))^{\vee}}{\Gamma ( \bQ_G ( s_i ), \cO_{\bQ _G ( s_i )} ( \lambda ))^{\vee}}.$$
We have a surjection
$$\theta_\la : K ( \la ) \longrightarrow \!\!\!\!\! \rightarrow \Gamma ( \bQ_G, \cO_{\bQ _G} ( \la - \varpi_i ))^{\vee}$$
and (\ref{nat-incl-lem}) is an isomorphism if this is an isomorphism for every $\la \in P_+$. Since the action of $\mathop{SL} ( 2, j )$ ($j \neq i \in \tI$) commutes with $\C_{\varpi_i} \otimes \bullet$ and preserves $\bQ_G(s_i)$, we find a $P_j$-action on $K ( \la )$ from (\ref{lemmaKNS}), that coincide with the $P_i$-action along the intersection $B = P_i \cap P_j$. This particularly implies that $K ( \la )$ is invariant under the Demazure functor for each $i \in \tI$, and hence it acquires the $G$-action (\cite[(5.6)]{Jos85}). Therefore, $K ( \la )$ admits a $G$-action, that is upgraded into a $G$-integrable $G[z]$-module structure.

Being a quotient of $\C_{\varpi_i} \otimes \Gamma ( \bQ_G, \cO_{\bQ _G} ( \la ))^{\vee}$, the $G[z]$-module $K ( \lambda )$ is generated by a cyclic vector with $(\Gm \times H)$-weight $( \varpi_i - \la )$ by the action of $\bI'$, and if a $H$-weight $- \mu$ appears in $K ( \la )$, then we have $\mu \le \la - \varpi_i$. It follows that
\begin{equation}
\gch \, K ( \la ) \in \sum_{\mu \le \la - \varpi_i} \Z [\![q^{-1}]\!] \cdot \ch \, L ( \mu )^*.\label{weight-est-K}
\end{equation}

By Theorem \ref{CI-proj} 3), the $G[z]$-module $\Gamma ( \bQ_G, \cO_{\bQ _G} ( \la - \varpi_i ))^{\vee}$ is the largest one generated by a cyclic vector with $(\Gm \times H)$-weight $- ( \lambda - \varpi_i )$ that satisfies the same condition as $\gch \, K ( \la )$ in (\ref{weight-est-K}). Therefore, we conclude that $\theta_\la$ must be an isomorphism for every $\la \in P_+$.

This in turn yields that (\ref{nat-incl-lem}) is an isomorphism as desired.
\end{proof}

\subsection{Equivariant $K$-groups of semi-infinite flag manifolds}\label{subsec:eK}

The $(\Gm \ltimes \bI)$-equivariant $K$-group $K_{\Gm \ltimes \bI} ( \bQ_G^\ra )$ of $\bQ_G^\ra$ is defined and studied in \cite[\S 5--\S 6]{KNS17}. However, the definition given there does not allow one to forget the $\Gm$-action (nor equivalently make the $q=1$ specialization). Thus, here we present the construction and basic properties of the $H$-equivariant $K$-group $K_{H} ( \bQ_G^\ra )$ of $\bQ_G^\ra$ obtained as a variant of $K_{\Gm \ltimes \bI} ( \bQ_G^\ra )$. Note that the structure constants of the tensor products of line bundles on $K_{H} ( \bQ_G^\ra )$  (Theorem \ref{HQc}) are genuinely infinite with respect to the Schubert basis (see \S \ref{subsec: sl(2)}), and this action is essentially used in our proof of Theorem \ref{wJcomp}. Therefore, replacing $K_{H} ( \bQ_G^\ra )$ with its (pure) algebraic variant breaks down our proof of the main theorem.

For each $\beta \in Q^{\vee}_+$, we define a free $\C[q^{\pm 1}] P$-module $\widetilde{K} ( \beta )$ as:
$$\widetilde{K} ( \beta ) := \bigoplus_{t_{\beta} \ge_\si w \in W_{\af}} \C [q^{\pm 1}] P \cdot [\cO_{\bQ_G ( w )}].$$
We have $\widetilde{K} ( \beta + \gamma ) \subset \widetilde{K} ( \beta )$ for each $\beta,\gamma \in Q^{\vee}_+$. We set
$$\widetilde{K} ( \bQ_G ) := \varprojlim_{\beta \in Q^{\vee}_+} \left( \widetilde{K} ( 0 ) /\widetilde{K} ( \beta ) \right).$$
Each element of $\widetilde{K} ( \bQ_G )$ is presented as
$$\sum_{e \ge_\si w \in W_{\af}} a_{w} [\cO_{\bQ_G ( w )}] \hskip 5mm a_w \in \C [q^{\pm 1}] P,$$
where the sum is understood to be formal in the sense that infinitely many coefficients can be non-zero. For each $\la \in P$ and $\sum_w a_w [\cO_{\bQ_G ( w )}] \in \widetilde{K} ( \bQ_G )$, we define the formal sum
\begin{equation}
\widetilde{\Theta} ( \la ) \, ( \sum_w a_w [\cO_{\bQ_G ( w )}] ) = \sum_w a_w \gch \, \Gamma ( \bQ_G, \cO_{\bQ_G ( w )} ( \la ) ) \hskip 5mm \text{or} \hskip 5mm \infty,\label{eqn:defTheta}
\end{equation}
where the first case occurs if and only if
$$\lim_{n \to \infty}\sum_{w\ge_\si t_{n \vartheta^{\vee}}} a_w \gch \, \Gamma ( \bQ_G, \cO_{\bQ_G ( w )} ( \la ) ) \in (\C (\!(q^{-1})\!)) P,$$
and the second case occurs otherwise. Here the formula in the RHS of (\ref{eqn:defTheta}) is equal to the corresponding Euler-Poincar\'e characteristic by Theorem \ref{BWB-KNS}. We have $\widetilde{\Theta} ( \la ) (\bullet) = 0$ if $\la \not\in P_+$. We set
$$\widetilde{K}' ( \bQ_G ) := \{ \sum_{w} a_{w} [\cO_{\bQ_G ( w )}] \in \widetilde{K} ( \bQ_G ) \mid \widetilde{\Theta} ( \la ) ( \sum_{w} | a_{w} | [\cO_{\bQ_G ( w )}] ) \in (\R (\!(q^{-1})\!)) P \},$$
where $\la$ runs over $\la \in P_{++}$, and the absolute value $|a_w|$ of $a_w$ is taken coefficientwise. This is a $\C [q^{\pm 1}]P$-submodule of $\widetilde{K} (\bQ_G)$, and is a variant of the $(\Gm \ltimes \bI)$-equivariant $K$-group $K_{\Gm \ltimes \bI} ( \bQ_G )$ of $\bQ_G$ defined in \cite[\S 4]{KNS17}.

We define
$$\mathrm{Fun}_P := \{ (f,S) \mid  \emptyset \neq S \subset P \text{ satisfies } S + P_+ \subset S, \, f : S \rightarrow ( \C(\!(q^{-1})\!) ) P\},$$
where we understand that $f$ is defined on $S$, and
$$\mathrm{Fun}_P^{\mathrm{neg}} := \{ (f,S) \in \mathrm{Fun}_P \mid f ( \la ) = 0 \hskip 3mm \text{if} \hskip 3mm \left< \al_i^{\vee}, \la \right> \gg 0, \forall i \in \tI \}.$$
For any $(f,S_f),(g,S_g) \in \mathrm{Fun}_P$, we define
$$(f,S_f) \pm (g,S_g) := ( f \pm g, S_f \cap S_g ) \in \mathrm{Fun}_P.$$
Together with the multiplication by $\C [q^{\pm 1}] P$, this makes $\mathrm{Fun}_P$ and $\mathrm{Fun}_P^{\mathrm{neg}}$ into $\C [q^{\pm 1}] P$-modules. We may drop $S$ from $(f,S) \in \mathrm{Fun}_P$ or its image in $\mathrm{Fun}_P/\mathrm{Fun}_P^{\mathrm{neg}}$ whenever the meaning is clear from the context. This convention is justified by
$$(f, S_f) - (g, S_g) \in \mathrm{Fun}_P^{\mathrm{neg}} \hskip 5mm \text{if} \hskip 5mm f = g \hskip 3mm \text{on} \hskip 3mm (S_f \cap S_g).$$
In view of \cite[\S 5]{KNS17}, the map
\begin{align*}
\widetilde{\Theta} : & \, \widetilde{K}' ( \bQ_G ) \ni \sum_{w}a_{w} [\cO_{\bQ_G ( w )}] \\
& \mapsto \left[ \la \mapsto \sum_w \sum_{i \ge 0} (-1)^i a_w \gch \, H^i ( \bQ_G, \cO_{\bQ_G ( w )} ( \la ) ) \right] \\
& \equiv \left[ \la \mapsto \sum_w a_w \gch \, \Gamma ( \bQ_G, \cO_{\bQ_G ( w )} ( \la ) ) \right] \in \mathrm{Fun}_P
\end{align*}
induces an inclusion
$$\Theta : \widetilde{K}' ( \bQ_G ) \hookrightarrow \frac{\mathrm{Fun}_P}{\mathrm{Fun}_P^{\mathrm{neg}}}.$$
We say that that a $(\Gm \ltimes \bI)$-equivariant quasi-coherent sheaf $\mathcal E$ that satisfies $(\bigstar)$ defines a class
$$[\mathcal E] = \sum_{w} a_{w} (q) [\cO_{\bQ_G(w)}] \in \widetilde{K}' ( \bQ_G ), \hskip 5mm a_w (q)\in \C [q^{\pm 1}] P$$
if and only if the following assignment
$$\la \mapsto \left( \chi_{q} ( \bQ_G, \mathcal E ( \la )) - \sum_{w \in W_\af} a_{w} \chi_{q} ( \bQ_G, \cO_{\bQ_G ( w )} ( \la ) ) \right)$$
defines a function on $P_{++}$ that belongs to $\mathrm{Fun}_P^{\mathrm{neg}}$.

\begin{lem}\label{KSES}
Suppose that we have a short exact sequence
$$0 \rightarrow \mathcal E_1 \rightarrow \mathcal E_2 \rightarrow \mathcal E_3 \rightarrow 0$$
of $(\Gm \ltimes \bI)$-equivariant quasi-coherent sheaves on $\bQ_G$ that satisfy $(\bigstar)$. If two of the above three sheaves define classes in $\widetilde{K}' ( \bQ_G )$, then the remaining one also define a class in $\widetilde{K}' ( \bQ_G )$. In this case, we have
$$[\mathcal E_2] = [\mathcal E_1] + [\mathcal E_3].$$
\end{lem}

\begin{proof}
The condition $(\bigstar)$ guarantee the existence of the functions $f_i : \la \mapsto \chi_{q} ( \bQ_G, \mathcal E_i ( \la ))$ ($i=1,2,3$) defined on $P_+$ in $\mathrm{Fun}_P$. We have $f_2 = f_1 + f_3 \in \mathrm{Fun}_P$ by Lemma \ref{KSESpre}. Thus, so are their images in $\frac{\mathrm{Fun}_P}{\mathrm{Fun}_P^{\mathrm{neg}}}$. Since $\widetilde{K}' ( \bQ_G )$ is a(n abelian) group, we conclude the results.
\end{proof}

Since each of the coefficient of an element of $\widetilde{K}' ( \bQ_G )$ belongs to $\C [q^{\pm 1}]P$, and the multiplication by $\C [q^{\pm 1}]$ preserves $\widetilde{K}' ( \bQ_G ) \subset \widetilde{K} ( \bQ_G )$, the following $q=1$ specialization make sense:
$$K_{H} ( \bQ_G ) := \C \otimes_{\C [q^{\pm 1}]} \widetilde{K}' ( \bQ_G ).$$
By abuse of notation, we denote the class of a $(\Gm \ltimes \bI)$-equivariant quasi-coherent sheaf $\mathcal E$ in $K_{H} ( \bQ_G )$ obtained from its class $[\mathcal E] \in \widetilde{K}' ( \bQ_G )$ by the same letter.

\begin{rem}
Our definition of $\widetilde{K}' ( \bQ_G )$ here and \cite{KNS17} depends on the $\C[q^{\pm 1}] P$-linear independence of the asymptotic behavior of the functions
$$\la \mapsto \chi_{q} ( \bQ_G, \cO_{\bQ_G(w)} ( \la )) \hskip 5mm e \ge_\si w \in W_\af.$$
By \cite[Appendix A]{Kat18d}, the push-forward of $\cO_{\bQ_G(w)}$ to a parabolic version of $\bQ_G^\ra$ is the structure sheaf of a Schubert variety and have no higher direct images. In particular, the sheaf $\cO_{\bQ_G(w)}$ also satisfies a relative version of the condition $(\bigstar)$. This cohomological affinity makes $K_H ( \bQ_G )$ functorial with respect to the push-forwards to its parabolic analogues \cite{Kat19a,Kat21}.
\end{rem}

\begin{lem}\label{K-basis}
We have
$$K_{H} ( \bQ_G ) = \{ \sum_{e \ge_\si w \in W_{\af}} a_{w} [\cO_{\bQ_G ( w )}] \mid a_w \in \C P \},$$
where the sum in the definition is understood to be formal.
\end{lem}

\begin{proof}
All elements $a_{w}$ in the RHS are in $\C [q^{\pm 1}] P$ by definition. In view of Theorem \ref{BWB-KNS}, only finitely many terms in
$$\{a_w \gch \, \Gamma ( \bQ_G, \cO_{\bQ_G ( w )} ( \la ) )\}_{w \le _\si e}$$
carry a non-zero coefficients in $q^m$ for each $m \in \Z$ and $\la \in P_{++}$. Therefore, we have
$$\widetilde{\Theta} ( \la ) ( \sum_{w} a_{w} [\cO_{\bQ_G ( w )}] ) \in (\C [\![q^{-1}]\!]) P \subset (\C (\!(q^{-1})\!)) P $$
for each $\la \in P_{++}$. Thus, the assertion holds.
\end{proof}

As a natural extension of $K_{H} ( \bQ_G )$, we define
$$K_{H} ( \bQ_G^\ra ) := \{ \sum_{w \in W_{\af}} a_{w} [\cO_{\bQ_G ( w )}] \mid a_w \in \C P, \, \exists \beta_0 \in Q^{\vee} \text{ s.t. } a _{u t_{\beta}} = 0, \, \forall u \in W, \beta \not> \beta_0 \},$$
where the sum is understood to be formal. We have $K_{H} ( \bQ_G ) \subset K_{H} ( \bQ_G^\ra )$.

For each $\beta \in Q^{\vee}$, we have an endomorphism of $K_H ( \bQ_G^{\ra} )$ defined by
\begin{equation}
a [\cO_{\bQ_G ( w )}] \mapsto a [\cO_{\bQ_G ( wt_{\beta} )}] \hskip 5mm \forall a \in \C P, w \in W_\af\label{right-Q}
\end{equation}
induced by the right action of $Q^{\vee} \subset W_\af$ (see Theorem \ref{si-Bruhat} 3)). The translations of $K_H ( \bQ_G )$ with respect to this right $Q^{\vee}$-action equip $K_H ( \bQ_G^{\ra} )$ with a local (open) base of a linear topology (at the point $0 \in K_H ( \bQ_G^{\ra} )$). The subset $\{t_{\beta}\}_{\beta \in Q^{\vee}_+} \subset Q^{\vee}$ acts on $K_H ( \bQ_G )$. We denote by $\mathcal R$ the ring consisting of the formal $\C P$-linear combinations of $\{t_{\beta}\}_{\beta \in Q^{\vee}_+}$, equipped with an induced topology from $K_H ( \bQ_G )$. We have $\C Q^{\vee}_+ \subset \mathcal R$ spanned by the constant coefficient monomials of $Q^{\vee}_+$.

\begin{lem}\label{KQ-free}
$K_H ( \bQ_G )$ and $K_H ( \bQ_G^{\ra} )$ are free modules over $\mathcal R$ and $\C Q^{\vee} \otimes_{\C Q^{\vee}_+} \mathcal R$, respectively. Moreover, their ranks are $|W|$.
\end{lem}

\begin{proof}
We have an explicit basis $\{ [\cO_{\bQ_G ( w )}]\}_{w \in W}$ in both cases.
\end{proof}

Now we transplant the basic properties of $K_{\Gm \ltimes \bI} ( \bQ_G^{\ra} )$ and $K_{\Gm \ltimes \bI} ( \bQ_G )$ considered in \cite{KNS17} to our $K_{H} ( \bQ_G^{\ra} )$ and $K_{H} ( \bQ_G )$. The first one is immediate from the expression:

\begin{thm}[\cite{KNS17} \S 6 and \cite{Kat18} Theorem A]\label{H-si}
The vector space $K_H ( \bQ_G^{\ra} )$ affords a representation of $\sH$ with the following properties:
\begin{enumerate}
\item the subalgebra $\C P \subset \sH$ acts by the multiplication as $\C P$-modules;
\item we have
$$D_i ( [\cO_{\bQ_G (w)}] ) = \begin{cases} [\cO_{\bQ_G (s_i w)}] & (s_i w >_{\si} w) \\ [\cO_{\bQ_G (w)}] & (s_i w <_{\si} w)\end{cases}.$$
\end{enumerate}
\end{thm}

By Theorem \ref{H-si}, we deduce that the right $Q^{\vee}$-action yield $\sH$-module endomorphisms of $K_H ( \bQ_G^{\ra} )$. The following is implicit in \cite[Theorem 5.13]{KNS17}:

\begin{thm}\label{denseK}
For each $\mu \in P$, the line bundle twist by $\cO_{\bQ_G^\ra} ( \mu )$ preserves the space $\widetilde{K}' ( \bQ_G )$. In other words, for each $w \in W_\af$ such that $w \le _\si e$, there exists a collection $\{ a_w^{v} ( \mu )\}_{v \le_\si w}$ of elements in $\Z [q^{- 1}] P$ such that the function on $P_{++}$
$$\la \mapsto \left( \chi_{q} ( \bQ_G, \cO_{\bQ_G(w)} ( \la + \mu )) - \sum_{v} a_w^v (\mu) \cdot \chi_{q} ( \bQ_G, \cO_{\bQ_G(v)} ( \la )) \right)$$
belongs to $\mathrm{Fun}_P^{\mathrm{neg}}$. In particular, we have
$$[\cO_{\bQ_G(w)} ( \mu )] = \sum_{v \le_\si w} a_w^v (\mu) [\cO_{\bQ_G(v)}] \in \widetilde{K}' ( \bQ_G ).$$
\end{thm}

\begin{proof}
Since the tensor product operation ($=$ shift of functions on $P$) commutes with each other, we concentrate into the case $\mu = \pm \varpi_i$ for $i \in \tI$.
	
We consider the case $\mu = \varpi_i$. In the Pieri-Chevalley rule \cite[Theorem 5.13]{KNS17}, the coefficients $\{ a_w^v ( \mu )\}_{v \le_\si w}$ are given by counting the set of semi-infinite LS paths with fixed initial/final directions. By definition (\cite[Definition 2.6]{KNS17}), a semi-infinite LS path $(w_1,w_2,\ldots,w_\ell; a_1, a_2, \ldots, a_\ell)$ is a strictly decreasing collection $( w_1,w_2,\ldots,w_\ell )$ of elements of $W_\af$ with respect to $\le_\si$, together with a collection of strictly increasing sequence $(0 = a_1, a_2, \ldots, a_\ell = 1)$ in $( \Q \cap [0,1] )$ that satisfies the integrality condition in \cite[Definition 2.5]{KNS17}. In particular, we have an explicit bound on the denominators of $(a_1, a_2, \ldots, a_\ell)$ once we fix $\mu$. In addition, the interval of two elements $x t_{\beta}$ and $y t_{\gamma}$ $(x, y \in W, \beta, \gamma \in Q^{\vee} )$ of $W_\af$ with respect to $\le_\si$ consists of only finitely many elements as its member $w t_{\kappa}$ ($w \in W, \kappa \in Q^{\vee}$) must satisfy $\gamma \le \kappa \le \beta$. Thus, the set of semi-infinite LS paths with fixed initial/final direction is finite. As this model describes the character of a Demazure module (\cite[Theorem 2.8]{KNS17}), the relative degree count must be always non-positive. In particular, we find $a_w^{v} ( \varpi_i ) \in \Z [q^{-1}] P$ for each $w \in W_\af$. Thus, the assertion for $\mu = \varpi_i$ is precisely the contents of the proof of \cite[Theorem 5.13]{KNS17}.
	
	Moreover, the set of paths with the same initial/final direction is unique, and hence the transition matrix between $\{ [\cO_{\bQ_G ( w )} ( \varpi_i ) ] \}_{w \in W_\af}$ and $\{ [\cO_{\bQ_G ( w )}] \}_{w \in W_\af}$ is unitriangular (up to diagonal matrix consisting of characters in $P$) with respect to $\le_\si$. Therefore, we can invert this matrix to obtain $[\cO _{\bQ_G ( w )} ( - \varpi_i )] \in \widetilde{K}' ( \bQ_G )$ for $i \in \tI$. Thus, we completed the proof of the assertion for $\mu = \pm \varpi_i$ ($i \in \tI$) as required.
\end{proof}

The following is a variant of \cite[Theorem 6.5]{KNS17} (see also \cite{Kat18}):

\begin{thm}\label{HQc}
For each $\la \in P$, consider the $\C P$-linear extension of the assignment
$$[\cO _{\bQ_G ( w )}] \mapsto [\cO _{\bQ_G ( w )} ( \la )] \in K _H ( \bQ_G^{\mathrm{rat}} ) \hskip 3mm w \in W_\af.$$
It gives rise to a $\sH$-module automorphism, that we call $\Xi ( \la )$, with the following properties:
\begin{itemize}
\item It commutes with the right $Q^{\vee}$-action, i.e. we have
$$\Xi ( \la ) ( [\cO _{\bQ_G ( w )}] ) t_{\beta} = \Xi ( \la ) ( [\cO _{\bQ_G ( w )}] t_{\beta} ) \hskip 5mm \beta \in Q^{\vee}_+;$$
\item We have $\Xi ( \la )\circ \Xi ( \mu )=\Xi ( \la + \mu )$ for $\la, \mu \in P$.
\end{itemize}
\end{thm}

\begin{proof}
The first assertion follows from \cite[Proposition 6.3 and its proof]{KNS17}. The second assertion follows from Theorem \ref{denseK} and Theorem \ref{si-Bruhat} 3). The last assertion follows as the effect of $\Xi ( \la )$ is just to shift functions on $P$.
\end{proof}

\begin{lem}\label{divc}
For each $i \in \tI$, we have an equality
$$[\cO_{\bQ_G ( s_i )}] = [\cO_{\bQ_G ( e )}] - e^{\varpi_i} [\cO_{\bQ_G (e)} ( - \varpi_{i} )]$$
inside $\widetilde{K}' ( \bQ_G )$. In particular, it also holds for $K_H ( \bQ_G )$.
\end{lem}

\begin{proof}
Apply Lemma \ref{KSES} to Lemma \ref{div} by $[\cO_{\bQ_G ( s_i )}], [\cO_{\bQ_G ( e )}] \in \widetilde{K}' ( \bQ_G )$.
\end{proof}

\begin{rem}\label{Kra}
Lemma \ref{divc} implies $[\cO_{\bQ_G} ( - \varpi_i )] \in K_H ( \bQ_G ) \subset K_H ( \bQ_G^\ra )$. However, this does not imply $[\cO_{\bQ_G^\ra} ( - \varpi_i )] \in K_H ( \bQ_G^\ra )$ as $[\cO_{\bQ_G^\ra}] \not\in K_H ( \bQ_G^\ra )$. By the same reason, $\Xi ( \la )$ in Theorem \ref{HQc} ($=$ tensor product with $\cO_{\bQ_G^\ra} ( \la )$) is the multiplication by an element in the ring $K_H ( \bQ_G )$, but not in $K_H ( \bQ_G^\ra )$. In fact, the definition of $K_H ( \bQ_G^\ra )$ does not equip it with a ring structure. This stems from the fact that $\bQ_G^\ra$ is a(n infinite) union of Schubert varieties (whose dimensions are also infinite), but not a Schubert variety by itself.\\
The ring structure of $K_H ( \bQ_G^\ra )$ afforded in Corollary \ref{fLLMS} is different from the above as it employs $[\cO_{\bQ_G}]$ as the identity, that is {\it not} the class of the structure sheaf of $\bQ_G^\ra$.
\end{rem}

Motivated by Lemma \ref{divc}, we consider a $\C P$-module endomorphism $H_i$ ($i \in \tI$) of $K_H ( \bQ_G^{\mathrm{rat}} )$ as:
$$H_i : [\cO_{\bQ_G ( w )}] \mapsto [\cO_{\bQ_G ( w )}] - e ^{\varpi_i} [\cO_{\bQ_G ( w )} ( - \varpi_i )] \hskip 5mm w \in W_\af.$$

\begin{prop}\label{K-gen}
Let $K$ be the $\C P$-linear subspace of $K_H ( \bQ_G^{\mathrm{rat}} )$ generated by $[ \cO_{\bQ_G} ( \la ) ]$ $(\la \in P)$, together with the right $Q^{\vee}$-actions. The inclusion $K \subset K_H ( \bQ_G^{\mathrm{rat}} )$ is dense.
\end{prop}

\begin{proof}
Let $K_H ( \bQ_G )_+$ be the (formal) $\C P$-span of $\{ [\cO _{\bQ_G ( w t_{\beta} )}] \}_{w \in W, 0 \neq \beta \in Q^{\vee}_+}$ in $K_H ( \bQ_G^{\mathrm{rat}} )$. In view of (\ref{KB}), we have $K_H ( \bQ_G ) / K_H ( \bQ_G ( e ) )_+ \cong K_H ( \sB )$ as $\C P$-modules that sends $[\cO _{\bQ_G ( w )}]$ to $[\cO _{\sB ( w )}]$ ($w \in W$). By Theorem \ref{H-si} and Theorem \ref{sH-reg}, this intertwines the action of $D_i$ ($i \in \tI$). By the Pieri-Chevalley formula \cite[Theorem 5.13]{KNS17}, we see that
$$[ \cO_{\bQ ( w_0 )} ( \la ) ] \mod K_H ( \bQ_G )_+ = e^{w_0 \la} [\cO_{\sB ( w_0 )}] = [\cO_{\sB ( w_0 )} ( \la )] \hskip 5mm \la \in P.$$
By the Demazure character formulae (\cite[Theorem A]{Kat18} and \cite[8.1.13 Theorem]{Kum02}), we conclude that
$$[ \cO_{\bQ ( w )} ( \la ) ] \mod K_H ( \bQ_G )_+ = [\cO_{\sB ( w )} ( \la )] \hskip 5mm w \in W, \la \in P.$$
Therefore, the first two actions generate $K_H ( \bQ_G ) / K_H ( \bQ_G )_+ \cong K_H ( \sB )$ from $[\cO_{\bQ _G}]$. Now we use the right $Q^{\vee}$-action to conclude the result.
\end{proof}

\subsection{Graph and map spaces}\label{subsec:GQL}

We refer \cite{KM94,FP95,GL03} for the precise definitions of the notions appearing in this subsection.

We have $W$-equivariant isomorphisms $H^2 ( \sB, \Z ) \cong P$ and $H_2 ( \sB, \Z ) \cong Q ^{\vee}$. This identifies the (integral points of the) nef cone of $\sB$ with $P_+ \subset P$ and the effective cone of $\sB$ with $Q_+^{\vee}$. For each non-negative integer $n$ and $\beta \in Q^{\vee}_+$, we set $\sGB_{n, \beta}$ to be the space of stable maps of genus zero curves with $n$-marked points to $( \P^1 \times \sB )$ of bidegree $( 1, \beta )$, that is also called the graph space of $\sB$. A point of $\sGB_{n, \beta}$ is a(n arithmetic) genus zero curve $C$ with $n$-marked points $\{x_1,\ldots,x_n\}$, together with a map to $\P^1$ of degree one. Hence, we have a unique $\P^1$-component of $C$ that maps isomorphically onto $\P^1$. We call this component the main component of $C$ and denote it by $C_0$. By discarding the map to $\P^1$, we obtain the space of stable maps $\sB_{n,\beta}$ of genus zero curves with $n$-marked points to $\sB$ of degree $\beta$, together with the natural projection map $\mathtt f : \sGB_{n,\beta} \longrightarrow \sB_{n,\beta}$. The spaces $\sGB_{n, \beta}$ and $\sB_{n,\beta}$ are normal projective varieties by \cite[Theorem 2]{FP95} that have at worst quotient singularities arising from the automorphism of stable maps. The natural $( \Gm \times H)$-action on $( \P^1 \times \sB )$ induces a natural $( \Gm \times H)$-action on $\sGB_{n, \beta}$.

Let $\widetilde{\mathtt{ev}}_j : \sGB_{n, \beta} \to \P^1 \times \sB$ ($1 \le j \le n$) be the evaluation at the $j$-th marked point, and let $\mathtt{ev}_j : \sGB_{n, \beta} \to \mathscr B$ be its composition with the second projection. For a $( \Gm \times H )$-equivariant coherent sheaf $\mathcal F$ on a projective $( \Gm \times H )$-variety $\mathcal X$, let $\chi_q ( \mathcal X, \mathcal F ) \in \C  [q^{\pm 1}] P$ denote its $( \Gm \times H )$-equivariant Euler-Poincar\'e characteristic.

Consider the formal power series ring $\C [\![Q^{\vee}_+]\!]$ with its variables $Q_i = Q^{\al_i^{\vee}}$ ($i \in \tI$). We set $Q^{\beta} := \prod_{i \in \tI} Q_i ^{\left< \beta, \varpi_i \right>}$ for each $\beta \in Q^{\vee}$.

With this notation, we define the $q$-deformed $n$-point equivariant $K$-theoretic Gromov-Witten correlation function for $\xi_1,\ldots,\xi_n \in K_{H} ( \sB )$ as:
\begin{equation}
\langle \xi_1, \ldots, \xi_n \rangle_{\GW}^q := \sum_{\beta \in Q^{\vee}_+} Q^{\beta} \chi_q ( \sGB_{n,\beta}, \bigotimes_{j = 1}^n \mathtt{ev}_j^* \xi_j  )\in ( \C [q^{\pm 1}] P ) [\![Q^{\vee}_+]\!],\label{GWP}
\end{equation}
where we regard $\xi_1,\ldots,\xi_n$ as $\Gm$-equivariant objects with trivial $\Gm$-action.

The following result is well-known:
\begin{thm}\label{graph-to-map}
The $q = 1$ specialization of the $q$-deformed $n$-point equivariant $K$-theoretic Gromov-Witten correlation function is the usual $n$-point equivariant $K$-theoretic Gromov-Witten correlation function calculated by replacing $\sGB_{n,\beta}$ $(0 \neq \beta \in Q^{\vee}_+)$ in $(\ref{GWP})$ with $\sB_{n,\beta}$.
\end{thm}

\begin{proof}[Sketch of proof]
By adjunction, this comparison follows if we have $\mathbb R^{\bullet} \mathtt f_* \cO_{\sGB_{n,\beta}} \cong \cO_{\sB_{n,\beta}}$. The latter fact can be deduced from \cite[Theorem 7.1]{Kol86} since both spaces are normal with at worst rational singularities (cf. Remark \ref{Rem:Sing}), $\mathtt f$ is projective with connected fibers, and the general fiber of $\mathtt f$ is $\P^3 = \overline{PGL ( 2, \C )}$.
\end{proof}

Thanks to Theorem \ref{graph-to-map}, the $n$-point equivariant $K$-theoretic Gromov-Witten correlation function is given as:
$$\langle \xi_1, \ldots, \xi_n \rangle_{\GW} := \left. \langle \xi_1, \ldots, \xi_n \rangle_{\GW}^q \right|_{q=1} \hskip 5mm \xi_1,\ldots,\xi_n \in K_{H} ( \sB ).$$

\subsection{Equivariant quantum $K$-group of $\sB$}\label{subsec:eqK}

We define the $H$-equivariant (small) quantum $K$-group of $\sB$ as:
\begin{equation}
qK_H ( \sB ) := K_H ( \sB ) [\![Q^{\vee}_+]\!],\label{qKdef}
\end{equation}
that contains $\C [\![Q^{\vee}_+]\!][\cO_{\sB}] \cong \C [\![Q^{\vee}_+]\!]$. Thanks to (the $H$-equivariant versions of) \cite{Giv00,Lee04}, it is equipped with the commutative and associative product $\star$ (called the quantum multiplication) characterized as:
$$\langle \xi_1 \star \xi_2, \xi_3 \rangle_{\GW} = \langle \xi_1, \xi_2, \xi_3 \rangle_{\GW}\hskip 5mm \xi_1,\xi_2,\xi_3 \in qK_H ( \sB ),$$
where the forms on the both sides are understood to be linear with respect to $\C [\![Q^{\vee}_+]\!]$. The product $\star$ satisfies the following properties:
\begin{enumerate}
\item the element $[\cO_\sB] =  [\cO_\sB]Q^0 \in qK_H ( \sB )$ is the identity;
\item the map $Q^{\beta}\star$ $(\beta \in Q^{\vee}_+)$ is the multiplication of $Q^{\beta}$ in the RHS of (\ref{qKdef});
\item we have $\xi \star \eta \equiv \xi \cdot \eta \mod ( Q_i ; i \in \tI )$ for every $\xi,\eta \in K_H ( \sB ) \otimes 1$.
\end{enumerate}

From the above properties, we can localize $qK_H ( \sB )$ with respect to the multiplicative system $\{Q^{\beta}\}_{\beta \in Q^{\vee}_+}$ to obtain a ring $qK_H ( \sB )_\lo$.

We set
$$qK_{\Gm \times H} ( \sB ) := \left( K_H ( \sB ) (\!(q)\!) \right) [\![Q^{\vee}_+]\!].$$

We sometimes identify $K_H ( \sB )$ with the submodule of $qK_H ( \sB )$ or $qK_{\Gm \times H} ( \sB )$ that is constant with respect to $Q_i$ ($i \in \tI$) and $q$. We set $p_i := [\cO_\sB ( \varpi_i )]$ for $i \in \tI$, and we consider it as an endomorphism of $qK_{H} ( \sB )$ or $qK_{\Gm \times H} ( \sB )$ through the scalar extension of the product of $K_H ( \sB )$ (i.e. the classical product; we always understand that $p_i^{\pm 1}$ acts via the classical product). For each $i \in \tI$, let $q^{Q_i \partial_{Q_i}}$ denote the $( \C P ) (\!(q)\!)$-endomorphism of $qK_{\Gm \times H} ( \sB )$ such that
$$q^{Q_i \partial_{Q_i}} ( \xi \otimes Q^{\beta} ) = q ^{\left< \beta, \varpi_i \right>} \xi \otimes Q^{\beta} \hskip 5mm \xi \in K_H ( \sB ), \beta \in Q^{\vee}_+.$$

Following \cite[\S 2.4]{IMT15}, we consider the operator $T \in \mathrm{End}_{( \C P ) (\!(q)\!)} \, qK_{\Gm \times H} ( \sB )$. This operator is obtained from the operator $T(t)$ defined as
\begin{equation}
\chi ( T (t) ( \xi ) \cdot \eta ) := \sum_{n \ge 0} \frac{1}{n!} \langle \xi, \overbrace{t,\ldots,t}^{n\text{-terms}},\frac{\eta}{1 - q \mathbb L} \rangle_{\GW} \hskip 5mm \xi, \eta \in K_H ( \sB )\label{eqn:defT}
\end{equation}
by setting the additional variable $t \in K ( \sB )$ to be $0$, where $\mathbb L$ is the cotangent class of the corresponding stable curve at that marked point (see e.g. \cite{KM94,Giv00,Lee04}), and $q$ is understood to be a formal variable. The operator $T(t)$ is the (adjoint form of the) fundamental solution of the quantum differential equation
\begin{equation}
(1-q)\frac{\partial}{\partial t_{\xi}} T (t) ( \bullet ) = T (t)( \xi \star \bullet ),\label{qdE}
\end{equation}
where $\frac{\partial}{\partial t_{\xi}}$ is the derivation that commutes with $qK_{\Gm \times H} ( \sB )$ and sends the additional variable $t \in K ( \sB )$ into $\xi$. In order to deduce (\ref{qdE}), one uses the $K$-theoretic WDVV equation (\cite[P300]{Giv00}), that also yields the following:

\begin{thm}[Givental and Lee \cite{GL03, Giv00, Lee04}, see also \cite{IMT15} Proposition 2.3]\label{GLK}
For each $\xi, \zeta \in K_H ( \sB )$, we have
\begin{equation}
\chi ( T ( \xi ) \cdot \overline{T} ( \zeta )) = \langle \xi, \zeta \rangle_{\GW} \in ( \C P ) [\![Q^{\vee}_+]\!].\label{GLK-eq}
\end{equation}
Here we warn that the operators $T$ and the $q$-conjugate $\overline{T}$ of $T$ introduces variables $q$ and $Q^{\beta}$ $(\beta \in Q^{\vee}_+)$ in the calculation of the LHS, and they are understood to be scalars.
\end{thm}

\begin{rem}
The operator $T$ intertwines the classical and ``$q$-deformed" quantum inner products if we replace $\overline{T} ( \zeta )$ with $\overline{T ( \zeta )}$ in (\ref{GLK-eq}). However, the operator $T$ usually has a pole at $q=1$, and hence it does not induce an intertwiner between the classical and the quantum inner products naively.
\end{rem}

We have the shift operator (also obtained from an operator $A_i ( q, t )$ in \cite{IMT15} by setting $t = 0$) defined as
\begin{equation}
A_i ( q ) = T^{-1} \circ p_i^{-1} q^{Q_i \partial_{Q_i}} \circ T \in \mathrm{End} \, qK_{\Gm \times H} ( \sB ) \hskip 5mm i \in \tI.\label{sShift}
\end{equation}
Since each term admits an inverse, we find that $A_i ( q )^{-1}$ exists. As the operator $T$ commutes with $q$ and $Q_j$ ($j \in \tI$), we find
\begin{equation}
A_i ( q ) Q^\gamma = q^{\left< \gamma, \varpi_i \right>} Q^\gamma A_i ( q ) \hskip 5mm \gamma \in Q_+.\label{eqn:commrel}
\end{equation}
An element $J ( Q, q ) := T ( [\cO_{\sB}] ) \in qK_{\Gm \times H} ( \sB )$ is called the (equivariant $K$-theoretic) small quantum $J$-function, and is computed in \cite{GL03,BF14a} (cf. Theorem \ref{zas} and its explanation).

\begin{thm}[Reconstruction theorem \cite{IMT15} Proposition 2.12]\label{reconst}
For each
$$f ( q, x_1,\ldots, x_r, Q ) \in (\C P [q^{\pm 1}, x_1,\ldots,x_r])[\![Q^{\vee}_+]\!],$$
we have the following equivalence:
\begin{align*}
f ( q, p_1^{-1} q^{Q_1 \partial_{Q_1}}, & \ldots, p_r^{-1} q^{Q_r \partial_{Q_r}}, Q ) J ( Q, q ) = 0 \in qK_{\Gm \times H} ( \sB )\\
\Leftrightarrow & f ( q, A_1 (q), \ldots, A_r (q), Q ) [\cO_{\sB}]= 0 \in qK_{\Gm \times H} ( \sB ).
\end{align*}
\end{thm}

\begin{rem}
The original form of Theorem \ref{reconst} is about big quantum $K$-group. We have made the specialization $t = 0$ to deduce our form. It should be noted that {\bf 1)} this equivariant setting is automatic from the construction, and {\bf 2)} we state Theorem \ref{reconst} for unmodified quantum $J$-function (specialized to $t=0$) instead of the modified one employed in \cite[Proposition 2.12]{IMT15}.
\end{rem}

\begin{thm}[Anderson-Chen-Tseng \cite{ACT18} Lemma 6, see also \cite{ACT17}]\label{sline}
For each $i \in \tI$, we have $A_i ( q ) ( [\cO_{\sB}] ) = [\cO_\sB ( - \varpi_i )] \in qK_{H} ( \sB ) \subset qK_{\Gm \times H} ( \sB )$.
\end{thm}

We give an alternative proof of Theorem \ref{sline} in \S \ref{subsec:conseq}.

For each $i \in \tI$, we set $a_i := A (1)$ (thanks to \cite[Remark 2.14]{IMT15}).

\begin{thm}[\cite{IMT15} Corollary 2.9]\label{smult}
For $i \in \tI$, the operator $a_i$ defines the quantum multiplication by $a_i ( [\cO_{\sB}] )$ in $qK_H ( \sB )$.
\end{thm}

\begin{proof}
By \cite[Corollary 2.9]{IMT15}, the set $\{a_i\}_{i \in \tI}$ defines mutually commutative endomorphisms of $qK_H ( \sB )$ that commutes with the $\star$-multiplication. Since $\mathrm{End} _R R \cong R$ for every ring $R$, we conclude the assertion.
\end{proof}

\begin{thm}[\cite{IMT15} Proposition 2.10 and its proof]\label{i-mod}
For $i \in \tI$, we have
\begin{align*}
A_i ( q ) & = p_i^{-1} + \sum_{0 \neq \beta \in Q^{\vee}_+} \sum_{k=0}^{\left< \beta, \varpi_i \right>-1} a_{i,\beta,k} (1-q)^{k} Q^{\beta} \hskip 5mm a_{i,\beta,k} \in \mathrm{End} _{\C P} K_H ( \sB )\\
A_i ( q )^{-1} & = p_i + \sum_{0 \neq \beta \in Q^{\vee}_+} b_{i,\beta} (q) Q^{\beta} \hskip 5mm b_{i,\beta} \in \C [q^{-1}] \otimes \mathrm{End} _{\C P} K_H ( \sB )
\end{align*}
as operators acting on $\C[q^{\pm 1}] \otimes K_{H} (\sB)$, where the sum is understood to be formal. In particular, we have
$$a_i ( \xi ) \equiv [\cO_{\sB} ( - \varpi_i)] \cdot \xi \equiv [\cO_{\sB} ( - \varpi_i)] \star \xi \mod (Q_i ; i \in \tI)$$
for each $\xi \in K_H ( \sB )$.
\end{thm}

\begin{proof}
The first equalities for $A_i ( q )$ and $a_i$ are the $t=0$ specializations of \cite[Proposition 2.10]{IMT15}.
By taking the formal inverse of $A_i ( q )$ inductively on the coefficients of $Q^{\beta}$ ($\beta \in Q^{\vee}_+$) and using the fact that $A_i (q)$ is invertible modulo $\{Q_i\}_{i}$, we find an expression of $A_i (q)^{-1}$ with
$$b_{i,\beta} ( q ) \in \C [q^{\pm 1}] \otimes \mathrm{End} _{\C P} K_H ( \sB ).$$
We have $b_{i,\beta} ( q ) = 0$ for $0 \neq \beta \in Q^{\vee}_+$ such that $\left< \beta, \varpi_i \right> = 0$ by the shape of $A_i (q)$. As in \cite[proof of Proposition 2.10]{IMT15}, we use the fact that $T$ and $T^{-1}$ ($=S$ in their convention) are regular at $q = 0,\infty$ to deduce that the singularities of $\{b_{i,\beta} ( q )\}_{\beta}$ as functions on $q$ arises from the effect of
$$q^{-Q_i \partial_{Q_i}} p_i T,$$
that introduces poles at $q=0$, and its specialization $q=\infty$ is zero in the coefficient of $Q^{\beta}$ with $\left< \beta, \varpi_i \right> > 0$. Thus, we find that
$$b_{i,\beta} ( q ) \in \C [q^{-1}] \otimes \mathrm{End} _{\C P} K_H ( \sB ),$$
that yields the equality on $A_i (q)^{-1}$. By the third property of the product $\star$, we conclude the last equality.
\end{proof}

\section{Relation with affine Grassmannians}\label{sec:Gr}

We work in the same settings as in the previous section. Here we establish an isomorphism (Theorem \ref{main}) between the localized equivariant $K$-group of $\Gr$ (\S \ref{subsec:Gr}) and the equivariant $K$-group of $\bQ_G^\ra$ (\S \ref{subsec:eK}), and exhibits an example (\S \ref{subsec: sl(2)}). The main tool here is the action of level zero nil-DAHA (\S \ref{subsec:nDAHA}).

\subsection{Transporting the $\sH$-action to $\sC$}\label{subsec:trans}

\begin{prop}\label{H0}
The $\sH$-action of $K _H ( \Gr )$ induces a $\sH$-action on $\sC$ as:
\begin{align*}
D_0 ( f \otimes t_{\beta} ) & =  \frac{f}{1 - e^{- \vartheta}} \otimes t_{\beta} - \frac{e^{-\vartheta}s_{\vartheta} ( f )}{1 - e ^{- \vartheta}} \otimes t_{s_\vartheta ( \beta - \vartheta^{\vee} )}  & f \in \C ( P )\\
D_i ( f \otimes t_{\beta} ) & = \frac{f}{1 - e^{\al_i}} \otimes t_{\beta} - \frac{e^{\al_i}s_i ( f )}{1 - e ^{\al_i}} \otimes t_{s_i \beta} & i \in \tI, \beta \in Q^{\vee}\\
e^{\mu} ( f \otimes t_{\beta} ) & = e^{\mu} f \otimes t_{\beta} & \mu \in P.
\end{align*}
\end{prop}

\begin{proof}
For $i \in \tI_{\af}$, the action of $D_i$ on $\sA$ is the left multiplication of $\frac{1}{1 - e^{\al_i}} \otimes 1 - \frac{e^{\al_i}}{1 - e^{\al_i}} \otimes s_i$ (if we understand $\al_0 = - \vartheta$). Applying to an element $f \otimes t_{\beta} u \in \sA$ ($f \in \C ( P ), \beta \in Q^{\vee}, u \in W$), we deduce
\begin{align*}
D_i ( f \otimes t_{\beta} u ) & = \frac{f}{1 - e^{\al_i}} \otimes t_{\beta} u  - \frac{e^{\al_i} s_i ( f )}{1 - e^{\al_i}} \otimes t_{s_i\beta} s_i u\hskip 5mm i \neq 0\\
D_0 ( f \otimes t_{\beta} u ) & = \frac{f}{1 - e^{- \vartheta} } \otimes t_{\beta} u - \frac{e^{- \vartheta} s_\vartheta ( f )}{1 - e^{- \vartheta}} \otimes s_0 t_{\beta} u \\
& = \frac{f}{1 - e^{- \vartheta}} \otimes t_{\beta} u - \frac{e^{- \vartheta} s_\vartheta ( f )}{1 - e^{- \vartheta}} \otimes s_\vartheta t_{- \vartheta^{\vee}} t_{\beta} u\\
& = \frac{f}{1 - e^{- \vartheta}} \otimes t_{\beta} u - \frac{e^{- \vartheta} s_\vartheta ( f )}{1 - e^{- \vartheta}} \otimes t_{s_{\vartheta} ( \beta - \vartheta^{\vee} )} s_\vartheta u.
\end{align*}
Hence, applying $\mathsf{pr}$ yields the desired formula on $D_i$ for $i \in \tI_\af$. Together with the left multiplication of $e^{\la} \otimes 1$, these formulae transplant the $\sH$-action from $K_H ( \Gr )$ to $K_H ( \Gr ) \cap \sC$.

Since $K_H ( \Gr ) = \sC \cap K_H ( \Fl )$, we have $\C ( P ) \otimes _{\C P} K _H ( \Gr ) \subset \sC$. By comparing the leading terms of $\{ [\Gr_\beta] \}_{\beta \in Q^{\vee}} \subset \sC$ with respect to the Bruhat order (on the second component of $\sC \subset \sA = \C ( P ) \otimes \C W_\af$), we derive $\sC \subset \C ( P ) \otimes _{\C P} K _H ( \Gr )$. It follows that $\sC = \C ( P ) \otimes _{\C P} K _H ( \Gr )$. Hence, the above formulae define the $\sH$-action on $\sC$ as the scalar extension of that on $K _H ( \Gr ) \subset \sC$ as required (one can also directly check the relations of $\sH$).
\end{proof}

Below, we may write the action of $D_i$ on $\sC$ by $D_i ^{\#}$ to distinguish from the action on $K _H ( \Fl )$ or $\sA$.

\begin{cor}\label{invt}
Let $i \in \tI$. Let $\xi \in \sC$ be an element such that $D_i ^{\#} ( \xi ) = \xi$. Then, $\xi$ is a $\C$-linear combination of
$$f \otimes t_{\beta} + s_i ( f ) \otimes t_{s_i \beta} \hskip 5mm f \in \C ( P ), \beta \in Q^{\vee}.$$
\end{cor}

\begin{proof}
By Proposition \ref{H0}, the action of $D_i ^{\#}$ preserves $\C ( P ) \otimes t_{\beta} + \C ( P ) \otimes t_{s_i \beta}$ for each $i \in \tI$ and $\beta \in Q^{\vee}$. Hence, it suffices to find a condition that $a \otimes t_{\beta} + b \otimes t_{s_i \beta}$ ($a, b \in \C ( P )$) is stable by the action of $D_i ^{\#}$. It reads as:
\begin{align*}
D_i ^{\#} ( a \otimes t_{\beta} + b \otimes t_{s_i \beta} ) & = \frac{a - e^{\al_i} s_i ( b )}{1 - e^{\al_i}} \otimes t_{\beta} + \frac{b - e^{\al_i}s_i ( a )}{1 - e ^{\al_i}} \otimes t_{s_i \beta}\\
& = a \otimes t_{\beta} + b \otimes t_{s_i \beta}.
\end{align*}
This is equivalent to $b = s_i ( a )$ (or $s_i ( a + b ) = a + b$ in the case of $s_i \beta = \beta$) as required.
\end{proof}

\begin{cor}\label{inv-fac}
Let $i \in \tI$. Let $\xi, \xi' \in \sC$ be elements such that $D_i ^{\#} ( \xi ) = \xi$. We have
$$D_i ^{\#} ( \xi \xi' ) = \xi D_i ^{\#} ( \xi' ).$$
\end{cor}

\begin{proof}
By Corollary \ref{invt}, it suffices to prove
$$D_i ^{\#} ( ( f \otimes t_{\beta} + s_i ( f ) \otimes t_{s_i \beta} ) g \otimes t _{\gamma} ) = ( f \otimes t_{\beta} + s_i ( f ) \otimes t_{s_i \beta} ) D_i^{\#} (  g \otimes t _{\gamma} )$$
for every $f,g \in \C ( P )$ and $\beta,\gamma \in Q^{\vee}$. We derive as:
\begin{align*}
D_i ^{\#} ( ( f \otimes t_{\beta} + s_i ( f ) \otimes t_{s_i \beta} ) &  g \otimes t _{\gamma} ) = D_i ^{\#} ( f g \otimes t_{\beta + \gamma} + s_i ( f ) g \otimes t_{s_i \beta + \gamma} )\\
& = \frac{fg}{1 - e^{\al_i}} \otimes t_{\beta + \gamma} - \frac{e^{\al_i}s_i ( f g )}{1 - e ^{\al_i}} \otimes t_{s_i \beta + s_i \gamma}\\
& + \frac{s_i ( f ) g}{1 - e^{\al_i}} \otimes t_{s_i \beta + \gamma} - \frac{e^{\al_i}f s_i ( g )}{1 - e ^{\al_i}} \otimes t_{\beta + s_i \gamma}\\
& = ( f \otimes t_{\beta} + s_i ( f ) \otimes t_{s_i \beta} ) ( \frac{g}{1 - e^{\al_i}} \otimes t_{\gamma} - \frac{e^{\al_i}s_i ( g )}{1 - e ^{\al_i}} \otimes t_{s_i \gamma} )\\
& = ( f \otimes t_{\beta} + s_i ( f ) \otimes t_{s_i \beta} ) D_i ^{\#} (  g \otimes t _{\gamma} ).
\end{align*}
This completes the proof.
\end{proof}

\begin{lem}\label{inv-fac0}
Let $\xi, \xi' \in \sC$ be elements such that $D_i ^{\#} ( \xi ) = \xi$ for every $i \in \tI$. We have
$$D_0 ^{\#} ( \xi \xi' ) = \xi D_0 ^{\#} ( \xi' ).$$
\end{lem}

\begin{proof}
By Corollary \ref{invt}, we deduce $w \xi w^{-1} = \xi \in \sA$ for every $w \in W$. In particular, we have $s_\vartheta \xi s _\vartheta = \xi$.

Therefore, it suffices to prove
$$D_0 ^{\#} ( ( f \otimes t_{\beta} + s_\vartheta ( f ) \otimes t_{s_\vartheta \beta} ) g \otimes t _{\gamma} ) = ( f \otimes t_{\beta} + s_\vartheta ( f ) \otimes t_{s _\vartheta \beta} ) D_0^{\#} (  g \otimes t _{\gamma} )$$
for every $f,g \in \C ( P )$ and $\beta,\gamma \in Q^{\vee}$. We derive as:
\begin{align*}
D_0 ^{\#} ( ( f \otimes t_{\beta} + s_\vartheta ( f ) \otimes t_{s_\vartheta \beta} ) &  g \otimes t _{\gamma} ) = D_0 ^{\#} ( f g \otimes t_{\beta + \gamma} + s_\vartheta ( f ) g \otimes t_{s_\vartheta \beta + \gamma} )\\
& = \frac{fg}{1 - e^{- \vartheta}} \otimes t_{\beta + \gamma} - \frac{e^{- \vartheta}s_\vartheta ( f g )}{1 - e ^{- \vartheta}} \otimes t_{s_\vartheta \beta + s_\vartheta ( \gamma - \vartheta^{\vee} )}\\
& + \frac{s_\vartheta ( f ) g}{1 - e^{- \vartheta}} \otimes t_{s_\vartheta \beta + \gamma} - \frac{e^{- \vartheta}f s_\vartheta ( g )}{1 - e ^{- \vartheta}} \otimes t_{\beta + s_\vartheta ( \gamma - \vartheta^{\vee})}\\& = ( f \otimes t_{\beta} + s_\vartheta ( f ) \otimes t_{s_\vartheta \beta} ) D_0 ^{\#} (  g \otimes t _{\gamma} ).
\end{align*}
This completes the proof.
\end{proof}

\begin{thm}\label{comm}
For each $\beta \in Q^{\vee}_{<}$ and $i \in \tI_\af$, we have
$$D_i ( [\cO_{\Gr_\beta}] \odot \xi ) = [\cO_{\Gr_\beta}] \odot D_i ( \xi ) \hskip 5mm \xi \in K_H ( \Gr ).$$
\end{thm}

\begin{proof}
By construction, we have
$$\pi^* ( [\cO_{\Gr_\beta}] ) = D_{t_{\beta}} D_{w_0} = D_i D_{t_{\beta}} D_{w_0} \hskip 5mm i \in \tI,$$
where the second identity follows from $\ell ( t_{\beta} ) = \ell ( s_i t_{\beta} ) + 1$. By Proposition \ref{H0}, we deduce that $r^* ( [\cO_{\Gr_\beta}] )$ satisfies the $D_i ^\#$-invariance for each $i \in \tI$. Therefore, Corollaries \ref{inv-fac} and \ref{inv-fac0} imply the result.
\end{proof}

\begin{cor}\label{indH}
For $\beta \in Q^{\vee}_<$ and $i \in \tI_\af$, we have $D _i = \mathtt t_{- \beta} \circ D_i \circ \mathtt t_{\beta}$. In particular, we have a natural extension of the $\sH$-action from $K_H ( \Gr )$ to $K_H ( \Gr )_\lo$.
\end{cor}

\begin{proof}
The first assertion is a direct consequence of Theorem \ref{comm}. As we have $K_H ( \Gr )_\lo = K_H ( \Gr ) [\mathtt t_{\beta} \mid \beta \in Q^{\vee}_{<}]$, the latter assertion follows.
\end{proof}

\subsection{Inclusion as $\sH$-modules}

\begin{lem}\label{s-reg}
Let $i \in \tI_\af$. For each $w \in W_\af^-$, we have
$$D_i ( [\cO _{\Gr_{w}}] ) = \begin{cases} [\cO_{\Gr_{s_i w}}] & (s_i w >_\si w) \\ [\cO _{\Gr_{w}}] & (s_i w <_\si w)\end{cases}.$$
\end{lem}

\begin{proof}
By Theorem \ref{LSS-formula} and Corollary \ref{indH}, we can replace $[\cO _{\Gr_{w}}]$ with $[\cO _{\Gr_{wt_{\beta}}}]$ for $\beta \in Q^{\vee}$ such that $\left< \beta, \varpi_i \right> \ll 0$ for all $i \in \tI$. Therefore, the assertion paraphrases Theorem \ref{sH-reg} and Theorem \ref{sH-sph} as $s_i w >_\si w$ is equivalent to $s_i w t_{\beta} > w t_{\beta}$ (see (\ref{si-ord})).
\end{proof}

\begin{lem}\label{cyclic}
The vector space $K_H ( \Gr )_{\lo}$ is a cyclic module with respect to the action of $\sH \otimes \C [ \mathtt t_{\gamma} \mid \gamma \in Q^{\vee}]$ with its cyclic vector $[\cO_{\Gr_0}]$.
\end{lem}

\begin{proof}
By construction, it suffices to find every $\{ [\cO_{\Gr_\beta}] \}_{\beta \in Q^{\vee}}$ in the linear span of $\sH \cdot \{ \mathtt t_{\gamma} \odot [\cO_{\Gr_0}] \}_{\gamma \in Q^{\vee}}$. This follows from a repeated application of the actions of $\{ D_i \}_{i \in \tI_\af}$ and Theorem \ref{LSS-formula} (cf. \cite[Theorem 4.6]{Kat18}).
\end{proof}

\begin{cor}\label{uniq}
An endomorphism of $K_H ( \Gr )_{\lo}$ as a $\sH \otimes \C [ \mathtt t_{\gamma} \mid \gamma \in Q^{\vee}]$-module is completely determined by the image of $[\cO_{\Gr_0}]$. \hfill $\Box$
\end{cor}

\begin{prop}\label{emb}
By sending $[\cO_{\Gr_0}] \mapsto [\cO_{\bQ_G ( e )}]$, we have a unique injective $\sH$-module morphism
$$K_H ( \Gr )_{\lo} \hookrightarrow K _H ( \bQ_G ^{\mathrm{rat}} )$$
such that twisting by $\mathtt t_{\beta}$ corresponds to the right action of $\beta \in Q^{\vee}$. This map particularly gives
$$[\cO _{\Gr_{u t_{\beta}}}] \mapsto [\cO_{\bQ_G ( u t_{\beta} )}] \hskip 5mm u \in W, \beta \in Q^{\vee}_<.$$
\end{prop}

\begin{proof}
The comparison of the $D_i$-actions on the basis elements in Lemma \ref{s-reg} and Theorem \ref{H-si} implies that we indeed obtain a $\sH$-module inclusion, by enhancing the assignment $[\cO _{\Gr_{u t_{\beta}}}] \mapsto [\cO_{\bQ_G ( u t_{\beta} )}]$ for $u \in W, \beta \in Q^{\vee}_<$ into a $\C P$-module homomorphism. We know the actions of $\mathtt t_{\beta}$ and $\beta$ on the both sides by Theorem \ref{LSS-formula} and (\ref{right-Q}), that coincide on elements that generates $K_H ( \Gr )_\lo$ by the actions of $\C P$ and $\{\mathtt t _{\beta} \}_{\beta \in Q^{\vee}}$. Hence, we deduce a $\sH$-module embedding $K_H ( \Gr )_\lo \hookrightarrow K _H ( \bQ_G ^{\mathrm{rat}} )$ that intertwines the $\mathtt t_{\beta}$-action to the right $\beta$-action. Such an embedding must be unique by Corollary \ref{uniq}.
\end{proof}

\begin{rem}
Lemma \ref{s-reg} is a purely combinatorial statement about the comparison of two orders on $W_\af$. As a consequence, we obtain an embedding
$$K_{\Gm \ltimes \bI} ( \Gr ) \hookrightarrow K_{\Gm \ltimes \bI} ( \bQ_G^{\mathrm{rat}})$$
of nil-DAHA modules (with a parameter $q$) that sends $[\cO _{\Gr_{w}}]$ to $[\cO _{\bQ_G^{\mathrm{rat}}(w)}]$ for each $w \in W_\af^-$ (cf. \cite{KK90,LSS10,KNS17}; see \cite{Kat20} for its further consequences).
\end{rem}

\subsection{The $\sH$-module embedding}\label{pmain}

\begin{thm}\label{main}
We have a $\sH$-module embedding
$$\Phi : K _H ( \Gr )_{\lo} \hookrightarrow K_{H} ( \bQ_G^{\ra} )$$
such that twisting by $\mathtt t_{\beta}$ corresponds to the right action of $\beta \in Q^{\vee}$,
$$[\cO _{\Gr_{u t_{\beta}}}] \mapsto [\cO_{\bQ_G ( u t_{\beta} )}] \hskip 5mm u \in W, \beta \in Q^{\vee}_<,$$
and sends the Pontryagin product on the LHS to the tensor product on the RHS. More precisely, we have: For each $i \in \tI$ and $\xi \in K _H ( \Gr )_{\lo}$, it holds
$$\Phi ( \bh_i \odot \xi ) = H_i ( \Phi ( \xi ) ).$$
In addition, the image of $\Phi$ is precisely the set of finite linear combinations of Schubert classes, that forms a dense subset of $K_{H} ( \bQ_G^{\ra} )$.
\end{thm}

\begin{rem}\label{gen}
{\bf 1)} It is known that $\{ \bh_i \}_{i \in \tI}$, $\C P$, and $\{ \texttt t_{\beta} \}_{\beta \in Q^{\vee}}$ generates the ring $K _H ( \Gr )_{\lo}$. One way to prove it is to apply Proposition \ref{K-gen} to $\mathrm{Im} \, \Phi$; {\bf 2)} We add an extra $\Gm$-action on $K _H ( \Gr )$ and prove an analogue of Theorem \ref{main} in \cite{Kat20} on the basis of the results presented here.
\end{rem}

The rest of this subsection is entirely devoted to the proof of Theorem \ref{main}. The embedding part of Theorem \ref{main} as based $\sH$-modules is already proved in Proposition \ref{emb}. It also implies that the image of this embedding is the set of finite linear combinations of Schubert classes.

Let $i \in \tI$. We have an endomorphism $\Xi ( - \varpi_i )$ of $K _H ( \bQ_G ^{\mathrm{rat}} )$ that commutes with the right $Q^{\vee}$-action and the left $\sH$-action. By Lemma \ref{divc}, the image of $[\cO_{\bQ_G ( e )}]$ under $\Xi ( - \varpi_i )$ belongs to the image of $K_H ( \Gr )_{\lo}$. In particular, $\Xi ( - \varpi_i )$ induces a $\sH$-module endomorphism of $K_H ( \Gr )_\lo$. In particular, $H_i$ also induces an endomorphism of $K_H ( \Gr )_\lo$. We denote the endomorphisms on $K_H ( \Gr )_\lo$ induced by $\Xi ( - \varpi_i )$ and $H_i$ by the same letter.

In order to identify the endomorphisms $\bh_i \odot$ and $H_i$, it suffices to compare some linear combination with the well-understood element, namely $\mathrm{id}$. Therefore, we compare the endomorphisms of $K_H ( \Gr )_\lo$ (as $\C P$-modules) induced by
$$\Theta_i := e ^{- \varpi_i} ( \mathrm{id} - \bh_i \odot )$$
and
$$\Xi ( - \varpi_i ) = e ^{- \varpi_i} ( \mathrm{id} - H_i ).$$
Both the endomorphisms send $[\cO_{\Gr_0}]$ to
$$e ^{- \varpi_i} ( [\cO_{\Gr_\beta}] - [\cO_{\Gr_{s_i t_\beta}}] ) \odot [\cO_{\Gr_{t_\beta}}]^{-1} \hskip 5mm (\beta \in Q^{\vee}_<)$$
by Proposition \ref{emb}, Corollary \ref{h-op}, and Lemma \ref{divc}.

We prove that both of $\Theta_i$ and $\Xi ( - \varpi_i )$ commute with the $\sH \otimes \C [ \mathtt t_{\gamma} \mid \gamma \in Q^{\vee}]$-action. It is Theorem \ref{HQc} for $\Xi ( - \varpi_i )$. Hence, we concentrate on the action of $\Theta_i$.

The action of $\Theta_i$ commutes with $\C P \otimes \C [ \mathtt t_{\gamma} \mid \gamma \in Q^{\vee}]$ as $( K_H ( \Gr )_\lo, \odot )$ is a commutative ring. Thus, Corollary \ref{inv-fac} and Lemma \ref{inv-fac0} (and Theorem \ref{comm}) reduces the problem to 
$$D_j ( e ^{- \varpi_i} ( [\cO_{\Gr_\beta}] - [\cO_{\Gr_{s_i t_\beta}}] ) ) = e ^{- \varpi_i} ( [\cO_{\Gr_\beta}] - [\cO_{\Gr_{s_i t_\beta}}] ) \hskip 5mm j \in \tI, \beta \in Q^{\vee}_<.$$
If $j \neq i$, then we have $s_j s_i t_\beta < s_i t_\beta$ and $s_j t_\beta < t_\beta$. Moreover, we have $D_j ( e ^{- \varpi_i} \bullet ) = e ^{- \varpi_i} D_j ( \bullet )$. It follows that
\begin{align*}
D_j ( e ^{- \varpi_i} ( [\cO_{\Gr_\beta}] - [\cO_{\Gr_{s_i t_\beta}}] ) ) & = e ^{- \varpi_i} D_j ( [\cO_{\Gr_\beta}] - [\cO_{\Gr_{s_i t_\beta}}] )\\
& = e ^{- \varpi_i} ( [\cO_{\Gr_\beta}] - [\cO_{\Gr_{s_i t_\beta}}] ).
\end{align*}
If $j = i$, then we compute as
\begin{align*}
D_i ( e ^{- \varpi_i} ( [\cO_{\Gr_\beta}] - [\cO_{\Gr_{s_i t_\beta}}] ) ) = & \, e ^{- \varpi_i + \al_i} D_i ( [\cO_{\Gr_\beta}] - [\cO_{\Gr_{s_i t_\beta}}] )\\
& + \frac{e ^{- \varpi_i} - e ^{- \varpi_i + \al_i}}{1 - e^{\al_i}} ( [\cO_{\Gr_\beta}] - [\cO_{\Gr_{s_i t_\beta}}] )\\
= & \, e ^{- \varpi_i + \al_i} ( [\cO_{\Gr_\beta}] - [\cO_{\Gr_{t_\beta}}] )\\
& + e ^{- \varpi_i} ( [\cO_{\Gr_\beta}] - [\cO_{\Gr_{s_i t_\beta}}] )\\
= & \, e ^{- \varpi_i} ( [\cO_{\Gr_\beta}] - [\cO_{\Gr_{s_i t_\beta}}] ).
\end{align*}
Hence, $\Theta_i$ defines an endomorphism of $K _H ( \Gr )$ that commutes with the $\sH \otimes \C [ \mathtt t_{\gamma} \mid \gamma \in Q^{\vee}]$-action.

Therefore, Corollary \ref{uniq} guarantees $\Theta_i = \Xi ( - \varpi_i ) \in \End ( K_H ( \Gr )_\lo )$. From this, we also deduce $\bh_i \odot = H_i \in \End ( K_H ( \Gr )_\lo )$ as required.

\subsection{Example: $\mathop{SL} ( 2 )$-case}\label{subsec: sl(2)}

Assume that $G = \mathop{SL} ( 2 )$. We make an identification $P_+ = \Z_{\ge 0} \varpi$, $\al = 2 \varpi$, and $Q^{\vee}_+ = \Z_{\ge 0} \al^{\vee}$. We have $W = \{e, s\}$. Let $\mathtt t$ denote the right $Q^{\vee}$-action on $K_H ( \bQ_G^{\ra} )$ (or $K_{\Gm \ltimes \bI} ( \bQ_G^{\ra} )$) corresponding to $\al^{\vee}$, and let $q$ denote the character of $\Gm$ that acts on the variable $z$ (in $G (\!(z)\!)$) by degree one character (so-called the loop rotation action).

The Pieri-Chevalley rule for $\varpi$ (\cite[Theorem 5.13]{KNS17}) yields the equations:
\begin{align*}
[\cO _{\bQ_G ( e )} ( \varpi )] & = \frac{1}{1 - q^{-1} \mathtt t} ( e ^{\varpi} [\cO_{\bQ_G ( e )} ] +  e ^{- \varpi} [\cO_{\bQ_G ( s )}] )\\
[\cO _{\bQ_G ( s )} ( \varpi )] & = \frac{1}{1 - q^{-1} \mathtt t} ( q^{-1} e ^{\varpi} \mathtt t [\cO_{\bQ_G ( e )} ] +  e ^{- \varpi} [\cO_{\bQ_G ( s )}] ).
\end{align*}

Forgetting the extra $\Gm$-action yield:
\begin{align*}
[\cO _{\bQ_G ( e )} ( \varpi )] & = \frac{1}{1 - \mathtt t} ( e ^{\varpi} [\cO_{\bQ_G ( e )} ] +  e ^{- \varpi} [\cO_{\bQ_G ( s )}] )\\
[\cO _{\bQ_G ( s )} ( \varpi )] & = \frac{1}{1 - \mathtt t} ( e ^{\varpi} \mathtt t [\cO_{\bQ_G ( e )} ] +  e ^{- \varpi} [\cO_{\bQ_G ( s )}] ).
\end{align*}
Here both of the above equations contain
$$\frac{1}{1 - \mathtt t} = \sum_{m \ge 0} \mathtt t^m \in \C [\![\mathtt t]\!],$$
that is a formal sum. In particular, our computations here involve infinite sums.

Inverting this equation yields that
\begin{align*}
[\cO _{\bQ_G ( e )} ( - \varpi )] & = e ^{- \varpi} [\cO_{\bQ_G ( e )} ] -  e ^{- \varpi} [\cO_{\bQ_G ( s )}]\\
[\cO _{\bQ_G ( s )} ( - \varpi )] & = - e ^{\varpi} \mathtt t [\cO_{\bQ_G ( e )} ] +  e ^{\varpi} [\cO_{\bQ_G ( s )}].
\end{align*}
Therefore, we obtain
\begin{align*}
[\cO _{\bQ_G ( e )}] - e^{\varpi} [\cO _{\bQ_G ( e )} ( - \varpi )] & = [\cO_{\bQ_G ( s )}]\\
[\cO _{\bQ_G ( s )}] - e^{\varpi} [\cO _{\bQ_G ( s )} ( - \varpi )] & = e ^{\al} \mathtt t [\cO_{\bQ_G ( e )} ] + ( 1 - e ^{\al} ) [\cO_{\bQ_G ( s )}].
\end{align*}
Applying $\mathtt t ^{-m}$ ($m > 0$) on the both sides, we have
\begin{align*}
[\cO _{\bQ_G ( t_{-m\alpha^{\vee}} )}] - e^{\varpi} [\cO _{\bQ_G ( t_{-m\alpha^{\vee}} )} ( - \varpi )] & = [\cO_{\bQ_G ( s t_{-m\alpha^{\vee}} )}]\\
[\cO _{\bQ_G ( s t_{-m\alpha^{\vee}} )}] - e^{\varpi} [\cO _{\bQ_G ( s t_{-m\alpha^{\vee}} )} ( - \varpi )] & = e ^{\al} [\cO_{\bQ_G ( t_{( 1-m )\alpha^{\vee}} )} ] + ( 1 - e ^{\al} ) [\cO_{\bQ_G ( s t_{-m\alpha^{\vee}} )}].
\end{align*}
In other words, we have
\begin{align*}
H ( [\cO _{\bQ_G ( t_{-m\alpha^{\vee}} )}] ) & = [\cO_{\bQ_G ( s t_{-m\alpha^{\vee}} )}]\\
H ( [\cO _{\bQ_G ( t_{-m\alpha^{\vee}} )}] ) & = e ^{\al} [\cO_{\bQ_G ( t_{( 1-m )\alpha^{\vee}} )} ] + ( 1 - e ^{\al} ) [\cO_{\bQ_G ( s t_{-m\alpha^{\vee}} )}]
\end{align*}
for $H \equiv H_i$ (as we have $|\tI| = 1$). By Theorem \ref{main}, this transplants to
\begin{align*}
\bh \odot [\cO_{\Gr_{t_{-m \al^{\vee}}}}] & = [\cO_{\Gr_{s t_{-m \al^{\vee}}}}]\\
\bh \odot [\cO_{\Gr_{s t_{-m \al^{\vee}}}}] & = e ^{\al} [\cO_{\Gr_{t_{(1-m) \al^{\vee}}}}] + ( 1 - e ^{\al} ) [\cO_{\Gr_{s t_{-m \al^{\vee}}}}]
\end{align*}
for $\bh \equiv \bh_i$. We have $\bh = [\cO_{\Gr_{s t_{-\al^{\vee}}}}] \odot [\cO_{\Gr_{t_{-\al^{\vee}}}}]^{-1}$. By Theorem \ref{LSS-formula}, we conclude
\begin{align*}
[\cO_{\Gr_{s t_{-\al^{\vee}}}}] \odot [\cO_{\Gr_{t_{-m \al^{\vee}}}}] & = [\cO_{\Gr_{s t_{-(m+1) \al^{\vee}}}}]\\
[\cO_{\Gr_{s t_{-\al^{\vee}}}}] \odot [\cO_{\Gr_{s t_{-m \al^{\vee}}}}] & = e ^{\al} [\cO_{\Gr_{t_{-m \al^{\vee}}}}] + ( 1 - e ^{\al} ) [\cO_{\Gr_{s t_{-(m+1) \al^{\vee}}}}].
\end{align*}
for $m > 0$. This coincides with the calculation in \cite[(17)]{LLMS17}.

\section{Relation with quantum $K$-group}

We continue to work in the setting of the previous section. In this section, we establish an isomorphism between the equivariant $K$-group of $\bQ_G^\ra$ (\S \ref{subsec:eK}) and the equivariant quantum $K$-group of $\sB$ (\S \ref{subsec:eqK}). The main tool here is the comparison of \S \ref{subsec:sif}--\ref{subsec:eK} and the calculations on the quasi-map spaces (\S \ref{subsec:QM}) as briefly recalled in \S \ref{subsec:QJ}. The passage from the original definition of the $K$-theoretic Gromov-Witten correlation functions (\S \ref{subsec:GQL}) and the corresponding quantities on quasi-map spaces is also briefly explained in \S \ref{subsec:QJ} after Givental-Lee \cite{GL03} and Braverman-Finkelberg \cite{BF14b}.

\subsection{Quasi-map spaces}\label{subsec:QM}

Here we recall basics of quasi-map spaces from \cite{FM99,FFKM}.

A quasi-map $( f, D )$ is a morphism $f : \P ^1 \rightarrow \sB$ together with a $\Pi^{\vee}$-colored effective divisor
\begin{equation}
D = \sum_{\alpha^{\vee} \in \Pi^{\vee}, x \in \P^1 (\C)} m_x (\alpha^{\vee}) \alpha^{\vee} \otimes [x] \in Q^{\vee} \otimes_\Z \mathrm{Div} \, \P^1 \hskip 3mm \text{with} \hskip 3mm m_x (\alpha^{\vee}) \in \Z_{\ge 0},\label{totdef}
\end{equation}
where the sum is essentially finite, i.e. $m_x(\alpha^{\vee})=0$ except for finitely many $x\in \P^1 (\C)$. We call $D$ the defect of $(f, D)$. We set $[D] := \{ x \in \P^1 ( \C ) \mid \sum_{\alpha^{\vee}\in \Pi^{\vee}} m_x ( \alpha^{\vee} ) > 0 \} \subset \P^1$ and call it the defect locus of $(f, D)$. We remark that $f$ can be also understood as a rational map defined outside of $[D]$ in view of the valuative criterion of properness.

We call $\sum_{\al \in \Pi^{\vee}} m_x (\alpha^{\vee}) \alpha^{\vee}$ the defect of $(f,D)$ at $x \in \P^1 ( \C )$ (that we denote by $| D |_x$). Here we define the total defect of  $( f, D )$ by
$$|D| := \sum_{\alpha^{\vee} \in \Pi^{\vee}, x \in \P^1 (\C)} m_x (\alpha^{\vee}) \alpha^{\vee} = \sum_{x \in \P^1 (\C)} |D|_x \in Q_+^{\vee}.$$

For each $\beta \in Q_+^{\vee}$, we set
$$\sQ ( \sB, \beta ) : = \{ (f, D) \mid \text{a quasi-map s.t. } f _* [ \P^1 ] + | D | = \beta \in H_2 (\sB,\Z) \},$$
where $f_* [\P^1]$ is the class of the image of $\P^1$ multiplied by the degree of $\P^1 \to \mathrm{Im} \, f$. We denote $\sQ ( \sB, \beta )$ by $\sQ ( \beta )$ in case there is no danger of confusion. We understand that $\sQ (\beta) = \emptyset$ for $\beta \in Q^{\vee} \setminus Q^{\vee}_+$.

\begin{defn}[Drinfeld-Pl\"ucker data]\label{Zas}
Consider a collection $\mathcal L = \{( \psi_{\lambda}, \mathcal L^{\lambda} ) \}_{\lambda \in P_+}$ of inclusions $\psi_{\lambda} : \mathcal L ^{\lambda} \hookrightarrow L ( \lambda ) \otimes _{\C} \mathcal O _{\P^1}$ of line bundles $\mathcal L ^{\lambda}$ over $\P^1$ into the trivial vector bundles $L( \la ) \otimes _{\C} \mathcal O _{\P^1}$ equipped with non-trivial $G$-actions afforded in \S \ref{subsec:prelim} (as coherent subsheaves). The data $\mathcal L$ is called a Drinfeld-Pl\"ucker data (DP-data) if the canonical inclusion of $G$-modules
$$\eta_{\lambda, \mu} : L ( \lambda + \mu ) \hookrightarrow L ( \lambda ) \otimes L ( \mu )$$
induces an isomorphism
$$\eta_{\lambda, \mu} \otimes \mathrm{id} : \psi_{\lambda + \mu} ( \mathcal L ^{\lambda + \mu} ) \stackrel{\cong}{\longrightarrow} \psi _{\lambda} ( \mathcal L^{\lambda} ) \otimes_{\cO_{\P^1}} \psi_{\mu} ( \mathcal L^{\mu} )$$
for every $\lambda, \mu \in P_+$.
\end{defn}

\begin{thm}[Drinfeld, see Finkelberg-Mirkovi\'c \cite{FM99}]\label{Dr}
The variety $\sQ ( \beta )$ is isomorphic to the variety formed by isomorphism classes of the DP-data $\mathcal L = \{( \psi_{\lambda}, \mathcal L^{\lambda} ) \}_{\lambda \in P_+}$ such that $\deg \, \mathcal L ^{\lambda} = \left< w_0 \beta, \lambda \right>$. In addition, we have
$$\dim \, \sQ ( \beta ) = \dim \sB + 2 \left< \beta,\rho \right>.$$
\end{thm}

For each $\beta \in Q^{\vee}_+$ and $w \in W$, we consider two varieties:
\begin{align*}
\mathring{\sQ} ( \beta ) & := \{ (f, D) \in  \sQ ( \beta ) \mid 0 \not\in [D] \} \subset \sQ ( \beta ), \\
\mathring{\sQ} ( \beta, w ) & := \{ (f, D) \in  \mathring{\sQ} ( \beta ) \mid f ( 0 ) \in \bO_\sB ( w ) \}.
\end{align*}
In case $\beta , \gamma \in Q^{\vee}_+$, we have a closed embedding
$$\imath_{\gamma} : \sQ ( \beta ) \ni (f, D) \mapsto (f, D + \gamma [0])\in \sQ ( \beta + \gamma ).$$
We set $\sQ ( \beta, w  ) := \overline{\mathring{\sQ} ( \beta, w  )} \subset \sQ ( \beta )$.

For each $\lambda \in P$, $w \in W$, and $\beta \in Q_+^{\vee}$, we have a $G$-equivariant line bundle $\cO _{\sQ ( \beta, w )} ( \lambda )$ obtained by the (tensor product of the) pull-backs $\cO _{\sQ ( \beta, w )}( \varpi_i )$ of the $i$-th $\cO ( 1 )$ via the closed embedding
\begin{equation}
\sQ ( \beta, w ) \hookrightarrow \prod_{i \in \mathtt I} \P ( L ( \varpi_i )^* \otimes_{\C} \C [z] _{\le \left< \beta, \varpi_i \right>} ),\label{Pemb}
\end{equation}
where we make an identification between the DP-data and the coordinates in the RHS of (\ref{Pemb}) by using
\begin{align*}
\Hom_{\cO_{\P^1}} ( \mathcal L^{\varpi_{\bar{\imath}}}, L ( \varpi_i ) \otimes_\C \cO_{\P^1} ) & \cong \Hom_{\cO_{\P^1}} ( \cO_{\P^1} ( \left< w_0 \beta, \varpi_i \right> ), L ( \varpi_i ) \otimes_\C \cO_{\P^1} )\\
& \cong \Hom_{\cO_{\P^1}} ( \cO_{\P^1} ( \left< w_0 \beta, \varpi_i \right> [\infty] ), L ( \varpi_i ) \otimes_\C \cO_{\P^1} )\\
& \cong L ( \varpi_i ) \otimes_{\C} \C [z] _{\le \left< \beta, - w_0 \varpi_i \right>} \hskip 5mm i \in \tI.
\end{align*}
Here we use an identification $L ( \varpi_i )^* \cong L ( - w_0 \varpi_i )$ ($i \in \tI$) in the comparison of the DP-data and (\ref{Pemb}).

By comparing (\ref{Pemb}) and (\ref{formal-proj-emb}), we find an embedding
$$\mathring{\sQ} ( \beta ) \subset G \bO( e ) \subset \bQ_G \hskip 10mm \beta \in Q^{\vee}_+.$$
In particular, every two (closed) points of $\mathring{\sQ} ( \beta )$ are transferred to each other by the $G[\![z]\!]$-action. In the same vein, every two points of $\mathring{\sQ} ( \beta, w )\subset \bO(w)$ ($\beta \in Q^{\vee}_+, w \in W$) are transferred to each other by the $\bI$-action.

We have $\sB = \sQ ( 0 )$ by the Pl\"ucker embedding. By expanding the map
$$\P^1 \to \sB \hookrightarrow \prod_{i \in \tI} \P ( L ( \varpi_i )^* )$$
into a collection of formal power series $L ( \varpi_i )^* \otimes \C [\![z]\!]$ ($i \in \tI$) using the coordinate variable $z$ on $\A^1 \subset \P^1$ that admits a degree one $\Gm$-action, we find an embedding $\sQ ( \beta ) \subset \bQ_G$ by (\ref{UN}). These result in the closed embeddings $\sB \subset \sQ ( \beta ) \subset \bQ_G$ for each $\beta \in Q^{\vee}_+$ such that the line bundles $\cO ( \la )$ ($\la \in P$) corresponds to each other by restrictions (cf. \S \ref{subsec:sif} and \cite[\S 4.4]{Kat18d}).

The $\Gm$-fixed points in the RHS (\ref{Pemb}) is
\begin{equation}
\prod_{i \in \tI} \Bigl( \bigsqcup_{m_i=0}^{\left< \beta, \varpi_i \right>} \P ( L ( \varpi_i )^* \otimes_{\C} \C z^{m_i} )\Bigr) \cong \prod_{i \in \tI} \Bigl( \bigsqcup_{m_i=0}^{\left< \beta, \varpi_i \right>} \P ( L ( \varpi_i )^* )\Bigr)  \label{eqn:torus-fixed}
\end{equation}
by inspection. Since $S^1 \subset \C^{\times}$ is Zariski dense, we find that that (\ref{eqn:torus-fixed}) is also the $S^1$-fixed points of the RHS of (\ref{Pemb}). In view of Theorem \ref{Dr} and (\ref{Pemb}), we find that the set of $S^1$-fixed  points of $\sQ ( \beta )$ is the disjoint union of $\sB$, one for each $0 \le \gamma \le \beta$.

\subsection{Factorization structure and its consequences}\label{subsec:fact}

The contents of this subsection are used only in the next section, and in the explanation of Theorem \ref{zas}. Hence, this subsection can be safely skipped to understand Theorem \ref{wJcomp} and Corollary \ref{Jcomp} if one can admit Theorem \ref{zas} without an additional explanation.

Here we temporarily switch\footnote{Although the whole results are of algebraic nature as recorded in \cite{BFGM, BFG06}, it looks simpler to present them by their analytic counterparts, see Remark \ref{rem:loc}.} to the complex analytic topology in order to state Theorem \ref{fact}. For each $\beta \in Q^{\vee}_+$, we set $\sZ ( \beta ) := \sQ ( \beta ) \cap \bO ( w_0 )$ and call it the zastava space (of degree $\beta$). This is an affine open subset of $\sQ(\beta,w_0)$ that is stable under the action of $( \Gm \times B )$. We set
$$C ^{(\beta)} : = \prod_{i \in \tI} \left( C^{m_i} / \mathfrak S_{m_i} \right),  \hskip 2mm \text{where} \hskip 2mm \beta = \sum_{i \in \tI} m_i \alpha_i^{\vee}$$
for a Riemann surface $C$ (or a finite set of points), where $\mathfrak S_m$ is the symmetric group of order $m$. We note that the space $C ^{(\beta)}$ is the same as the space of $\Pi^{\vee}$-colored divisors of degree $\beta$ on $C$. We set $\A^1_x := \mathbb P^1 \setminus \{ x \}$ for a point $x \in \mathbb P^1$ for the sake of notational simplicity, that we may regard it as an open Riemann surface and an algebraic curve interchangeably.

For each $i \in \tI$, a point $(f, D) \in \sZ ( \beta )$ defines an element of $u_i ( f, D ) \in L ( \varpi_i )^* \otimes_{\C} \C [z]$ through (\ref{Pemb}) for each $i \in \tI$ (up to a scalar multiple), that also yields a polynomial $\phi_i ( f, D ; z ) \in \C [z]$ by pairing with the lowest weight vector of $L ( \varpi_i )$. By examining the roots of $\phi_i ( f, D ; z )$ (the multiplicity at $\infty$ is understood as $\left< \beta, \varpi_i \right> - \deg \, \phi_i ( f, D ; z)$), we obtain the factorization morphism
$$\mathfrak f^{\beta} : \sZ ( \beta ) \longrightarrow ( \A^1_0 )^{( \beta )}$$
since $0 \in \P^1$ is never a root of such polynomials (see e.g. \cite[\S 5.2.2]{FM99}). By construction, the point $\mathfrak f^{\beta} ( f, D ) \in ( \P^1 )^{(\beta)}$ contains at least $\left< |D|_x, \varpi_i \right>$-copies of the point $x \in \P^1$ in the $i$-th configuration for $(f, D) \in \sZ ( \beta )$. The constant map $c: \P^1 \rightarrow B/B \subset \sB$ yields a point $( c, \beta [\infty]) \in \sZ ( \beta )$, that we refer as the origin $\mathbf 0$ of $\sZ ( \beta )$. The $\Gm$-action attracts every point of $\sZ (\beta)$ to the origin.

\begin{thm}[Finkelberg-Mirkovi\'c \cite{FM99} \S 6.3.2]\label{fact}
Let $\beta, \beta_1, \beta_2$ be elements in $Q^{\vee}_+$ such that $\beta = \beta_1 + \beta_2$, and let $\mathcal U_1, \mathcal U_2 \subset \A^1_0$ be a pair of disjoint complex analytic open subsets. We have an isomorphism of complex analytic sets
$$( \mathfrak f ^{\beta} )^{-1} ( \mathcal U_1 ^{(\beta_1)} \times \mathcal U_2 ^{(\beta_2)} ) \cong ( \mathfrak f^{\beta_1} )^{-1} ( \mathcal U_1 ^{(\beta_1)} ) \times ( \mathfrak f^{\beta_2} )^{-1} ( \mathcal U_2 ^{(\beta_2)} ),$$
where $\mathcal U_1 ^{(\beta_1)} \times  \mathcal U_2 ^{(\beta_2)} \subset ( \P^1 )^{( \beta )}$ is a natural inclusion.
\end{thm}

\begin{rem}\label{rem:fact}
{\bf 1)} Theorem \ref{fact} implies that $\sZ ( \beta )$, and hence also $\sQ ( \beta )$, must be irreducible as an algebraic variety over $\C$ (through the projectivity of $\sQ ( \beta )$ and the GAGA \cite{Ser55}); {\bf 2)} In \cite{BF14a,BF14c}, it is established that $\sZ ( \beta )$ is a normal variety and the morphism $\mathfrak f^{\beta}$ is flat (in the course of their proof of Theorem \ref{rat-res}).
\end{rem}

\begin{rem}\label{rem:loc}
In Theorem \ref{fact}, we can fix a point $x = (f, D) \in ( \mathfrak f^{\beta_1} )^{-1} ( \mathcal U_1 ^{(\beta_1)} )$ and take the completion $\mathcal O^{\wedge}_x$ of the germ of the analytic structure sheaf of $( \mathfrak f ^{\beta_1} )^{-1} ( \mathcal U_1 ^{(\beta_1)} )$ at a point $x$. Theorem \ref{fact} asserts that $\Spec \, \mathcal O^{\wedge}_x$ is the formal completion of the normal direction of
$$\{ x \} \times ( \mathfrak f^{\beta_2} )^{-1} ( \mathcal U_2 ^{(\beta_2)} ) \subset \sZ ( \beta ).$$
In this description, $\mathfrak f^{\beta_1} ( x )$ corresponds to a specific configuration of points in $\A^1_0$, that defines a finite set $[D] \subset S \subset \mathbb A^1_{0}$. Therefore, we can set $\mathcal U _2 := ( \mathbb A^1_{0} \setminus S )^{(\beta_2)}$ as the limiting case. In view of Theorem \ref{fact}, we obtain the locally ringed spaces and the maps
\begin{equation}
\{ x \} \times ( \mathfrak f^{\beta_2} )^{-1} ( ( \mathbb A^1_{0} \setminus S )^{(\beta_2)} ) \subset \Spec \, \mathcal O^{\wedge}_x \times ( \mathfrak f^{\beta_2} )^{-1} ( ( \mathbb A^1_{0} \setminus S )^{(\beta_2)} ) \stackrel{\wp_x}{\longrightarrow} \mathcal Z ( \beta ).\label{loc-sch}
\end{equation}
The algebraic version (\cite{BFGM}) of the factorization map $\mathfrak f^{(\beta_1,\beta_2)}$ is over
\begin{equation}
\Bigl( ( \mathbb A^1_0 )^{(\beta_1)} \times ( \mathbb A^1_0 )^{(\beta_2)} \Bigr) \setminus \Delta^{(\beta_1,\beta_2)},\label{covA}
\end{equation}
that is a local covering of $\mathbb A_0^{(\beta)}$, where $\Delta^{(\beta_1,\beta_2)}$ denotes the divisor such that the first group of configuration of points and the second group of configuration of points have a common point of $\mathbb A^1_0$. The analytification of $\mathfrak f^{(\beta_1,\beta_2)}$ is the base change of $\mathfrak f^{\beta}$ from $\mathbb A_0^{(\beta)}$ to (\ref{covA}). We have
$$\mathfrak f^{\beta_1} ( x ) \times ( \mathbb A^1_0 \setminus S )^{(\beta_2)} \subset \Bigl( ( \mathbb A^1_0 )^{(\beta_1)} \times ( \mathbb A^1_0 )^{(\beta_2)} \Bigr) \setminus \Delta^{(\beta_1,\beta_2)}.$$
It follows that the natural map
$$\widetilde{\wp}_x : ( \mathfrak f^{(\beta_1,\beta_2)} )^{-1} ( \Bigl( ( \mathbb A^1_0 )^{(\beta_1)} \times ( \mathbb A^1_0 )^{(\beta_2)} \Bigr) \setminus \Delta^{(\beta_1,\beta_2)} ) \rightarrow \sZ ( \beta )$$
factors through the complex analytic map $\wp_x$ when we restrict the domain to $( \mathfrak f^{\beta} )^{-1} \bigl( \mathfrak f^{\beta_1} ( x ) \times ( \mathbb A^1_0 \setminus S )^{(\beta_2)} \bigr)$. Hence, $\wp_x$ is the analytification of an algebraic map amplified by a trivial fiber space structure offered by $\Spec \, \mathcal O_x^{\wedge}$, that is possible since the formal completion with respect to a maximal ideal is in common between an algebraic variety and its analytification.

Therefore, we can regard (\ref{loc-sch}) as schemes and morphisms between them. Now the base space $\mathfrak f^{\beta_1} ( x ) \times ( \A_0^1 \setminus S )^{(\beta_2)}$ is a subvariety of $( \A_0^1)^{(\beta)}$. This exhibits the reason why we need formal completions in the proof of Theorem \ref{Q-rat-sing}, and why Proposition \ref{ltr} can be useful in its analysis.
\end{rem}

\begin{prop}\label{ltr}
Let $\beta , \gamma \in Q^{\vee}_+$. For each $y \in \mathring{\sQ} ( \beta )$, we have a morphism
$$\eta : \mathcal S_y  \times U \hookrightarrow \sQ ( \beta + \gamma ),$$
where $\mathbf 0 \in \mathcal S_y \subset \sZ ( \gamma )$ and $y \in U \subset \mathring{\sQ} ( \beta )$ are $($complex analytic$)$ open subsets and $\imath_{\gamma} ( y ) = \eta ( \mathbf 0, y )$.
\end{prop}

We call $\mathcal S_y$ in Proposition \ref{ltr} a local transversal slice of $\mathring{\sQ} ( \beta ) \subset \sQ ( \beta + \gamma )$ along $y \in \mathring{\sQ} ( \beta )$.

\begin{proof}[Proof of Proposition \ref{ltr}]
We make a swap $z \mapsto z^{-1}$ that acts on the coordinate of $\mathbb P^1$ (and hence the origin of a zastava space have defect only at $0$ instead of $\infty$). By convention, $\mathcal Z ( \beta + \gamma )$ now consists of quasi-maps $(f, D)$ of degree $( \beta + \gamma )$ such that $\infty \not\in [D]$ and $f ( \infty ) \in B / B$. Since the former is an open condition and $N^- \times B/B \subset \sB$ is open dense, we deduce $N^- \times \sZ ( \beta + \gamma ) \subset \sQ ( \beta + \gamma )$ is also a dense open subset. Since $y \in \mathring{\sQ} ( \beta )$, we find that the quasi-map $(f_y, D_y)$ corresponding to $\imath_{\gamma} ( y ) \in \sQ ( \beta + \gamma )$ (that we might also denote by $y$ in the following) satisfies $|D_y|_0 = \gamma$ and $|D_y|_p = 0$ for some $p \in \A^1_{\infty}$. Let us apply the action of some $\psi\in\mathop{PSL} ( 2, \C )$ on $\P^1$ that fixes $0$ and send $p$ to $\infty$. In addition, we apply some $g \in G$ such that $g f_y (p) \in B/B$. The actions of $\psi$ and $G$ preserve $\mathring{\sQ} ( \beta )$ (as $0 \in \P^1$ is fixed and the space is $G$-stable). Thus, we have $\imath_\gamma ( y ) \in ( \psi^{-1} \times g^{-1} )( N^- \times \sZ ( \beta + \gamma ) ) \cap \imath_{\gamma} ( \mathring{\sQ} ( \beta ) ) \subset \sQ ( \beta + \gamma )$. In particular, it suffices to choose $\imath_\gamma ( y ) \in \sZ ( \beta + \gamma ) \cap \imath_\gamma ( \mathring{\sQ} ( \beta ) )$ to construct a local transversal slice.

Now we have $|D_y|_0 = \gamma$ and $|D_y|_{\infty} = 0$. Since the order of common zero at $0$ of the vector valued function $u_i ( f, D )$ is exactly $\left< \gamma, \varpi_i \right>$, we have some $h \in N$ such that $h u_i ( f, D )$, when paired with the lowest weight vector of $L(\varpi_i)$, yields zero of order exactly $\left< \gamma, \varpi_i \right>$ at $0$ (since the $N$-action on $L ( \varpi_i )^*$ is cocyclic to the highest weight vector, we can throw in the lowest $z$-degree part into the coefficient of the highest weight vector by the $N$-action). The twist by $h$ preserves $\sZ ( \beta + \gamma )$, and changes the factorization morphism only. We employ this modified factorization morphism (that we denote by $\mathfrak f$) and define an enhanced version of the factorization morphism
$$\mathfrak f^{\mathsf{enh}} : N^- \times \sZ ( \beta + \gamma ) \ni ( h, (f, D)) \mapsto \mathfrak f ( f, D ) \in ( \mathbb A_{\infty} )^{(\beta + \gamma)}.$$
Let us find disjoint open subsets $\mathcal U_1, \mathcal U_2 \subset \mathbb A^1_{\infty}$ such that $0 \in \mathcal U_1$ and $\mathcal U_2$ contains the support of the configuration of points $\mathfrak f^{\mathsf{enh}} ( y )$ except for $0$. By identifying $\mathfrak f$ with the original factorization morphism, Theorem \ref{fact} yields an isomorphism
$$ (\mathfrak f^{\mathsf{enh}})^{-1} ( \mathcal U_1^{(\gamma)} \times \mathcal U_2^{(\beta)}) \cong ( \mathfrak f^{\gamma} )^{-1} ( \mathcal U_1^{(\gamma)} ) \times \left( N^- \times ( \mathfrak f^{\beta} )^{-1} ( \mathcal U_2^{(\beta)} ) \right)$$
that defines an open subset of $\sQ ( \beta + \gamma )$ such that
$$( \mathfrak f^{\gamma} )^{-1} ( \mathcal U_1^{(\gamma)} ) \times \left( N^- \times ( \mathfrak f^{\beta} )^{-1} ( \mathcal U_2^{(\beta)} ) \right) \cap \imath_{\gamma} ( \mathring{\sQ} ( \beta ) ) = \{ \mathbf 0 \} \times \left( N^- \times ( \mathfrak f^{\beta} )^{-1} ( \mathcal U_2^{(\beta)} ) \right)$$
and $y \in ( \mathfrak f^{\beta} )^{-1} ( \mathcal U_2^{(\beta)} ) \subset N^- \times ( \mathfrak f^{\beta} )^{-1} ( \mathcal U_2^{(\beta)} )$. Now we set
$$\mathcal S_y := ( \mathfrak f^{\gamma} )^{-1} ( \mathcal U_1^{(\gamma)} ) \hskip 5mm \text{and} \hskip 5mm U :=  \left( N^- \times ( \mathfrak f^{\beta} )^{-1} ( \mathcal U_2^{(\beta)} ) \right)$$
to conclude the assertion.
\end{proof}

\begin{cor}\label{ltr-comp}
Keep the setting of Proposition \ref{ltr} $($with possible rearrangement of $\mathcal S_y$ and $U)$. Let $p \in \P^1 (\C)$. If $y = (f, D)$ satisfies $0 \neq p \not\in [D]$, then we have the following commutative diagram for an arbitrary $\delta \in Q^{\vee}_+$:
\begin{equation*}
	\xymatrix{
	& \mathcal S_y \times U \ar@{^{(}->}[r] \ar@{^{(}->}[d]_{\mathrm{id} \times \imath}& \sQ ( \beta + \gamma )\ar@{^{(}->}[d]_{\imath'} & \mathring{\sQ} ( \beta )\ar@{^{(}->}[d]_{\imath''}\ar@{_{(}->}[l]^{\hskip 4mm \imath_\gamma}\\
	\mathcal S_y \times U \times U' ( \delta ) \ar[r]^{\hskip 4mm \cong}& \mathcal S_y \times U ( \delta ) \ar@{^{(}->}[r] & \sQ ( \beta + \gamma + \delta ) & \mathring{\sQ} ( \beta + \delta )\ar@{_{(}->}[l]^{\hskip 4mm \imath_\gamma}
	}
\end{equation*}
such that
$$\mathcal S_y \times U \times \{ \mathbf 0_\delta \} \stackrel{\cong}{\longrightarrow} \left ( ( \mathcal S_y \times U ( \delta ) ) \cap \sQ ( \beta + \gamma ) \right) \subset \sQ ( \beta + \gamma + \delta ),$$
where the map $\imath' : \sQ ( \beta + \gamma ) \hookrightarrow \sQ ( \beta + \gamma + \delta )$ is obtained by adding the defect $\delta [p]$ to each point, $U ( \delta ) \subset \mathring{\sQ} ( \beta + \delta )$ is an neighborhood of the image of $y$ in $\mathring{\sQ} ( \beta + \delta )$, $U' ( \delta ) \subset \sZ (\delta)$ is an open neighborhood of the origin $\mathbf 0_\delta \in \sZ (\delta)$, and $\imath$ and $\imath''$ are the induced maps.
\end{cor}

\begin{proof}
In the proof of Proposition \ref{ltr}, we can modify the factorization morphism $\mathfrak f^{(\beta)}$ and $\mathcal U_2$ if necessary to assume that $\mathcal U_2 = \mathcal U_2 ^{(1)} \sqcup \mathcal U_2 ^{(2)}$, $p \in \mathcal U_2^{(2)}$, and $\mathfrak f^{\beta} ( y )$ does not contain a point in $\mathcal U_2^{(2)}$. Then, Theorem \ref{fact} separates out the effect of $\delta [p]$ as a product factor $U' ( \delta )$ isomorphic to an open neighborhood of $( c, \delta [p] ) \in \sZ ( \delta )$ (where $c$ is the constant map to $B/B \subset \sB$) as required.
\end{proof}

In view of Proposition \ref{ltr} and Corollary \ref{ltr-comp}, the structure of local transversal slices of the open subset $\mathring{\sQ} ( \beta )$ with respect to $\imath_{\gamma}$ only depends on $\gamma$, and not on the choice of $\beta$. We need an analogous local transversal slices inside $\sQ ( \beta + \gamma, w )$ in the course of our proof of Theorem \ref{Q-rat-sing}. The main obstacle there is that it is not clear whether a local transversal slice exists on the neighborhood of every point of $\mathring{\sQ} ( \beta )$ in a uniform fashion, that is guaranteed by Theorem \ref{fact} when $w = e$. This uniformity is resurrected by identifying the situation with the (formal completions of the) transversal slices between $\mathbf I$-orbits of $\bQ_G^\ra$ by the fact that all the points of $\mathring{\sQ} ( \beta )$ lie on the same $G[\![z]\!]$-orbit (note that two points in $\mathring{\sQ} ( \beta )$ are not transferred to each other by the action of the smaller group $G[z]$ in general as it preserves the defect at $\P^1 \setminus \{0,\infty\}$).

\subsection{$K$-theoretic $J$-functions and generating functions}\label{subsec:QJ}

In this subsection, we reformulate results provided in Givental-Lee \cite{GL03} and Braverman-Finkelberg \cite{BF14a}. Hence, both ``theorems" in this subsection are understood as blends of their results, and their ``proofs" are just explanations on how they work.

\begin{thm}\label{zas}
There exists an element $J ' ( Q, q ) \in ( \C P [\![q]\!] ) [\![Q^{\vee}_+]\!] \cap \C ( P, q ) [\![Q^{\vee}_+]\!]$ with the following properties:
\begin{enumerate}
\item the composition of maps
$$( \C P [\![q]\!] )[\![Q^{\vee}_+]\!] \cong ( K_G ( \sB ) [\![q]\!] ) [\![Q^{\vee}_+]\!] \subset ( K_H ( \sB ) [\![q]\!] ) [\![Q^{\vee}_+]\!]$$
sends $J ' ( Q, q )$ to $J ( Q, q )$;
\item for each $\la \in P$, we have an identity in $( \C [q^{\pm 1}] P ) [\![Q^{\vee}_+]\!]$:
$$D_{w_0} ( J ' ( Qq^{\la}, q ) e^{w_0 \la} J' ( Q, q^{-1} ) ) = \sum_{\beta \in Q^{\vee}_+} \chi_{q} ( \sQ ( \beta ), \cO_{\sQ} ( \la ) ) Q^{\beta},$$
where we understand that $Qq^{\la}$ sends $Q^{\beta}$ to $Q^{\beta} q^{- \left< \beta, \la \right>}$ for each $\la \in P$.
\end{enumerate}
\end{thm}

\begin{proof}
Theorem \ref{zas} is proved in \cite[\S 2.2]{GL03} for the case $G = \mathop{SL} (n, \C )$ and is extended to the case of general $G$ in \cite[\S 1.3]{BF14a}. The discussion in \cite[\S 1.3]{BF14a} seems to aim to convince readers who are acquainted with \cite[\S 2.2]{GL03} in the presence of their main result recorded here as Theorem \ref{rat-res}. To that end, here we spell out some portion of the discussions in \cite[\S 2.2]{GL03} in the setting of the space of stable maps with the aid of \cite{IMT15}, so that it might clarify why the authors of \cite{BF14a,BF14b} do not include a detailed discussion about the reasoning of Theorem \ref{zas} for general $G$.

By the inclusion $\sQ ( \beta ) \subset \bQ_G$, the set of $(\Gm \times H)$-fixed points of $\sQ ( \beta )$ is in bijection with
$$\{ p_{u t_{\gamma}} \mid u \in W, 0 \le \gamma \le \beta \} \subset W_\af.$$
We have $p_{t_{\gamma}} \in \imath_{\gamma} ( \mathring{\sQ} ( \beta - \gamma ) )$. In view of Proposition \ref{ltr}, we have a $( S^1 \times H )$-stable analytic open neighborhood $\mathfrak U ( \gamma ) \subset \sQ ( \beta )$ of $p_{t_{\gamma}}$ that is isomorphic to an open subset of
\begin{equation}
N p_{t_{\gamma}} \times \sZ ( \beta - \gamma ) \times \sZ ( \gamma ),\label{Lnbd}
\end{equation}
where the first factor $N p_{t_{\gamma}}$ is $S^1$-invariant, the $(S^1 \times H)$-action on the second factor $\sZ ( \beta - \gamma )$ is twisted by $\dot{w}_0$ on $H$, and the $(S^1 \times H)$-action on the third factor $\sZ ( \gamma )$ is twisted by $(\mathrm{inv},\dot{w}_0)$ on $(S^1 \times H )$. The factor $N p_{t_{\gamma}}$ is a Zariski open dense subset of a connected component of the $S^1$-fixed part of $\sQ ( \beta )$, that is isomorphic to $\sB$. (Here the reduction of $S^1 \subset \Gm$ enabled us to separate out the open subset of the image of the factorization map with respect to the absolute values of the coordinate on $\P^1$.) In particular,  we can take
$$\mathfrak U ( 0 ) = \sQ ( \beta ) \setminus \bigcup_{w > e} \sQ ( \beta, w )$$
and it acquires the structure of an algebraic variety with a unique attracting $( \Gm \times H )$-fixed point $p_{t_0}$ in the sense its analytification recovers $\mathfrak U ( 0 )$ as a complex analytic space. (The limit of the $\Gm$-action here is toward $0$ inside $\R_{>0} \subset \Gm ( \R )$.) Henceforth, we regard $\mathfrak U ( 0 )$ as an algebraic variety in the following.

We have a rational resolution of $\sQ ( \beta )$ offered by the space of stable maps (see Theorem \ref{X-can}), that we denote as
$$\pi_{\beta} : \sX ( \beta ) \rightarrow \sQ ( \beta ).$$

We set $\widetilde{\mathfrak U} := \pi_{\beta}^{-1} ( \mathfrak U ( 0 ) )$ and $\mathfrak C (\beta) := \pi_{\beta}^{-1} ( p_{t_0} )$. Since $\pi_\beta$ is the rational resolution of singularities in our sense (Remark \ref{Rem:Sing}), we have
\begin{equation}
\left( \R^{\bullet} ( \pi_{\beta} )_* \cO_{\widetilde{\mathfrak U}} \right) \biggl\lvert _{\mathfrak U (0)}\biggr. \cong \left( ( \pi_{\beta} )_* \cO_{\widetilde{\mathfrak U}} \right) \biggl\lvert _{\mathfrak U (0)}\biggr. \cong \cO_{\mathfrak U (0)} \cong \cO_{N} \boxtimes \cO_{\sZ ( \beta )}.\label{Zpush}
\end{equation}
Here $\pi_{\beta}$ is proper by the definition of the resolution of singularities, and thus $\mathfrak C (\beta)$ is a smooth proper orbifold. Note that a smooth proper orbifold that is locally an algebraic variety defines a normal algebraic variety that admits at worst quotient singularities. In particular, we have
\begin{equation}
\gch \, \Gamma ( \widetilde{\mathfrak U}, \cO_{\widetilde{\mathfrak U}} ) = \gch \, \Gamma ( \mathfrak U ( 0 ),\left( \R^{\bullet} ( \pi_{\beta} )_* \cO_{\widetilde{\mathfrak U}} \right) \biggl\lvert _{\mathfrak U (0)}\biggr. ) = \gch \, \C [N] \cdot \gch \, \C [\sZ (\beta )].\label{Zpush-num}
\end{equation}

Since the $\Gm$-action attracts $\mathfrak U ( 0 )$ to $N p_{t_0}$, it attracts $\widetilde{\mathfrak U}$ to $N \times \mathfrak C (\beta)$. By the Bia{\l}ynicki-Birula theorem (see \cite[Theorem 2.5]{BB73} and \cite[Lemma 7]{KP01}), the orbifold $\widetilde{\mathfrak U}$ defines a vector bundle on $N \times \mathfrak C ( \beta )$ up to finite group actions that commute with the $(\Gm \times H)$-action, and $\mathfrak C (\beta)$ is a $\Gm$-invariant smooth proper orbifold. Let
$$\mathsf{c} : \widetilde{\mathfrak U} \longrightarrow ( N \times \mathfrak C ( \beta ) ) \longrightarrow \mathfrak C ( \beta )$$
be the composition of the projection map to the zero section and the second projection. Since the fibers of $\mathsf{c}$ are affine spaces modulo the action of finite groups, we have $\mathbb R^{\bullet} \mathsf c _* = \mathsf{c}_*$. We have $\mathtt a \circ \pi_{\beta} = \pi_{\beta} \circ \mathsf{c}$, where $\mathtt a$ is the projection to $\mathrm{pt} = \Spec \, \C$ (the structure map). Therefore, we obtain a spectral sequence
$$H^{j} ( N \times \sZ ( \beta ), \left( \mathbb R^{i} ( \pi_{\beta} )_* \cO_{\widetilde{\mathfrak U}} \right) ) \Rightarrow \mathbb R^{i+j} ( \pi_{\beta} )_* \left( \mathsf c_* \cO_{\widetilde{\mathfrak U}} \right),$$
that degenerates at the $E_1$-stage and whose output is concentrated in degree $0$. Thus, we decompose $\mathsf c_* \cO_{\widetilde{\mathfrak U}}$ with respect to the $\Gm$-degree to obtain a family of coherent sheaves
$$\mathsf c_* \cO_{\widetilde{\mathfrak U}} = \C [N] \otimes \bigoplus _{m \ge 0} \mathcal F_m$$
on $\mathfrak C ( \beta )$ such that
$$\bigoplus_{m \ge 0} \Gamma ( \mathfrak C ( \beta ), \mathcal F_m ) = \C [\sZ ( \beta )].$$
Thus, we can calculate $\C [N \times \sZ ( \beta )]$ as the space of sections of $\cO_{\mathfrak U}$ in two ways, one is over $\widetilde{\mathfrak U}$, and the other is over $\mathfrak C ( \beta )$.

Now we apply the Kawasaki localization theorem \cite{Kaw78} for differential orbifold to the LHS of (\ref{Zpush-num}) by reinterpreting the localization factors on $\widetilde{\mathfrak U}$ as the weighted sum of these of $\mathfrak C ( \beta )$, whose twists are coming from the characters of the fibers of vector bundles obtained from local resolutions of $\{\mathcal F_m\}_m$ and $\C [N]$, to obtain:
\begin{equation}
\sum_{r} \frac{1}{\Delta ( p'_r )} = \sum_{r} \frac{\gch \, \C [N]}{\Delta' ( p'_r )} = \gch \, \C [N] \cdot \gch \, \C [\sZ ( \beta )],\label{sumlocQ}
\end{equation}
where $\{p'_r\}_{r} \subset \mathfrak C (\beta)$ is the set of complete representatives of the connected components of their $(\Gm \times H)$-fixed locus, and $\Delta ( p )$ and $\Delta' ( p )$ denotes the (orbifold) localization factors of $\sX ( \beta )$ and $( N \backslash \widetilde{\mathfrak U} )$ along $p$, respectively.

We compare this calculation with the sum of the collection of terms in the localization calculation of the shift operators $A_i (q)$ for line bundles in \cite[Proposition 2.13 and its proof]{IMT15}. We have
$$\sX ( \beta ) \cong \widetilde{\mathtt{ev}}_1^{-1} (0 \times \sB) \cap \widetilde{\mathtt{ev}}_2^{-1} (\infty \times \sB) \subset \sGB_{2,\beta}.$$
Thus, the localization computations on $\sX ( \beta )$ can be performed on $\sGB_{2,\beta}$ with an extra factor $(1-q)(1-q^{-1})$. This is included in the numerator of the first display formula of \cite[Proof of Proposition 2.13]{IMT15}, and cancels out with the corresponding factor in the denominator. By setting $t=0$ and taking the extra localization with respect to the additional $H$-action, we identify the most LHS of (\ref{sumlocQ}) with the portion of the first display formula \cite[Proof of Proposition 2.10]{IMT15} summing up the localization factors of the ($B^-$-fixed) $\Gm$-attracting fixed points. This yields
\begin{equation}
\sum_{\beta \in Q^{\vee}_+} Q^{\beta} \gch \, \C [\sZ ( \beta )] = \sum_{\beta \in Q^{\vee}_+}Q^{\beta} \sum_{r} \frac{1}{\Delta' ( p'_r )} = \overline{T} ( [\cO_{\sB}] ) = J' ( Q, q^{-1} ),\label{JZ}
\end{equation}
that belongs to $\C ( P, q )[\![Q^{\vee}_+]\!]$ by the expression in the second term, and belongs to $\C P [\![q^{-1}]\!][\![Q^{\vee}_+]\!]$ by the expression in the first term. This yields the first assertion.

We have
\begin{align}\label{BFcalc}
\sum_{\beta \in Q^{\vee}}Q^\beta \chi_q ( \sQ ( \beta ), \cO_{\sQ ( \beta )} ( \la )) & = \sum_{\beta \in Q^{\vee}}Q^\beta \chi_q ( \sX ( \beta ), \pi_\beta^* \cO_{\sQ ( \beta )} ( \la )) \\\nonumber
&= \chi ( \sB, J' ( Qq^\la, q ) \cdot e^{w_0 \la}\cdot J' ( Q, q^{-1} ) )\\\nonumber
& = D_{w_0} ( J' ( Qq^\la, q ) \cdot e^{w_0 \la}\cdot J' ( Q, q^{-1} ) ),
\end{align}
where the first equality is the property of rational resolutions, and the second equality is obtained by the substitution of (\ref{JZ}) into the second display formula \cite[Proposition 2.13]{IMT15}, presented there for the case $\la = - \varpi_i$ and using the original definition (\ref{eqn:defT}) of $T([\cO_{\sB}])$. This implies the second assertion (see also \cite[Lemma 5]{BF14b}).
\end{proof}

For $\vec{n} = (n_1,\ldots,n_r) \in \Z^r_{\ge 0}$, we set $x^{\vec{n}} := x_1^{n_1} \cdots x_r^{n_r}$. For $\la \in P$, we set $\la[\vec{n}] := \la + \sum_{i = 1}^r n_i \varpi_i$.

\begin{thm}\label{loc}
For each $\sum_{\beta \in Q^{\vee}_+, \vec{n} \in \Z^r_{\ge 0}} f_{\beta, \vec{n}} ( q ) x^{\vec{n}} Q^{\beta} \in ( \C P [q^{\pm 1}, x_1,\ldots,x_r] ) [\![Q^{\vee}_+]\!]$ such that
\begin{equation}
\sum_{\beta \in Q^{\vee}_+, \vec{n} \in \Z^r_{\ge 0}} f_{\beta, \vec{n}} ( q ) \otimes_{R(G)}\left( \prod_{j=1}^r ( p_i^{-1} q^{Q_i \partial_{Q_i}})^{- n_i} \right) Q^{\beta} J ( Q, q ) = 0,\label{W-eq-J}
\end{equation}
we have the following equalities:
$$\sum_{\beta \in Q^{\vee}_+, \vec{n} \in \Z^r_{\ge 0}} f_{\beta, \vec{n}} ( q ) q^{- \left< \beta, \la[\vec{n}] \right>} \chi_{q} ( \sQ ( \gamma - \beta ), \cO_{\sQ} ( \la[\vec{n}] ) ) = 0 \hskip 3mm \la \in P_+, \gamma \in Q^{\vee}_+.$$
\end{thm}

\begin{proof}
The assertion is \cite[\S 4.2]{GL03} (see also \cite[Lemma 5]{BF14b} and \cite[\S 5]{BF14c}), that employs the localization theorem applied to a resolution of $\sQ (\beta)$, such as the Laumon spaces (when $G = \mathop{SL} ( n, \C )$) or $\sGB_{0,\beta}$ (cf. \S \ref{subsec:GQL2}).

Here we give an alternative proof (it depends on the argument in the previous paragraph through Theorem \ref{zas}, though). We can substitute $Q$ with $Q q ^{\la}$ in (\ref{W-eq-J}) multiplied with $\prod_{i \in \tI} p_i ^{-m_i}$ for $\la = \sum_{i \in \tI} m_i \varpi_i$, that is identified with $e^{w_0 \la}$ through the isomorphism $K_G ( \sB ) \cong K_B ( \mathrm{pt} ) \cong \C P$. By factoring out the effect of additional powers of $q$ coming from $q^{Q_i \partial_{Q_i}}$'s, we derive a formula
$$\sum_{\beta \in Q^{\vee}_+, \vec{n} \in \Z^r_{\ge 0}} f_{\beta, \vec{n}} ( q )  \otimes _{R ( G )} q^{- \left< \beta, \la[\vec{n}] \right>} Q^{\beta}  J ' ( Qq^{\la[\vec{n}]}, q ) e ^{w_0 ( \la[\vec{n}] )} = 0.$$
Applying Theorem \ref{zas} 2), we conclude the desired equation.
\end{proof}

\subsection{Identification of defining equations}\label{subsec:ide}

The following identity reflects of the fact that $\bigcup_{\beta \in Q^{\vee}_+} \sQ ( \beta )$ is a Zariski dense subset of $\bQ_G$:

\begin{prop}\label{char-limit}
For each $\la \in P$, we have
\begin{equation}
\lim_{\beta \to \infty} \chi_q ( \sQ ( \beta ), \cO_{\sQ ( \beta )} ( \la ) ) = \gch \, H^0 ( \bQ_G, \cO_{\bQ_G} ( \la ) ) \in \Z [\![q^{-1}]\!] P,\label{chi-limit}
\end{equation}
where the limit in the LHS means the coefficient of each $q^{\bullet} e^{\bullet}$ stabilizes. In addition, each coefficient of the LHS is valued in $\Z_{\ge 0}$, and increases monotonically as $\beta$ grows with respect to $\le$ when $\la \in P_+$.
\end{prop}

\begin{proof}
The limit in (\ref{chi-limit}) exists, and it gives the character of the (dual of the) global Weyl module by \cite[\S 4.2]{BF14b} and \cite{BF14c} (here we work with quasi-map spaces, and particularly use Theorem \ref{zas} 2)). This is the same as the character of the RHS of (\ref{H0-BWB}), computed within the theory of semi-infinite flags (\cite{Kat18d}), and hence we conclude the equality. The last part of the assertion follows from \cite[Corollary C]{Kat18d}, that is recorded as Theorem \ref{Qnorm} 3).
\end{proof}

\begin{thm}\label{wJcomp}
We have a unique $\C P$-linear isomorphism
$$\Psi : qK_H ( \sB )_{\lo} \longrightarrow K_H ( \bQ_G^{\mathrm{rat}} )$$
such that:
\begin{itemize}
\item[$(a)$] We have $\Psi ( [\cO_{\sB}] ) = [\cO_{\bQ_G}]$;
\item[$(b)$] For each $i \in \tI$, we have $\Psi \circ a_i = \Xi ( - \varpi_i ) \circ \Psi$;
\item[$(c)$] For each $\beta \in Q^{\vee}$ and $\kappa \in qK_H ( \sB )_{\lo}$, we have $\Psi ( Q^{\beta} \kappa ) = t_{\beta} \Psi ( \kappa )$.
\end{itemize}
In addition, we have $\Psi ( qK_H ( \sB ) ) \subset K_H ( \bQ_G )$.
\end{thm}

\begin{rem}\label{qwJ}
Our proof of Theorem \ref{wJcomp} says that $\Psi$ is actually the $q = 1$ specialization of a dense embedding
$$\Psi_q : \C [q^{\pm 1}] \otimes_\C qK_H ( \sB )_{\lo} \longrightarrow \C Q^{\vee} \otimes _{\C Q^{\vee}_+} \widetilde{K} ( \bQ_G )$$
such that $\Psi_q \circ A_i ( q ) = \Xi_q ( - \varpi_i ) \circ \Psi_q$ ($i \in \tI$), where $\Xi_q ( - \varpi_i )$ denotes the tensor product of $\cO_{\bQ_G^\ra} ( - \varpi_i )$ in the $\Gm$-equivariant setting (see Theorem \ref{denseK}).
\end{rem}

\begin{proof}[Proof of Theorem \ref{wJcomp}]
For each $i \in \tI$, we have
\begin{equation}
A_i ( q )^{-1} \in p_i + \sum_{0 \neq \beta \in Q^{\vee}_+} b_{i,\beta} ( q ) Q^{\beta} \hskip 5mm b_{i,\beta} ( q ) \in \C [q^{-1}] \otimes \mathrm{End}_{\C P} K_H ( \sB )\label{Ainv}
\end{equation}
as an operator acting on $\C[q^{\pm 1}] \otimes K_{H} (\sB)$ (Theorem \ref{i-mod}) extended to the whole $qK_{\Gm\times H} (\sB)$ by (\ref{eqn:commrel}). The linear space $qK_{\Gm \times H} ( \sB )$ contains a subspace
$$\mathfrak K:= ( \C P [q^{-1}, A_1(q)^{-1},\ldots,A_r(q)^{-1}] ) [\![Q^{\vee}_+]\!] \cdot [\cO_{\sB}] \subset qK_{\Gm \times H} ( \sB ).$$
We know that $\C P$ and $\{p_i\}_{i \in \tI}$ generates $K_H ( \sB )$ via the classical product. Hence, a $\C P [q^{-1}]$-linear combination of arbitrary monomials of (\ref{Ainv}) yields
$$[\cO_{\sB ( w )}]Q^{\beta} \in \mathfrak K \hskip 5mm w \in W, \beta \in Q^{\vee}_+,$$
by first making the constant part with respect to $\{Q_i\}_{i =1}^r$, and then removing unnecessary higher degree terms with respect to $\{ Q_i \}_i$ inductively. From this, we find
\begin{equation}
\mathfrak K \cong (\C[q^{- 1}]\otimes K_H(\sB))[\![Q^{\vee}_+]\!]\label{eqn:KpolQ}
\end{equation}
since $A_i(q)^{-1}$ ($i\in \tI$) defines an endomorphism of the RHS by (\ref{Ainv}). We consider a $\C P [q^{- 1}]$-linear functional
$$F^\la_{\beta} : \mathfrak K \longrightarrow \C[q^{\pm 1}]P$$
depending on $\la \in P$ and $\beta \in Q^{\vee}_+$ defined as
\begin{align}\nonumber
\sum_{\beta \in Q^{\vee}_+} F^\la_{\beta} & \, \Bigl( \bigl( \prod_{i=1}^r A_i(q)^{-m_i}\bigr)([\cO_{\sB}]) \Bigr) Q^{\beta} = \sum_{\beta \in Q^{\vee}_+} \chi_q ( \sQ ( \beta ), \cO_{\sQ(\beta)} ( \la + \sum_{i\in\tI}m_i \varpi_i )) Q^{\beta}\\
&= D_{w_0} ( J' ( Qq^{\la + \sum_{i\in\tI}m_i \varpi_i}, q ) \cdot e^{w_0 (\la + \sum_{i\in\tI}m_i \varpi_i)}\cdot J' ( Q, q^{-1} ) ) \label{eqn:functF}
\end{align}
for each $\{m_i\}_{i =1}^r \in \Z_{\ge 0}^{r}$, where we used Theorem \ref{zas} in the second equality. By (\ref{eqn:KpolQ}), the equation (\ref{eqn:functF}) determines $F^\la_{\beta}$ uniquely if we additionally require the skew $\C Q^{\vee}_+$-linearity condition
\begin{equation}
F^\la_{\beta} (Q^{\gamma}  \kappa) = q^{- \left< \gamma,\la \right>}  F^\la_{\beta-\gamma} ( \kappa ) Q^{\gamma} \hskip 5mm \la \in P, \beta,\gamma\in Q^{\vee}_+, \kappa \in \mathfrak K,\label{eqn:QFcomm}
\end{equation}
that arises from (\ref{eqn:commrel}) and the compatibility with the shift operators:
$$F^\la_{\beta} ( A_i^{-1}(q) (\kappa)) = F^{\la+\varpi_i}_{\beta} ( \kappa ) \hskip 5mm \la \in P, \beta \in Q^{\vee}, i \in \tI, \kappa \in \mathfrak K.$$
\begin{claim}\label{claim:Fdef}
For each $\kappa \in \mathfrak K$, we have
$$F^\la ( \kappa ) := \lim_{\beta \to \infty} F^\la_{\beta} ( \kappa ) \in \C P [\![q^{- 1}]\!] \hskip 5mm \la \in P_{++}.$$
\end{claim}
\begin{proof}
We first prove the assertion for
$$( \prod_{i=1}^r A_i(q)^{-m_i} )([\cO_{\sB}]Q^{\gamma}) \in \mathfrak K \hskip 10mm \{m_i\}_{i=1}^r \in \Z^r_{\ge 0} ,\gamma \in Q^{\vee}_+.$$
Since the effect of the shift operators can be absorbed by the choice of $\la \in P_{++}$, it suffices to prove
\begin{equation}
F^\la ( [\cO_{\sB}]Q^{\gamma} ) = \lim_{\beta \to \infty} F^\la_{\beta} ( [\cO_{\sB}] Q^{\gamma} ) \in q^{-\left<\gamma, \la \right>} \C [\![q^{- 1}]\!] P\label{eqn:Fest}
\end{equation}
for each $\la \in P_{++}$ and $\gamma \in Q^{\vee}_+$. By (\ref{eqn:functF}) and (\ref{eqn:QFcomm}), this is equivalent to see
\begin{equation}
\lim_{\beta \to \infty} \chi_q ( \sQ ( \beta ), \cO_{\sQ(\beta)} ( \la )) \in \C [\![q^{-1}]\!] P
\end{equation}
for each $\la \in P_{++}$, that follows from Proposition \ref{char-limit}.

Now we prove the assertion for a general element
$$\kappa = \sum_{\beta \in Q^{\vee}} f _{\beta}([\cO_{\sB}]Q^{\beta}) \in \mathfrak K \hskip 10mm f_{\beta} \in \C P [q^{-1}, A_1(q)^{-1},\ldots,A_r(q)^{-1}],$$
where the sum is understood to be formal. By (\ref{eqn:Fest}), we have
$$F^\la( f_{\beta}([\cO_{\sB}]Q^{\beta}) ) \in q^{-\left<\rho,\beta\right>} \C P [\![q^{-1}]\!] \hskip 10mm \la \in P_{++}.$$
Here the convergence of each coefficient is absolute in Proposition \ref{char-limit} for $\la \in P_{++}$. Hence, we conclude $F^\la(a) \in \C P [\![q^{-1}]\!]$ from the fact that we have finitely many contributions to the coefficient of each power of $q$.
\end{proof}
We return to the proof of Theorem \ref{wJcomp}. Thanks to Claim \ref{claim:Fdef}, we have the following diagram:
\begin{equation}
\xymatrix{
\mathfrak K \ar[d]_{F^{\bullet}} \ar@{-->}[r]&\widetilde{K}'(\bQ_G)\ar[dl]_{\widetilde{\Theta}}\ar[d]^{\Theta}\\
\mathrm{Fun}_P \ar[r]&\mathrm{Fun}_P/\mathrm{Fun}_P^{neg}
},\label{eqn:commFun}
\end{equation}
where we set $F^\bullet ( a ) := [P_{++} \ni \la \mapsto F^\la ( a ) ]$. Assume that we have a $\C P[q^{-1}]$-linear map $\Psi_q'$ that realizes the dashed arrow in (\ref{eqn:commFun}) and makes the diagram (\ref{eqn:commFun}) commutative. Let $\Xi_q ( \la )$ ($\la \in P$) denote the tensor product operation in $\widetilde{K}'(\bQ_G)$ in Theorem \ref{denseK}. By Proposition \ref{char-limit} and (\ref{eqn:functF}) (for the first two equalities), and the comparison of (\ref{eqn:QFcomm}) and Theorem \ref{BWB-KNS} (for the third equality), we have necessarily
\begin{align}\nonumber
\Psi_q'([\cO_{\sB}]) = [\cO_{\bQ_G}], & \hskip 5mm \Psi'_q \circ A_i(q) = \Xi_q ( - \varpi_i ) \circ \Psi'_q \hskip 5mm (i \in \tI), \hskip 5mm\text{and}\\
\Psi_q'( Q^\beta \kappa ) & = t_{\beta} \Psi_q'(\kappa) \hskip 5mm (\kappa \in \mathfrak K,\beta\in Q^{\vee}_+),\label{eqn:propPsiq}
\end{align}
respectively. These properties yield an assignment rule that defines $\Psi_q'$ uniquely by the $\C P [q^{-1}]$-linearity and the skew $\C[\![Q^{\vee}]\!]$-linearity if it is well-defined.

In order to guarantee that $\Psi_q'$ is well-defined, it suffices to write down all relations in $\mathfrak K$ and see that the map $(\widetilde{\Theta} \circ \Psi_q')$ yields relations of $\widetilde{K}' ( \bQ_{G} )$ (i.e. elements of $\mathrm{Fun}_P^{neg}$).

By (\ref{eqn:KpolQ}), we have
$$A_i ( q )^{-1} ( [\cO_{\sB} ( \la )] ) = \sum_{w \in W, \beta \in Q^{\vee}_+} c^{w,\beta}_i ( \la) [\cO_{\sB ( w )}]Q^\beta \hskip 5mm c^{w,\beta}_i ( \la ) \in \C [q^{-1}] P,$$
where we understand the RHS to be a formal sum. (It follows that all the relations of $qK_H( \sB )$ are obtained as the $q=1$ specialization of an expression of $0 \in \mathfrak K$.) We present a general element of $\mathfrak K$ as:
$$f ([\cO_{\sB}]), \hskip 5mm \text{where} \hskip 2mm f = \sum_{\beta \in Q^{\vee}_+, \vec{n} \in \Z^r_{\ge 0}} f_{\beta, \vec{n}} ( q ) x^{\vec{n}} Q^{\beta} \in ( \C P [q^{-1}, x_1,\ldots,x_r] ) [\![Q^{\vee}_+]\!]$$
(where $x_i$ acts by $A_i(q)^{-1}$ for each $i \in \tI$). The equations of the form
$$f ( q, q^{-Q_1 \partial_{Q_1}} p_1, \ldots, q^{-Q_r \partial_{Q_r}}p_r , Q ) J ( Q, q ) = 0,$$
yield all representatives of $0 \in \mathfrak K$ by Theorem \ref{reconst}. By Theorem \ref{loc}, we conclude
\begin{equation}
\sum_{\beta \in Q^{\vee}_+, \vec{n} \in \Z^r_{\ge 0}} f_{\beta, \vec{n}} ( q ) q^{- \left< \beta, \la [\vec{n}] \right>} \chi_q ( \sQ ( \gamma - \beta ), \cO_{\sQ ( \gamma - \beta )} ( \la[\vec{n}] ) ) = 0 \hskip 3mm \la \in P, \gamma \in Q^{\vee}_+.\label{Q-zero}
\end{equation}
The equation (\ref{Q-zero}) is equivalent to
$$F^\la_{\gamma} ( f(q,A_1(q)^{-1},\ldots,A_r(q)^{-1},Q)([\cO_{\sB}]) ) = 0$$
by (\ref{eqn:functF}) and (\ref{eqn:QFcomm}). We have
$$\Psi_q' ( f ([\cO_{\sB}])) = \sum_{\beta \in Q^{\vee}_+, \vec{n} \in \Z^r_{\ge 0}} f_{\beta, \vec{n}} ( q ) [ \cO_{\bQ ( t_{\beta} )} ( \sum_{i = 1}^r n_i \varpi_i ) ]$$
by the required properties (\ref{eqn:propPsiq}) of $\Psi'_q$. Since the convergence in Proposition \ref{char-limit} is coefficientwise and absolute for $\la \in P_{++}$ (and the same reasoning as in the last step of the proof of Claim \ref{claim:Fdef}), the equality (\ref{Q-zero}) implies the numerical identity
$$\sum_{\beta \in Q^{\vee}_+, \vec{n} \in \Z^r_{\ge 0}} f_{\beta, \vec{n}} ( q ) \gch \, H^0 ( \bQ ( t_{\beta} ), \cO_{\bQ ( t_{\beta} )} ( \la[\vec{n}] ) ) = 0 \hskip 3mm \la \in P_{++}.$$
In particular, we derive
$$( \widetilde{\Theta} \circ \Psi'_q ) ( f([\cO_{\sB}])) = \widetilde{\Theta} ( \sum_{\beta \in Q^{\vee}_+, \vec{n} \in \Z^r_{\ge 0}} f_{\beta, \vec{n}} ( q ) [ \cO_{\bQ ( t_{\beta} )} ( \sum_{i = 1}^r n_i \varpi_i ) ] ) \in \mathrm{Fun}_P^{\mathrm{neg}}.$$
Thus, we find that the map $\Psi_q'$ is well-defined.

The specialization $q=1$ is possible on $\C P [q^{-1}, A_1(q)^{-1},\ldots,A_r(q)^{-1}]$ (as in \S \ref{subsec:eqK}). Since the $Q^{\beta}$-coefficient ($\beta \in Q^{\vee}_+$) of an element of $( \C P [q^{-1}, A_1(q)^{-1},\ldots,A_r(q)^{-1}] ) [\![Q^{\vee}_+]\!]$ is a polynomial of the $Q^{\gamma}$-coefficients ($0 \le \gamma \le \beta$) of $A_i (q)^{-1}$ ($i \in \tI$) with $\C P$-coefficients, we deduce that the $q=1$ specialization of $\mathfrak K$ is possible.

In particular, the $q = 1$ specialization of $\Psi_q'$ yields
$$\sum_{\beta \in Q^{\vee}_+, \vec{n} \in \Z^r_{\ge 0}} f_{\beta, \vec{n}} ( 1 ) [ \cO_{\bQ ( t_{\beta} )} ( - \sum_{i = 1}^r n_i \varpi_i ) ] = 0 \in K_H ( \bQ_G ).$$
This induces a unique $\C P$-linear morphism $\Psi' : qK_H ( \sB ) \longrightarrow K_H ( \bQ_G )$ which satisfies the required properties from (\ref{eqn:propPsiq}). Using the right actions of $Q^{\vee}$ on the both sides, we extend $\Psi'$ to
$$\Psi : qK_H ( \sB )_\lo \equiv qK_H ( \sB ) \otimes_{\C Q^{\vee}_+} \C {Q^{\vee}} \longrightarrow K _H ( \bQ_G ) \otimes_{\C Q^{\vee}_+} \C {Q^{\vee}} \rightarrow K_H ( \bQ_G^\ra )$$
by scalar extension with keeping the required properties.

By Proposition \ref{K-gen}, the map $\Psi$ is surjective. It must be injective as the both sides are free modules of rank $|W|$ over the Noetherian rings
$$\C P \otimes ( \C Q^{\vee} \otimes_{\C Q^{\vee}_+} \C [\![Q^{\vee}_+]\!] ) \stackrel{\cong}{\longrightarrow} \C Q^{\vee} \otimes_{\C Q^{\vee}_+} \mathcal R$$
identified through $\Psi'$ (see Lemma \ref{KQ-free}). We deduce $\Psi ( qK_H ( \sB ) ) \subset K_H ( \bQ_G )$ by the property of $\Psi'$.
\end{proof}

\begin{cor}\label{Jcomp}
We have a $\C P$-module isomorphism
$$\Psi : qK _H ( \sB )_\lo \stackrel{\cong}{\longrightarrow} K_{H} ( \bQ_G^{\ra} ),$$
that sends $[\cO_{\sB}]$ to $[\cO_{\bQ_G}]$, quantum product of a line bundle $\cO_{\sB} ( - \varpi_i)$ $(i \in \tI)$ to the tensor product of $\cO_{\bQ_G^{\ra}} ( - \varpi_i)$, and the multiplication by $Q^{\beta}$ to the right $Q^{\vee}$-action of $\beta$ for each $\beta \in Q^{\vee}$.
\end{cor}

\begin{proof}
Combine Theorem \ref{wJcomp} and Theorem \ref{sline} (cf. Theorem \ref{smult}).
\end{proof}

\section{Schubert classes under $\Psi$ and consequences}

Keep the setting of the previous sections. Here we show that the map $\Psi$ introduced in Theorem \ref{wJcomp} respects the classes of Schubert varieties (Theorem \ref{b-trans}). This, together with Theorem \ref{main} and Theorem \ref{wJcomp}, yields Corollary \ref{triangle}, a conjecture by Lam-Li-Mihalcea-Shimozono (Corollary \ref{LLMSc}), and consequently the finiteness of the quantum $K$-group of $\sB$ (Corollary \ref{LLMSfin}). We also offer the finiteness of the shift operators (Corollary \ref{qfin}), that is used in \cite{ACT18}. As we reduce the computations of $K$-theoretic Gromov-Witten correlation functions to the Euler-Poincar\'e characteristic of certain coherent sheaves on quasi-map spaces, we need to guarantee that the higher direct images involved are identically zero. This follows once the singularities appearing in the construction are all rational (Theorem \ref{Q-rat-sing}). However, even with the normality of $\sQ( \beta, w )$ (see \S \ref{subsec:QM} and \S \ref{subsec:GQL2}) established in \cite{Kat18d}, proving rationality of singularities of $\sQ( \beta, w )$ are not straight-forward as the candidates of their resolutions, certain subvarieties of $\sGB_{2,\beta}$ in \S \ref{subsec:GQL} (defined in \S \ref{subsec:GQL2}), are far from well-understood. To overcome this difficulty, we establish a ``decomposition theorem" of the singularities involved (Corollary \ref{tsg}), that makes it possible to explore the structure of singularities by induction (Theorem \ref{tr-rat} and its proof). It is interesting to see that the core of the proof of Corollary \ref{tsg} lies on the factorization property (\S \ref{subsec:fact}), that is tightly connected to a realization of quantum groups (\cite{FFKM}).

\subsection{Graph spaces and quasi-map spaces}\label{subsec:GQL2}

A variety $\mathfrak Y$ is said to have rational singularities if there exists a resolution of singularities $g : \mathfrak Z \rightarrow \mathfrak Y$ that satisfies $g_* \cO_{\mathfrak Z} \cong \cO _{\mathfrak Y}$ and $\R^{>0} g_* \cO_{\mathfrak Z} \cong \{0\}$ (\cite[Theorem 5.10]{KM98}). This is equivalent to assume that all the resolutions of singularities of $\mathfrak Y$ have the same property ([{\it loc. cit.}]).

By the theorem on formal functions (\cite[I\!I \S 11]{Har77}), the rationality of singularities on a normal variety (i.e. the vanishing of the higher direct images of the structure sheaf from a resolution) is detected by its completion. Therefore, in case $\mathfrak Y$ is a normal variety with its (not necessarily closed) point $y$, we say the completion of $\mathfrak Y$ along $y$ has a rational singularity\footnote{As far as the author understands, the study of local rings with rational singularities (in characteristic $0$) have been recognized as a fruitful research topic at least after \cite{Ha98,MS97}. A reader who is unfamiliar to this topic is recommended to consult \cite[Introduction]{MS97}.} if $\mathfrak Y$ has a rational singularity along $y$. Thanks to Boutot's theorem \cite{Bou87}, this is equivalent to requiring that the formal completion of the local ring $\cO_{\mathfrak Y,y}$ has only rational singularity as the formal completion of a local noetherian ring is faithfully flat (\cite[\href{https://stacks.math.columbia.edu/tag/00MC}{Lemma 00MC}]{stacks-project}), and hence is universally injective (\cite[\href{https://stacks.math.columbia.edu/tag/05CK}{Lemma 05CK}]{stacks-project}), that implies purity in the sense of \cite[Definition 1]{Bou87}.

\begin{rem}\label{Rem:Sing}
Theorem \ref{rat-res} and other assertions about the rational resolutions of singularities $f : \mathfrak X \rightarrow \mathfrak Y$ are used only to ensure $f_* \cO_{\mathfrak X} \cong \cO _{\mathfrak Y}$ and $\R^{>0} f_* \cO_{\mathfrak X} \cong \{0\}$ below. To this end, it suffices to assume that $f$ is a birational projective morphism and $\mathfrak X$ has only rational singularities (or quotient singularities by \cite[Proposition 5.15]{KM98}) as we can replace $\mathfrak X$ with its resolution of singularities.
\end{rem}

We have a morphism $\pi_{n, \beta} : \sGB_{n, \beta} \rightarrow \sQ ( \beta )$ ($n \in \Z_{\ge 0}, \beta \in Q^{\vee}_+$) that factors through $\sGB_{0, \beta}$ (Givental's main lemma \cite{Giv96}; see \cite[\S 8]{FFKM} and \cite[\S 1.3]{FP95}). We define
\begin{equation}
\cO_{\sGB_{n, \beta}} ( \la ):= \pi_{n, \beta}^* \cO_{\sQ ( \beta )} ( \la ) \hskip 5mm \la \in P.\label{defOGB}
\end{equation}

\begin{thm}[Braverman-Finkelberg \cite{BF14a,BF14b,BF14c}]\label{rat-res}
The morphism $\pi_{0,\beta}$ is a rational resolution of singularities. \hfill $\Box$
\end{thm}

We note that $\sGB_{n,\beta}$ is irreducible (\cite{Tho98,KP01}).

Let $\sX ( \beta )$ denote the subvariety of $\sGB_{2,\beta}$ consisting of the stable maps whose first marked point projects to $0 \in \P^1$, and whose second marked point projects to $\infty \in \P^1$ through the projection of a genus zero domain curve $C$ to the main component $C_0 \cong \P^1$. Let us denote the restriction of $\mathtt{ev}_i$ $(i=1,2)$ to $\sX ( \beta )$ by the same letter. By Theorem \ref{rat-res}, $\sX ( \beta )$ also gives a resolution of singularities $\pi_{\beta} : \sX ( \beta ) \rightarrow \sQ ( \beta )$. The following is a result of Buch-Chaput-Mihalcea-Perrin \cite{BCMP}:

\begin{thm}[\cite{BCMP} Corollary 3.8, cf. \cite{Kat18d} Theorem 5.1]\label{X-can}
The variety
$$\mathtt{ev}_1^{-1} ( \sB ( w ) ) \cap \mathtt{ev}_2^{-1} ( \sB^\op ( v ) ) \subset \sX ( \beta )$$
is irreducible, normal, and has rational singularities $($that we denote by $\sX (\beta, w, v))$ for each $w, v \in W$. \hfill $\Box$
\end{thm}

We remark that $\sX ( \beta ) = \sX (\beta, e, w_0 )$. We set $\sX ( \beta,w ) := \sX ( \beta,w,w_0 )$ and $\sQ ( \beta, w, v ) := \pi_{\beta} ( \sX (\beta, w, v) ) \subset \sQ ( \beta, w )$. Then, the map $\pi_{\beta}$ restricts to a $( \Gm \times H)$-equivariant birational proper map
$$\pi_{\beta, w, v} : \sX (\beta, w, v) \to \sQ ( \beta, w, v )$$
by \cite[\S 5.2]{Kat18d}. We denote $\pi_{\beta, w,w_0}$ by $\pi_{\beta, w}$ for simplicity. Let $\cO _{\sX ( \beta, w, v )} ( \la )$ denote the restriction of $\cO _{\sGB_{2,\beta}} ( \la )$ to $\sX ( \beta, w, v )$ for each $\la \in P$.

The space $\sQ ( \beta, w, v )$ is called the Richardson variety of $\bQ_G^\ra$ in \cite{Kat18d} (see also Remark \ref{oppSchubert}). In particular, we have
\begin{equation}
\sQ ( \beta, w, v ) = \overline{\bI p_w \cap \bI^- p_{v t_{\beta}}} = \bQ_G ( w ) \cap \bQ_G^- (v t_{\beta}) \subset \prod_{i \in \tI} \P ( L ( \varpi _i )^* \otimes \C (\!(z)\!) )\label{si-Ric}
\end{equation}
by \cite[\S 4.1]{Kat18d} (that defines the third term) and \cite[Proposition 5.3]{Kat18d} (that identifies the third term with $\sQ ( \beta, w, v )$). We have
\begin{equation}
\dim \, \sQ ( \beta, w, v ) = 2 \left< \beta, \rho \right> - \ell ( w ) + \ell ( v )\label{dimQ}
\end{equation}
and $\sQ ( \beta, w, v ) \neq \emptyset$ if and only if $w \ge_\si v t_{\beta}$ by \cite[Corollary 5.4]{Kat18d}.

Let us gather several results on $\sQ ( \beta, w, v )$ from various places in \cite{Kat18d}:

\begin{thm}[\cite{Kat18d}]\label{Qnorm}
For each $w, v \in W$ and $\beta \in Q^{\vee}_+$, we have:
\begin{enumerate}
\item {\rm (Theorem 5.20)} The variety $\sQ (\beta, w, v)$ is normal;
\item {\rm (Theorem 4.33)} For each $\la \in P_+$, we have
$$H^{>0} ( \sQ ( \beta,w,v ), \cO_{\sQ ( \beta,w,v )} ( \la ) ) = \{0\};$$
\item {\rm (Theorem 4.33)} For $\beta' \in Q^{\vee}_+$ such that $\beta < \beta'$ and $\la \in P_{+}$, the natural restriction map
$$H^{0} ( \sQ ( \beta',w ), \cO_{\sQ ( \beta',w )} ( \la ) ) \longrightarrow H^{0} ( \sQ ( \beta,w ), \cO_{\sQ ( \beta,w )} ( \la ) )$$
is surjective;
\item {\rm (Proposition 4.39 and \S 5.2)} Let $i \in \tI$ be such that $s_i w < w$ and $s_i v < v$. Then, the variety $\sQ ( \beta, w, v )$ is $B_i$-stable, and we have an inflation map $\pi_i : \mathop{SL} ( 2, i ) \times^{B_i} \sQ ( \beta, w, v ) \to \sQ ( \beta, s_i w, v )$. We have
$$\mathbb R^{\bullet} ( \pi_i )_* \cO_{\mathop{SL} ( 2, i ) \times^{B_i} \sQ ( \beta, w, v )} \cong \cO_{\sQ ( \beta, s_i w, v )};$$
\item {\rm (Proposition 4.39 and Lemma 4.6)} Assume that $s_{\vartheta} w > w$ and $s_{\vartheta} v > v$. Then, the variety $\sQ ( \beta, w,v )$ admits a $B_0$-action and we have an inflation map
$$\pi_0 : \mathop{SL} ( 2, 0 ) \times^{B_0} \sQ ( \beta, w, v ) \to \sQ ( \beta - w^{-1} \vartheta^{\vee}, s_\vartheta w, v ).$$
We have $\mathbb R^{\bullet} ( \pi_0 )_* \cO_{\mathop{SL} ( 2, 0 ) \times^{B_0} \sQ ( \beta, w, v )} \cong \cO_{\sQ ( \beta - w^{-1} \vartheta^{\vee}, s_\vartheta w, v )}$.
\end{enumerate}
\end{thm}

For each $\beta \in Q^{\vee}_+, w, v \in W$, we consider the open subset $\mathring{\sQ} ( \beta, w, v ) \subset \sQ ( \beta, w, v )$ consisting of quasi-maps defined at $0 \in \P^1$ (i.e. have no defect at $0$) and their values at $0$ belong to $\bO_\sB ( w )$. The variety $\sQ ( \beta, w, v )$ is decomposed as
$$\sQ ( \beta, w, v )= \bigsqcup_{0 \le \beta' \le \beta} \bigsqcup_{\scriptsize \begin{matrix}w' \in W \\ w \ge_\si w't_{\beta'} \end{matrix}} \mathring{\sQ} ( \beta - \beta', w', v ),$$
where the inclusion map is given by the restriction of the map $\imath_{\beta'}$ to $\sQ ( \beta -\beta', w', v ) \subset \sQ ( \beta - \beta' )$, that lands on $\sQ (\beta, w, v )$ by (\ref{si-Ric}). We have $\mathring{\sQ} ( \beta, w )$ = $\mathring{\sQ} ( \beta, w, w_0 )$, and $\mathring{\sQ} ( \beta ) = G \mathring{\sQ} ( \beta, w )$ for every $w \in W$.

\subsection{Formal neighborhoods of Bruhat cells}\label{subsec:fnb}

For two (affine) schemes $\mathfrak X$ and $\mathfrak Y$ such that $\mathfrak Y$ is the spectrum of a complete local ring (and hence has a unique closed point $0 \in \mathfrak Y$), we denote by $\mathfrak X \, \widehat{\times} \, \mathfrak Y$ the formal completion of $\mathfrak X \times \mathfrak Y$ along the point $\mathfrak X \times 0$.

This is a particular instance of the spectrums of (admissible) formal completions for general commutative rings explained in \cite[Chap 0 \S 7]{EGAI} (see also \cite[\S 10]{EGAI}). Here localization commutes with taking the formal completions (\cite[Chap 0 Corollaire 7.6.3]{EGAI}). It is clear from the definition that the ring quotients commute with the formal completions. By \cite[Chap 0 (7.6.11)]{EGAI}, we can glue formal completions of an affine open covering $\{\mathfrak U_i\}_i$ of a scheme $\mathfrak X$ with respect to their closed subschemes $\{\mathfrak Z_i\}_i$ such that $\mathfrak Z_i = ( \mathfrak Z \cap \mathfrak U_i)$ for a closed subscheme $\mathfrak Z \subset \mathfrak X$ to obtain the formal completion of $\mathfrak X$ along $\mathfrak Z$. It follows that if $\mathfrak Z$ is irreducible and each $\mathfrak U_i$ defines a trivial algebraic fiber bundle over $\mathfrak Z_i$, then the resulting formal completion admits the structure of a locally trivial algebraic fiber bundle. In particular, we can glue a collection of schemes $\{\mathfrak X_i \widehat{\times} \mathfrak Y\}_i$ along $\{\mathfrak X_i\}_i$ to obtain $\mathfrak X \widehat{\times} \mathfrak Y$ for $\mathfrak X = \bigcup_i \mathfrak X_i$.

\begin{lem}\label{alg-fact}
Let $\beta , \gamma \in Q^{\vee}_+$. For each $y = ( f, D ) \in \mathring{\sQ} ( \beta )$, the complex analytic map in Proposition \ref{ltr} induces a morphism
$$\eta^{\mathrm{alg}} : \mathcal S_y^{\wedge} \times U \longrightarrow \sQ ( \beta + \gamma ),$$
where $\mathcal S_y^{\wedge}$ is the formal completion of the complex analytic open subset $\mathbf 0 \in \mathcal S_y \subset \sZ ( \gamma )$, $y \in U \subset \mathring{\sQ} ( \beta )$ is an algebraic open subset, and $\imath_{\gamma} ( y ) = \eta ( \mathbf 0, y )$. In addition, the map $\eta^{\mathrm{alg}}$ satisfies the following diagram:
\begin{equation*}
	\xymatrix{
\mathcal S_y^{\wedge} \times U \ar[r] \ar@{^{(}->}[d]_{\mathrm{id} \times \imath}& \sQ ( \beta + \gamma )\ar@{^{(}->}[d]_{\imath}\\
\mathcal S_y^{\wedge} \times U ( \delta ) \ar[r] & \sQ ( \beta + \gamma + \delta )
	},
	\end{equation*}
where $\imath$ is the map adding the defect $\delta [p]$ such that $0 \neq p \not\in [D]$, and $U (\delta) \subset \mathring{\sQ} ( \beta + \delta )$ is an open subset such that $U ( \delta ) \cap \mathring{\sQ} ( \beta ) = U$ for each $\delta \in Q^{\vee}_+$.
\end{lem}

\begin{proof}
The passage from Theorem \ref{fact} to Remark \ref{rem:loc} (\ref{loc-sch}) applies to (the proofs of) Proposition \ref{ltr} and Corollary \ref{ltr-comp}.
\end{proof}

\begin{lem}\label{assgr}
Let $\beta \in Q^{\vee}_+$. Let $\C [\sZ (\beta)]_{\le -m}$ denote the $\Gm$-degree $(\le -m)$-parts of $\C [\sZ (\beta)]$ for each $m \in \Z_{\ge 0}$. We have an isomorphism of rings
$$\C [\mathscr Z ( \beta )]_{\mathbf 0}^{\wedge} \cong \prod_{m \ge 0}\left( \frac{\C [\sZ (\beta)]_{\le -m}}{\C [\sZ (\beta)]_{\le -(m+1)}} \right),$$
where the multiplication of the LHS is that of the completion, and the multiplication of the RHS is that of the associated graded of a decreasing filtration $\{ \C [\sZ (\beta)]_{\le - m} \}_m$. This isomorphism is also $B$-equivariant.
\end{lem}

\begin{proof}
The ring $\C [\sZ (\beta)]$ is $\Gm$-stable, its grading is concentrated in $\le 0$, and the degree $0$-part is $\C$. We consider the ($\Gm \times H$-stable) ideals $\C [\sZ (\beta)]_{\le -m} \subset \C [\sZ (\beta)]$ consisting of functions of degree $\le -m$ for each $m \in \Z_{\ge 0}$. Let $\mathfrak m_\mathbf 0 \subset \C [\sZ (\beta)]$ be the maximal ideal corresponding to $\mathbf 0$. Since $\C [\sZ ( \beta )]$ is of finite type, we have
$$\mathfrak m_{\mathbf 0} = \C [\sZ (\beta)]_{\le -1}, \hskip 5mm \text{and} \hskip 5mm \C [\sZ (\beta)]_{\le -Nj} \subset ( \mathfrak m_{\mathbf 0} )^j \subset \C [\sZ (\beta)]_{\le -j} \hskip 5mm j \in \Z_{\ge 0},$$
where $N$ is the largest $\Gm$-degree among the homogeneous generators of $\C [\sZ ( \beta )]$. Thus, two sequences of ideals
$$\mathfrak m_{\mathbf 0}^j \hskip 5mm \text{and} \hskip 5mm \C [\sZ (\beta)]_{\le -j} \hskip 5mm j \in \Z_{\ge 0}$$
induce equivalent linear topologies on $\C [\sZ (\beta)]$. In particular, we have an isomorphism of rings (\cite[Chap. 0, Proposition 7.1.4]{EGAI})
$$\C [\sZ (\beta)]^{\wedge}_{\mathbf 0} \cong \prod_{m \ge 0} \left( \frac{\C [\sZ (\beta)]_{\le -m}}{\C [\sZ (\beta)]_{\le -(m+1)}} \right),$$
where the multiplication of the LHS is that of the completion, and the multiplication of the RHS is that of the associated graded of a decreasing filtration $\{ \C [\sZ (\beta)]_{\le -m} \}_m$. This identification is also $B$-equivariant since the $\Gm$-action commutes with the $B$-action.
\end{proof}

\begin{prop}\label{f-nbd}
Let $\beta \in Q^{\vee}_+$. The formal completion $\mathfrak N$ of $\bQ _G$ along $\bO ( t_{\beta} )$ is isomorphic to
\begin{equation}
\mathfrak N \cong \bO ( t_{\beta} )\,\widehat{\times} \, \Spec \,  \C [\sZ (\beta)]^{\wedge}_{\mathbf 0}.\label{tube-prod}
\end{equation}
\end{prop}

\begin{proof}
We have an embedding $\bQ _G ( t_{\beta} ) \subset \bQ_G$ as schemes of infinite type, equipped with the $\mathbf I$-action. In particular, we have a sheaf $\mathcal I$ of ideals of $\bQ_G ( t_{\beta} )$ in $\cO_{\bQ _G}$. The completion of $\cO_{\bQ _G}$ with respect to $\mathcal I$ yields the infinitesimal neighborhood $\mathfrak N^+$ of $\bQ _G ( t_{\beta} )$ in $\bQ_G$, that restricts to the infinitesimal neighborhood $\mathfrak N$ of $\bO ( t_{\beta } )$. Since $\bO(t_{\beta})$ is an affine scheme, $\mathfrak N$ is also affine. Let us employ sections $\{ \phi_{u,i} \}_{u,i}$ from \S \ref{subsec:sif} and impose additional relations $\phi_{u,i} / \phi_{t_{\beta},i}= 0$ for each $(u,i) \in S^-(w_0 t_{\beta})$ on
$$\C [\mathfrak N] \supset \sum_{i \in \tI} \Gamma ( \bQ_G, \cO ( \varpi_i ) ) \phi_{t_{\beta},i}^{-1}  \longrightarrow \sum_{i \in \tI} \Gamma ( \bQ_G(t_{\beta}), \cO ( \varpi_i ) ) \phi_{t_{\beta},i}^{-1} \subset \C [\bO (t_{\beta})],$$
and consider their reduced quotients. By Lemma \ref{phi-def}, these equations cut out $( \sQ ( \beta ) \cap \mathfrak N )$ set-theoretically. This construction factors through the formal completion of $\sQ (\beta + \gamma)$ along $\imath_{\beta} ( \mathring{\sQ} ( \gamma ) )$ for an arbitrary $\gamma \in Q^{\vee}_+$ in view of (\ref{si-Ric}). By Lemma \ref{alg-fact}, the completion of $\sQ (\beta + \gamma)$ along $\imath_{\beta} ( \mathring{\sQ} ( \gamma ) )$ yields the product of the spectrum of the formal completion of $\C [\sZ (\beta) ]$ along the origin $\mathbf 0$ and an open neighborhood of the $(\Gm \times H)$-fixed point $p_{t_\beta}\in \bQ_G$ in $\imath_{\beta} (\mathring{\sQ} ( \gamma )) \subset \bQ_G$. Thus, we conclude an isomorphism
$$\C [\mathfrak N] \otimes_{\C [\bO ( t_{\beta})]} \frac{\C [\bO ( t_{\beta})]}{\mathfrak m_{t_{\beta}}} = \C [\mathfrak N] \otimes_{\C [\bO ( t_{\beta})]} \C_{p_{t_{\beta}}} \cong \C [\sZ (\beta) ]_{\mathbf 0}^{\wedge},$$
where $\mathfrak m_{t_{\beta}} \subset \C [\bO ( t_{\beta})]$ denote the maximal ideal corresponding to $p_{t_{\beta}}$.

By Lemma \ref{assgr}, we have
$$\C [\sZ (\beta)]^{\wedge}_{\mathbf 0} \cong \prod_{m \ge 0} \left( \frac{\C [\sZ (\beta)]_{\ge m}}{\C [\sZ (\beta)]_{\ge (m+1)}} \right),$$
where $\C [\sZ (\beta)]_{\ge m}$ is $\C [\sZ (\beta)]_{\le -m}$ in Lemma \ref{assgr} since the $\Gm$-grading is opposite by convention. This identification is also $B^-$-equivariant (instead of $B$, due to the choice of $p_{t_{\beta}}$ as the origin).

Here, $\bO ( t_{\beta} )$ admits a free homogeneous action by a pro-unipotent subgroup of $\mathbf I$ (namely $H[\![z]\!]_1N[\![z]\!]$, where $H[\![z]\!]_1 = \ker \, (H[\![z]\!] \to H )$, see e.g. \cite[\S 4.2]{Kat18d}). The scheme $\mathfrak N$ also admits a $\Gm \ltimes \bI$-action. Thus, we use the $B[\![z]\!]$-action to construct a map
\begin{equation}
(\Gm \ltimes B[\![z]\!] ) \times^{( \Gm \times H )} \Spec \, \C [\sZ (\beta)]^{\wedge}_{\mathbf 0} \longrightarrow \mathfrak N\label{inf-action}
\end{equation}
such that the zero section maps isomorphically to $\bO (t_{\beta})$. In the LHS of (\ref{inf-action}), the formal completion and the (suitable) associated graded are still isomorphic as the effect of $(\Gm \ltimes B[\![z]\!] ) \times^{( \Gm \times H )} \bullet$ to the coordinate ring is just to take the (completed) tensor product with $\C [H[\![z]\!]_1 N[\![z]\!]]$.

We take a $(\Gm \times H)$-stable ambient space
$$\mathbf 0 \in \sZ (\beta ) \subset \mathscr S := \prod_{i \in \tI} \bigl( \phi_{t_\beta,i} + L ( \varpi_i )^* \otimes \C [z]_{< \left< \beta, \varpi_i \right>}\bigr) \subset \prod_{i \in \tI} \P (L(\varpi_i)^* \otimes \C [\![z]\!])$$
of $\sZ (\beta )$. The map (\ref{inf-action}) can be also obtained as the formal completion of the map
\begin{equation}
(\Gm \ltimes B[\![z]\!] ) \times^{( \Gm \times H )} \mathscr S \longrightarrow \prod_{i \in \tI} \P (L(\varpi_i)^* \otimes \C [\![z]\!])\label{real-action}
\end{equation}
along $\bO (t_{\beta})$ (and its preimage) and then restrict to $\sZ (\beta ) \subset \mathscr S$ in the fiber direction. Here we warn that (\ref{real-action}) {\it cannot} be injective. Nevertheless, Theorem \ref{si-Bruhat} asserts that the image $(\Gm \ltimes B[\![z]\!] ) \times^{(\Gm \times H)} \sZ(\beta)$ under (\ref{real-action}) contains a dense open subset of $\bO (e) \subset \bQ_G$ since
$$\mathscr S \cap \mathbb O ( e ) \subset \sZ ( \beta )$$
is Zariski open dense (and we apply the $B[\![z]\!]$-action). Thus, we conclude that the induced map
\begin{equation}
\imath : \C [\mathfrak N] \longrightarrow \C [\bO (t_{\beta})] \,  \widehat{\otimes} \, \C [\sZ (\beta)]^{\wedge}_{\mathbf 0}\label{i-incl-N}
\end{equation}
is injective. The common quotient maps
$$\C [\bO (t_{\beta})] \,  \widehat{\otimes} \, \C [\sZ (\beta)]^{\wedge}_{\mathbf 0} \longrightarrow \C [\bO (t_{\beta})] \longleftarrow \C [\mathfrak N]$$
are $(\Gm \ltimes B[\![z]\!])$-equivariant, and hence $\imath$ induces a map
$$I_{\texttt{ev}} := \ker ( \C [\mathfrak N] \to \C [\bO (t_{\beta})] ) \longrightarrow \C [\bO (t_{\beta})] \,  \widehat{\otimes} \prod_{m \ge 1} \left( \frac{\C [\sZ (\beta)]_{\ge m}}{\C [\sZ (\beta)]_{\ge (m+1)}} \right)$$
as $\C [\bO(t_{\beta})]$-modules. Since $I_{\texttt{ev}}$ is the ideal of definition of the topology of $\C [\mathfrak N]$ (in the sense of \cite[Chap. 0 \S 7]{EGAI}), we find that the map $\imath$ is continuous if $\imath$ is surjective, and hence is an isomorphism. The scheme $\mathfrak N$ admits $(\Gm \ltimes B[\![z]\!])$-action and its restriction to $\mathscr S$ admits the $(\Gm \times H)$-action. Hence, we obtain a $(\Gm \ltimes B[\![z]\!])$-equivariant $\C [\bO(t_{\beta})]$-algebra map
\begin{equation}
\mathrm{gr} \, \C [\mathfrak N]:= \prod_{m \ge 0} \frac{I_{\texttt{ev}}^m}{I_{\texttt{ev}}^{m+1}} \longrightarrow \C [\bO (t_{\beta})] \,  \widehat{\otimes} \prod_{m \ge 0} \left( \frac{\C [\sZ (\beta)]_{\ge m}}{\C [\sZ (\beta)]_{\ge (m+1)}} \right).\label{eqn:grNS}
\end{equation}
The map $\imath$ is a surjection if and only if (\ref{eqn:grNS}) is a surjection since the latter implies $\imath$ is continuous and $\mathrm{Im} \, \imath$ is dense (note that $\C [\mathfrak N]$ is complete by construction). If we specialize (\ref{eqn:grNS}) to $p_{t_\beta} \in \bO ( t_{\beta} )$, then the discussion in the first paragraph turns (\ref{eqn:grNS}) into a surjection. The same is true for every $B[\![z]\!]$-translation of $p_{t_{\beta}}$. It follows that (\ref{i-incl-N}) becomes surjective if we specialize to a (closed) point of $\bO(t_{\beta})$.

To prove (\ref{tube-prod}), we need to globalize the isomorphism (\ref{i-incl-N}) obtained only after specializing to a (closed) point of $\bO(t_{\beta})$ to an isomorphism over the whole of $\bO(t_{\beta})$.

For each $\gamma \in Q^{\vee}_+$, the formal completion of $\sQ (\beta + \gamma)$ along $\imath_{\beta} ( \mathring{\sQ} ( \gamma ) )$ yields a closed subscheme $\mathfrak Q _{\gamma, \beta} \subset \mathfrak N$. We have the commutative diagram
$$\xymatrix{
\C [\mathfrak N] \ar@{^{(}->}[r] \ar@{->>}[d] & \C [\bO (t_{\beta})] \,  \widehat{\otimes} \, \C [\sZ (\beta)]^{\wedge}_{\mathbf 0} \ar@{->>}[d]\\
\C [\mathfrak Q_{\beta,\gamma}] \ar[r]^{h_{\gamma} \hskip 10mm}& \C [\mathring{\sQ} ( \gamma )]\widehat{\otimes} \, \C [\sZ (\beta)]^{\wedge}_{\mathbf 0}
},$$
where $h_{\gamma}$ is the map obtained by restricting the pullback under (\ref{real-action}). In view of the previous paragraph, the map $h_{\gamma}$ is an isomorphism when we specialize to each closed point of $\mathring{\sQ} ( \gamma )$ (note that $\mathfrak Q_{\beta,\gamma}$ is a locally trivial algebraic fiber bundle over $\mathring{\sQ} ( \gamma )$ whose fiber is $\Spec \,  \C [\sZ (\beta)]^{\wedge}_{\mathbf 0}$). If we consider a finite-dimensional quotient of $\C [\sZ (\beta)]^{\wedge}_{\mathbf 0}$, then it defines a morphism of finite-rank locally free sheaves on $\mathring{\sQ} ( \gamma )$ which is an isomorphism of fibers over each closed point of $\mathring{\sQ} ( \gamma )$. In this setting, $h_{\gamma}$ induces an isomorphism of locally free sheaves by Nakayama's lemma. Therefore, the morphism $h_{\gamma}$ is itself an isomorphism.

In particular, we find a commutative diagram
$$\xymatrix{
\C [\mathfrak N] \ar@{^{(}->}[rr] \ar@{->>}[rd]_{q_\gamma}&& \C [\bO (t_{\beta})] \,  \widehat{\otimes} \, \C [\sZ (\beta)]^{\wedge}_{\mathbf 0} \ar@{->>}[ld]^{r_\gamma}\\
& \C [\mathfrak Q_{\beta,\gamma}] &
}.$$
By taking the projective limit (cf. Lemma \ref{alg-fact}), we deduce that both of
\begin{equation}
\C [\mathfrak N] \subset \varprojlim_{\gamma} \C [\mathfrak Q_{\beta,\gamma}] \supset \C [\bO (t_{\beta})] \,  \widehat{\otimes} \, \C [\sZ (\beta)]^{\wedge}_{\mathbf 0}\label{eqn:densecomp2}
\end{equation}
are dense subsets with respect to the topology arising from the inverse limit, in which we endow $\C [\mathfrak Q_{\beta,0}]$ ($\cong \C [\mathscr Z(\beta)]_{\mathbf 0}^{\wedge}$) with the discrete topology. This induces a short exact sequence of projective systems
$$0 \rightarrow \{ \mathrm{Im} \, q_{\gamma} \}_{\gamma} \rightarrow \{ \mathrm{Im} \, r_{\gamma} \}_{\gamma} \rightarrow \{ \mathrm{Im} \, r_{\gamma} / \mathrm{Im} \, q_\gamma \}_{\gamma} \rightarrow 0$$
satisfying \cite[Chap. 0 \S 13.1 (ML)]{EGAIII-1}. Here \cite[Chap. 0 Proposition 13.2.2]{EGAIII-1} asserts
$$\varprojlim_{\gamma} \frac{\mathrm{Im} \, r_{\gamma}}{\mathrm{Im} \, q_\gamma} \cong \frac{\varprojlim_{\gamma} \mathrm{Im} \, r_{\gamma}}{\varprojlim_{\gamma} \mathrm{Im} \, q_{\gamma}} = \frac{\varprojlim_{\gamma} \C [\mathfrak Q_{\beta,\gamma}]}{\varprojlim_{\gamma} \C [\mathfrak Q_{\beta,\gamma}]} = 0.$$
From this, we find that
$$\frac{\C [\bO (t_{\beta})] \,  \widehat{\otimes} \, \C [\sZ (\beta)]^{\wedge}_{\mathbf 0}}{\C [\mathfrak N]} \hookrightarrow \varprojlim_{\gamma} \frac{\mathrm{Im} \, r_{\gamma}}{\mathrm{Im} \, q_\gamma} = 0.$$
Therefore, we conclude that $\imath$ is an isomorphism as required.
\end{proof}

\begin{cor}\label{tsg}
Let $w,w',v \in W$ and $\beta,\beta' \in Q^{\vee}_+$, such that
$$w \ge_\si w' t_{\beta'} \ge _\si v t_{\beta}.$$
Then, the formal completion of $\bQ_G (w )$ along $\bO ( w' t_{\beta'} )$ defines a locally trivial fiber bundle $\mathfrak N_{\beta', w,w'}$ over $\bO (w't_{\beta'})$ whose fiber $\mathfrak X_{\beta', w,w'}$ is the spectrum of a normal Noetherian ring completed at the maximal ideal. In addition, the restriction $\mathfrak Q_{\beta',w,w'} ( \beta, v )$ of the fiber bundle $\mathfrak N_{\beta',w,w'}$ to
$$\mathring{\sQ} ( \beta - \beta', w', v ) \stackrel{\imath_{\beta'}}{\longrightarrow} ( \sQ ( \beta ) \cap \bO ( w' t_{\beta'}  ) ) \subset \bO ( w't_{\beta'})$$
induces the following diagram, in which both squares are Cartesian:
\begin{equation}
\xymatrix{
\bQ_G ( w ) & \ar[l] \mathfrak N_{\beta',w,w'} \ar[r] & \bO ( w' t_{\beta'} )\\
\sQ (\beta, w, v ) \ar[u] & \ar[l] \mathfrak Q_{\beta',w,w'} ( \beta, v ) \ar[r]  \ar[u] & \mathring{\sQ} ( \beta - \beta', w', v ) \ar[u] . 
}.\label{NQX}
\end{equation}
In particular, the formal completion of $\sQ ( \beta, w, v )$ along $\mathring{\sQ} ( \beta - \beta', w', v )$ defines a locally trivial fiber bundle over $\mathring{\sQ} ( \beta - \beta', w', v )$, and the structure of its fiber $\mathfrak X_{\beta', w,w'}$ does not depend on the choice of $v$ and $\beta$.
\end{cor}

\begin{proof}
Consider the formal completion $\mathfrak Q$ of $\sQ ( \beta, e, v )$ along $\mathring{\sQ} ( \beta - \beta', e, v )$. We have a map $\eta : \mathfrak Q \hookrightarrow \mathfrak N$ by the natural inclusions $\mathring{\sQ} ( \beta - \beta', e, v ) \subset \bO (t_{\beta'})$ and $\sQ ( \beta, e, v ) \subset \bQ_G$ induced from (\ref{si-Ric}). We have set-theoretic defining equations of $\sQ ( \beta, e, v ) \subset \bQ_G$ near $\bO ( t_{\beta'} )$ consisting of $\phi_{u,i} / \phi_{t_{\beta'},i}$ for each $(u,i) \in S^-(vt_{\beta})$ by Lemma \ref{phi-def} and (\ref{si-Ric}). Since they are obtained as the (uniquely determined) preimages of the middle map
\[
\C [\mathfrak N] \supset \Gamma (\bQ_G, \cO_{\bQ_G} (\varpi_i)) \longrightarrow \!\!\!\!\! \rightarrow \Gamma (\bQ_G (t_{\beta'}), \cO_{\bQ_G (t_{\beta'})} (\varpi_i)) \subset \C [\bO (t_{\beta'})],
\]
we can regard them as functions on $\C [\bO (t_{\beta'})] \subset \C [\mathfrak N]$. Thus, the locally trivial fiber bundle structure of $\mathfrak N$ in Proposition \ref{f-nbd} restricts to a locally trivial algebraic fiber bundle structure on $\mathfrak Q$ along $\mathring{\sQ} ( \beta - \beta', e, v )$ whose fiber is isomorphic to $\Spec \, \C [\sZ (\beta')]^{\wedge}_{\mathbf 0}$ (as $\mathfrak Q$ is reduced, each fiber must be reduced). In other words, we can understand this fiber bundle as the restriction of (\ref{tube-prod}) from $\bO (t_{\beta'})$ to $\mathring{\sQ} ( \beta - \beta', e, v )$ by intersecting the base with the opposite Schubert variety corresponding to $vt_{\beta}$.

For each $u \in W$, the orbit $\bO ( u )$ is the preimage of $\bO _{\sB} ( u )$ under the (uncountable dimensional) affine fibration $G \bO ( e ) \rightarrow \sB$ obtained by setting $z = 0$. As $\bO _{\sB} ( u ) \subset \dot{u} \bO _{\sB} ( e )$, we deduce that $\bO ( u ) \subset \dot{u} \bO ( e )$.  In addition, the normal direction of $\bO ( u ) \subset \dot{u} \bO ( e )$ is given by the free action of $( N^- \cap \dot{u} N \dot{u}^{-1} )$. The same is true if we multiply $u, e$ in $\bO (\bullet)$ by $t_{\beta'}$ (cf. Theorem \ref{si-Bruhat}). Let $\mathfrak N_{w'}$ denote the formal completion of $\bQ_G$ along $\bO (w' t_{\beta'})$. In view of Proposition \ref{f-nbd}, we deduce
\begin{equation}
\mathfrak N_{w'} \cong \bO ( w' t_{\beta'} )\, \widehat{\times}\, \mathfrak D',\label{f-nbd-w}
\end{equation}
where $\mathfrak D'$ is the completion of $\Spec \, \C [\sZ (\beta')]^{\wedge}_{\mathbf 0} \times ( N^- \cap \dot{u} N \dot{u}^{-1} )$ at $(\mathbf 0, 1 )$. If we further replace $\bQ_G = \bQ_G ( e )$ with $\bQ_G ( w )$ for $w \in W$, then we take the closed subscheme of $\mathfrak N_{w'}$ by imposing the defining equations of $\bQ_G ( w ) \subset \bQ_G$ near $\mathbb O (w't_{\beta'})$. We set $E:= \Gm \ltimes ( ( \dot{w}'B (\dot{w}')^{-1} ) [\![z]\!] \cap \mathbf I ) \subset \Gm \ltimes \mathbf I$. The $E$-stabilizer of $p_{w't_{\beta'}}$ is $(\Gm \times H)$. Since $\bO (w't_{\beta'})$ is homogeneous under the $E$-action, we obtain a scheme $\mathfrak D''$ such that
\begin{equation}
( \mathfrak N_{w'} \cap \bQ_G ( w ) ) \cong \bO ( w' t_{\beta'} )\, \widehat{\times}\, \mathfrak D''\label{f-nbd-w''}
\end{equation}
from (\ref{f-nbd-w}), where the LHS is the fiber product that is given as the pullback of $\cO_{\bQ_G(w)}$ to $\mathfrak N_{w'}$. This is our fiber bundle $\mathfrak N_{\beta',w,w'}$, whose restriction to $\mathring{\sQ} ( \beta - \beta', w',v )$ makes the right square of (\ref{NQX}) into a Cartesian square. The scheme $\mathfrak D''$ is identified with the fiber of the fiber bundle structure on the formal completion of $\sQ ( \beta, w, v )$ along $\mathring{\sQ} ( \beta - \beta', w', v )$ as we can interpret the construction of $\mathfrak D''$ to be imposing the local defining equations of $\sQ ( \beta, w, v ) \subset \sQ ( \beta, e, v )$ near $\imath_{\beta'} ( \mathring{\sQ} ( \beta - \beta', w', v ) )$ to $\mathfrak D'$ by (\ref{si-Ric}). This construction must yield a locally trivial family as being the restriction of an $E$-equivariant closed subfamily of the case $w=e$. In particular, the left square of (\ref{NQX}) is also a Cartesian square (see Remark \ref{rcart} for more detail). The scheme $\mathfrak D''$ is the completion of a Noetherian ring since $\Spec \, \C [\sZ (\beta')]^{\wedge}_{\mathbf 0}$, and hence $\mathfrak D'$ is so. Thus, it is normal as $\sQ ( \beta, w )$, and hence the neighborhood of $\imath_{\beta-\beta'} ( \mathring{\sQ} (\beta', w' ) ) \subset \sQ ( \beta, w )$ is normal by \cite[Theorem 32.2 and Theorem 32.4]{Mat86}.

These complete the construction of (the family of) the required locally trivial fibrations.
\end{proof}

\begin{rem}\label{rcart}
We spell out the Cartesian structure of the square in the LHS of (\ref{NQX}) in Corollary \ref{tsg} more explicitly. Let $\bQ_G ( w )'$ and $\sQ (\beta, w, v )'$ denote the schemes obtained as the localizations of their structure sheaves along $\bO ( w't_{\beta'} ) \subset \bQ_G ( w )$ and $\mathring{\sQ} (\beta', w', v ) \subset \sQ (\beta, w, v )$, respectively\footnote{For a scheme $X$ and its locally closed affine subscheme $Z$, we refer the spectrum of
$$\Gamma ( Z, \cO_{X} ) := \varinjlim_{Z \subset U} \Gamma ( U, \cO_{X} ), \hskip 5mm \text{where $U$ runs over an open neighbourhood of $Z$},$$
as the scheme obtained as the localization of $\cO_X$ along $Z$.}. Since localizations commute with formal completions, the left square of (\ref{NQX}) yields a commutative diagram of schemes
\begin{equation}
\xymatrix{
\bQ_G ( w ) & \ar[l] \bQ_G ( w )' & \ar[l] \mathfrak N_{\beta',w,w'}\\
\sQ (\beta, w, v ) \ar[u] & \ar[l]\ar[u]\sQ (\beta, w, v )' & \ar[l] \mathfrak Q_{\beta',w,w'} ( \beta, v ) \ar[u]
}.\label{NQX2}
\end{equation}

 The left square of (\ref{NQX2}) is a Cartesian square arising from localizations (and hence their horizontal arrows are dominant morphisms). Thus, the whole diagram of (\ref{NQX2}) is a Cartesian diagram if its right square is.

We have a set
$$\Lambda := \{\phi_{v,i}\mid (v,i) \in S^-(vt_{\beta})\}$$
of set-theoretic defining equations of both of
$$\imath_{\beta'} ( \sQ ( \beta - \beta', w', v ) )\subset \bQ ( w' t_{\beta'} ) \hskip 5mm\text{and} \hskip 5mm\sQ ( \beta, w, v )\subset \bQ ( w )$$
offered by (\ref{si-Ric}) and Lemma \ref{phi-def}. In particular, we have a variant of the diagram (\ref{NQX2}) such that the vertical arrows are imposing $\Lambda$. This variant is a Cartesian diagram by construction.

By \cite[Theorem 32.2 and Theorem 32.4]{Mat86}, the scheme $\mathfrak Q_{\beta',w,w'} ( \beta, v )$ is reduced. Since the localizations of integral rings respect reducedness and taking the formal completions commute with quotient rings, taking the reduced induced structure of each item in the variant diagram yields (\ref{NQX2}), that is Cartesian.
\end{rem}

In view of Corollary \ref{tsg}, we refer $\mathfrak X_{\beta', w,w'}$ as the (infinitesimal) transversal slice of $\bO ( w' t_{\beta'})$ along $\bQ_G ( w )$, or that of $\mathring{\sQ} (\beta-\beta', w', v )$ along $\sQ (\beta, w, v )$.

\begin{cor}\label{cor:inf-real}
Let $\beta' \in Q^{\vee}_+$, and let $w,w' \in W$ such that $\bO ( w' t_{\beta'}) \subset \bQ_G ( w )$.
\begin{enumerate}
\item Assume that $\mathfrak X_{\beta',w',v}$ has a singularity worse than rational singularities. We have
$$\sQ ( \beta - \beta', w', v ) \subset \mathrm{Supp} \, \R^{> 0} ( \pi_{\beta, w, v} )_* \cO_{\sX ( \beta, w, v)} \hskip 5mm \text{if} \hskip 5mm w' \ge_\si v t_{\beta-\beta'};$$
\item Assume that $\mathfrak X_{\beta',w',v}$ has only rational singularities. If $\sQ ( \beta, w, v )$ has a singularity worse than rational singularities at a point in $\mathring{\sQ} ( \beta - \beta', w', v ) \subset \sQ ( \beta, w, v )$, then $\sQ ( \beta - \beta', w', v )$ itself has a singularity worse than rational singularities at the same point.
\end{enumerate}
\end{cor}

\begin{proof}
We set $\mathfrak X := \mathfrak X_{\beta', w,w'}$. Note that a product of spectrums of local rings have rational singularities if and only if each of them admits at worst rational singularities. Let $\mathfrak Q$ denote the formal completion of $\sQ ( \beta - \beta', w', v ) \subset \sQ ( \beta, w, v )$. In view of Lemma \ref{alg-fact}, the restriction of $\mathfrak Q$ to (the neighborhood of) $\mathring{\sQ} ( \beta - \beta', w', v )$ defines a locally trivial fibration whose base is $\mathring{\sQ} ( \beta - \beta', w', v )$ and whose fiber is $\mathfrak X$.

We consider the first assertion, and hence we assume that $\mathfrak X$ has a singularity worse than rational singularities. Then, $\mathfrak Q$ has singularity worse than rational singularities on $\mathring{\sQ} ( \beta - \beta', w', v )$. It follows that 
$$\sQ ( \beta - \beta', w', v ) = \mathrm{Supp} \, \bigl( \R^{> 0} ( \pi_{\beta, w, v} )_* \cO_{\sX ( \beta, w, v)} \bigr) \otimes_{\cO_{\sQ ( \beta, w, v)}} \cO_{\mathfrak Q}$$
as a subset of $\mathfrak Q$. Since the formal completion is flat (for Noetherian schemes), we conclude
$$
\sQ ( \beta - \beta', w', v ) \subset \mathrm{Supp} \, \R^{> 0} ( \pi_{\beta, w, v} )_* \cO_{\sX ( \beta, w, v)}.
$$
The last condition ($w' \ge_\si v t_{\beta-\beta'}$) is equivalent to $\sQ ( \beta - \beta', w', v ) \neq \emptyset$, and hence the first assertion holds.

We consider the second assertion, and hence we assume that $\mathfrak X$ has at worst rational singularities. Then, $\mathfrak Q$ has a singularity worse than rational singularities at a point of $\mathring{\sQ} ( \beta - \beta', w', v )$ if and only if $\mathring{\sQ} ( \beta - \beta', w', v )$ has so at the same point. Since $\mathfrak Q$ has a singularity worse than rational singularities at a point of $\mathring{\sQ} ( \beta - \beta', w', v )$ if and only if $\sQ ( \beta, w, v )$ has so at the same point (by Boutot's theorem), we conclude the second assertion.

These complete the proof.
\end{proof}

\subsection{Cohomology calculation for $\sX ( \beta, w )$}\label{subsec:ccalc}

\begin{lem}
For each $\beta \in Q^{\vee}_+, w, v \in W$, we have
\[
( \pi_{\beta, w, v} )_* \cO_{\sX ( \beta, w, v )} \cong \cO_{\sQ ( \beta, w, v)}.
\]
\end{lem}
\begin{proof}
This follows from the normality of $\sQ ( \beta, w, v )$ and the fact that all the fibers of $\pi_{\beta, w, v}$ are connected (\cite[Corollary 5.19]{Kat18d}).
\end{proof}

\begin{thm}\label{tr-rat}
Let $\beta' \in Q^{\vee}_+$, and let $w,w' \in W$ such that $\bO ( w' t_{\beta'}) \subset \bQ_G ( w )$. The scheme $\mathfrak X_{\beta', w,w'}$ in Corollary \ref{tsg} has at worst rational singularities.
\end{thm}

\begin{proof}
We set $\mathfrak X := \mathfrak X_{\beta', w,w'}$. We assume to the contrary to deduce contradiction. By Corollary \ref{cor:inf-real} 1), we have
\begin{equation}
\sQ ( \beta - \beta', w', v ) \subset \mathrm{Supp} \, \R^{> 0} ( \pi_{\beta, w, v} )_* \cO_{\sX ( \beta, w, v)}.\label{Q-inc}
\end{equation}
This containment is independent of the choice of $v$ and $\beta$ since the infinitesimal transversal slice $\mathfrak X$ is independent of the choice of $v \in W$ and $\beta \in Q^{\vee}_+$ whenever the LHS is nonempty by Corollary \ref{tsg}. We enlarge $\beta$ if necessary to guarantee the following two equivalent inequalities
$$w' >_{\si} t_{\beta-\beta'} \hskip 3mm \stackrel{(\ref{si-ord})}{\Leftrightarrow}\hskip 3mm w't_{\beta'} >_{\si} t_{\beta},$$
that yields $\sQ ( \beta - \beta', w', v ) \neq \emptyset$ for every $v \in W$, and some point (and hence general points) of $\sQ ( \beta - \beta', w', v )$ has no defect at $\infty$ and its value belongs to $\bO_\sB^\op ( v )$ by \cite[Lemma 8.5.1]{FM99}. Let $Z ( v )$ be an irreducible component of $\mathrm{Supp} \, \R^{> 0} ( \pi_{\beta, w, v} )_* \cO_{\sX ( \beta, w, v)}$ that contains $\sQ ( \beta - \beta', w', v )$. General points of $Z ( v )$ have no defect at $\infty$, and their values at $\infty$ belong to $\bO_\sB^\op ( v )$. We set $U := N \cap \dot{v} N \dot{v}^{-1}$. Then, the multiplication map $U \times \bO_\sB^\op ( v ) \subset \sB$ defines an embedding of an open dense subset. Thus, $\overline{U \times Z ( v )}$ is an irreducible component of the support of (\ref{Q-inc}) for $v = w_0$ that contains $\sQ ( \beta - \beta', w', w_0 )$. Therefore, we obtain a family $\{Z ( v )\}_{v \in W}$ of irreducible components of $\mathrm{Supp} \, \R^{> 0} ( \pi_{\beta, w, v} )_* \cO_{\sX ( \beta, w, v)}$ such that $\sQ ( \beta - \beta', w', v ) \subset Z ( v )$, $Z ( v ) \subset Z ( u )$ if $v \le u$, and $\overline{B Z ( v )}$ is independent of $v$.

Consider the smallest $\sQ ( \theta, t, v )$ ($\theta \in Q^{\vee}_+$ and $t \in W$), embedded into $\sQ ( \beta,w,v )$ through $\imath_{\beta-\theta}$, that contains $Z ( v )$. Here $\mathring{\sQ} ( \theta, t, v )$ contains a point of $Z ( v )$ as $Z ( v )$ is irreducible and the inclusion relations among $\sQ ( \beta - \bullet, \bullet, v )$ obey the closure relation of $\bI$-orbits of $\bQ_G^{\mathrm{rat}}$ described in Theorem \ref{si-Bruhat} by their irreducibility and (\ref{si-Ric}). Hence, the condition $B_i Z (v) \subset Z(v)$ and $\mathop{SL} ( 2, i ) Z ( v ) \not\subset Z ( v )$ $(i \in \tI_\af)$ is achieved if it holds for $\sQ ( \theta, t, v ) \subset \bQ_G^\ra$. Since $\overline{B Z ( v )}$ is common for every $v \in W$, we deduce that $\theta$ and $u$ are independent of $v$. Thus, we can rearrange $v$ if necessary to find $i \in \tI_\af$ such that $\mathop{SL} ( 2, i ) Z ( v ) \not\subset Z ( v )$ for $i \neq 0$ and $s_i v < v$, or $i = 0$ and $s_{\vartheta} v > v$, inside $\sQ ( \beta )$ ($i \neq 0$) or $\sQ ( \beta - w^{-1} \vartheta )$ ($i = 0$, see Theorem \ref{Qnorm} 5)). Since $B_i$ acts on $\sQ ( \beta, w, v )$ (by Theorem \ref{Qnorm}), it follows that $Z ( v )$ is $B_i$-stable. We have $s_i w < w$ ($i \neq 0$) or $s_\vartheta w > w$ ($i = 0$) as otherwise $\mathop{SL} ( 2, i )$ acts on $\sQ ( \beta, w, v )$ and hence on $Z ( v )$, that is a contradiction.

The map $\pi_i$ restricted to $\mathop{SL} ( 2, i ) \times^{B_i} Z ( v )$ is birational onto its image. In particular, there exists a Zariski open subset $V \subset \mathop{SL} ( 2, i ) Z ( v )$ such that $\pi_i^{-1} ( V ) \cap \mathop{SL} ( 2, i ) \times^{B_i} Z ( v )$ forms an irreducible component of
$$\pi_i^{-1} ( V ) \cap \left( \mathop{SL} ( 2, i ) \times^{B_i} \mathrm{Supp}\, \R^{> 0} ( \pi_{\beta, w, v} )_* \cO_{\sX ( \beta, w, v)} \right) \subset \mathop{SL} ( 2, i ) \times^{B_i} \sQ ( \beta, w, v ).$$
Hence, $( \pi_i )_*$ sends the inflation of $\R^{> 0} ( \pi_{\beta, w, v} )_* \cO_{\sX ( \beta, w, v)}$ to a non-zero sheaf whose support contains $\mathop{SL} ( 2, i ) Z ( v )$.

If the variety $\sQ ( \beta, w, v )$ is $B_i$-stable for $i \in \tI$, then the $G$-action on $\sX ( \beta )$ restricts to the $B_i$-action on $\sX ( \beta, w, v )$. If the variety $\sQ ( \beta, w, v )$ is $B_0$-stable, then there exists a smooth projective variety $\sX' ( \beta, w, v )$ with the $B_0$-action that yields a $B_0$-equivariant resolution of singularities of $\sQ ( \beta, w, v )$ (see e.g. \cite[Corollary 7.6.3]{Vil92}). We can replace $\pi_{\beta, w, v} : \sX ( \beta, w, v ) \rightarrow \sQ ( \beta, w, v )$ with $\sX' ( \beta, w, v ) \to \sQ ( \beta, w, v )$ in this case since both of $\sX ( \beta, w, v )$ and $\sX' ( \beta, w, v )$ have rational singularities and there exists yet another resolution of singularities of $\sQ ( \beta, w, v )$ that dominates both. Let us consider the map
\begin{equation}
\mathop{SL} ( 2, i ) \times ^{B_i} \sX ( \beta, w, v ) \rightarrow \mathop{SL} ( 2, i ) \times ^{B_i} \sQ ( \beta, w, v ) \stackrel{\pi_i}{\longrightarrow} \sQ ( \gamma, u, v ),\label{sX-inf}
\end{equation}
where the first map is the inflation of $\pi_{\beta, w, v}$, $\gamma = \beta$ and $u = s_i w$ ($i \neq 0$), or $\gamma = \beta - w^{-1} \vartheta^{\vee}$ and $u = s_{\vartheta} w$ ($i = 0$). Since $\sX ( \beta, w, v )$ has rational singularities, so is $\mathop{SL} ( 2, i ) \times ^{B_i} \sX ( \beta, w, v )$. In addition, the composition map (\ref{sX-inf}) is birational and projective. Thus, it is another resolution of $\sQ ( \gamma, u, v )$ by a variety that has rational singularities. Therefore, we can replace $\sX ( \gamma, u, v )$ with $\mathop{SL} ( 2, i ) \times ^{B_i} \sX ( \beta, w, v )$ to compute $\R^{\bullet} ( \pi_{\gamma, u, v} )_* \cO_{\sX ( \gamma, u, v )}$. Applying the Leray spectral sequence to (\ref{sX-inf}) using Theorem \ref{Qnorm} 4), 5), we have
\begin{align}\nonumber
\R^{0} ( \pi_{\gamma, u, v} )_* \cO_{\sX ( \gamma, u, v )} & \cong \cO_{\sQ ( \gamma, u, v )} \hskip 5mm \text{and}\\
\R^{>0} ( \pi_{\gamma, u, v} )_* \cO_{\sX ( \gamma, u, v )} & \neq \{0\}.\label{i0-push}
\end{align}
Moreover, the support of (\ref{i0-push}) contains $\mathop{SL} ( 2, i ) Z ( v )$. By construction, general points of $\mathop{SL} ( 2, i ) Z ( v )$ have no defect at $\infty$, and their values belong to $\bO_{\sB}^{\op} ( v )$. Therefore, an irreducible component $Z' ( v )$ of the support of (\ref{i0-push}) that contains $\mathop{SL} ( 2, i ) Z ( v )$ again comes as a family $\{Z' ( v ) \}_{v \in W}$ such that $Z' ( v ) \subset Z' ( u )$ if $v \le u$, and $\overline{B Z' ( v )}$ is independent of $v$ (in particular, we have $Z' ( v )$ even if $i \neq 0$ and $s_i v > v$, or $i = 0$ and $s_{\vartheta} v < v$). Thus, we can repeat the above procedure by replacing $\sQ ( \beta, w, v )$ with $\sQ (\gamma, u, v )$ and $\{Z ( v )\}_{v \in W}$ with $\{ Z' ( v ) \}_{v \in W}$. Note that we eventually attain $Z' ( v ) = \mathop{SL} ( 2, i ) Z ( v )$ for any application of the above procedures as the strict inclusion forces
$$(0 < ) \hskip 3mm \mathrm{codim}_{\sQ ( \gamma, u, v )} \, Z' ( v ) < \mathrm{codim}_{\sQ ( \gamma, u, v )} \, \mathop{SL} ( 2, i )Z ( v ) = \mathrm{codim}_{\sQ ( \beta, w, v )} \, Z ( v ),$$
that cannot be repeated infinitely many times.

Consider the smallest $\sQ ( \theta, t, v )$ ($\theta \in Q^{\vee}_+$ and $t \in W$) that contains $Z ( v )$ again. As discussed above, $\theta$ and $t$ are independent of the choice of $v \in W$. We choose $i \in \tI_\af$ (and $v \in W$) such that $B_i$ preserves $\sQ ( \theta, t, v )$ and $\mathop{SL} ( 2, i ) Z ( v ) \not\subset Z ( v )$. Then, $B_i$ preserves $\sQ ( \theta, t, v )$ and $\mathop{SL} ( 2, i ) \sQ ( \theta, t, v ) \not\subset \sQ ( \theta, t, v )$ by the above discussion. It follows that $\sQ ( \theta, t, v )$ is transformed to $\sQ ( \theta, s_i t, v )$ $(i \neq 0)$ or $\sQ ( \theta - t^{-1} \vartheta, s_\vartheta t, v )$ $(i = 0)$ by an application of the above procedure.

In view of \cite[Theorem 4.6]{Kat18} (cf. arguments around there), we repeat these procedures if necessary to assume $t = w_0$. The condition $\mathop{SL} ( 2, i ) Z ( v ) \not\subset Z ( v )$ implies $\mathop{SL} ( 2, i ) \sQ ( \beta, w, v ) \not\subset \sQ ( \beta, w, v )$ (otherwise $Z ( v )$ is $\mathop{SL} ( 2, i )$-stable) asserts that $s_i t < t$ implies $s_i w < w$ ($i \in \tI$). In case $t = w_0$, this implies $w = w_0$. Again by repeating the above procedures, we can rearrange the situation to assume $w = t = e$ and $v = w_0$. In this case, we have $\sQ (\beta) = \sQ ( \beta, e, w_0 )$, that has rational singularities by Theorem \ref{rat-res}. From this, we find a contradiction on the existence of $Z ( v )$. This in turn implies a contradiction to the existence of $(\beta ', w')$ such that $\mathfrak X$ has worse than rational singularities. Therefore, we conclude that our $\mathfrak X$ have at worst rational singularities.
\end{proof}

The following is our main geometric result in this paper, that is an extension of Theorem \ref{rat-res} due to Braverman-Finkelberg \cite{BF14a,BF14b,BF14c} from $\sQ ( \beta )$ to an arbitrary Richardson variety of $\bQ_G^\ra$:

\begin{thm}\label{Q-rat-sing}
For each $w, v \in W$ and $\beta \in Q^{\vee}_+$, the variety $\sQ ( \beta,w, v )$ has rational singularities. In particular, $\sQ ( \beta,w, v )$ is Cohen-Macaulay.
\end{thm}

\begin{rem}
We use only a special case of Theorem \ref{Q-rat-sing} ($v=w_0$) in the rest of this paper.
\end{rem}

For each $\beta \in Q^{\vee}_+, w, v \in W$, we set
$$\ringring{\sQ} ( \beta, w, v ) := \left\{ f \in  \sQ ( \beta ) \middle| 0, \infty \not\in [D], f (0) \in \bO_\sB ( w ), f(\infty) \in \bO_\sB^{\op} ( v )\right\}.$$
In view of \cite[Corollary 5.4]{Kat18d}, the inclusion $\ringring{\sQ} ( \beta, w, v ) \subset \sQ ( \beta, w, v )$ is open dense.

\begin{proof}[Proof of Theorem \ref{Q-rat-sing}]
We assume to the contrary to deduce contradiction. Namely, we assume that $\sQ ( \beta, w, v )$ has singularities worse than rational singularities. By Theorem \ref{tr-rat} and Corollary \ref{cor:inf-real} 2), if the worse than rational singularities locus of $\sQ ( \beta, w, v )$ is contained in some $\sQ ( \gamma, u, v )$, then the variety $\sQ ( \gamma, u, v )$ itself must have singularities worse than rational singularities. Therefore, by rearranging $( \beta, w )$ if necessary, we can assume
\begin{equation}
\mathring{\sQ} ( \beta, w, v ) \cap \mathrm{Supp} \, \R^{> 0} ( \pi_{\beta, w, v} )_* \cO_{\sX ( \beta, w, v)} \neq \emptyset.\label{genSing}
\end{equation}
In the above discussion on (\ref{genSing}), we can swap $z$ with $z^{-1}$, and $B$ with $B^-$, that makes us to rearrange $\beta$ and $v$ instead of $\beta$ and $w$. Since the (infinitesimal) transversal slices of $\mathring{\sQ}$ and $\ringring{\sQ}$ are in common, we can further rearrange $\beta$ and $v$ to conclude
\begin{equation}
\ringring{\sQ} ( \beta, w, v ) \cap \mathrm{Supp} \, \R^{> 0} ( \pi_{\beta, w, v} )_* \cO_{\sX ( \beta, w, v)} \neq \emptyset\label{genSing2}
\end{equation}
by
$$\mathring{\sQ} ( \beta, w, v ) = \bigsqcup_{0 \le \beta' \le \beta} \bigsqcup_{\scriptsize{\begin{matrix}u \in W\\ u t_{\beta'} \ge_\si vt_{\beta}\end{matrix}}}\ringring{\sQ} ( \beta - \beta', w, u).$$
The variety $\ringring{\sQ} ( \beta, w, v )$ admits only $H$-action, but its ambient space $\sQ ( \beta )$ admits a $G$-action. The action of $( N^- \cap \dot{w} N \dot{w}^{-1} )$ applied to the subspace $\bO_\sB ( w ) \subset \sB$, as well as the action of $( N \cap \dot{v} N \dot{v}^{-1} )$ applied to the subspace $\bO_\sB^{\op} ( v ) \subset \sB$ have trivial stabilizers. As a consequence, the action of $( N^- \cap \dot{w} N \dot{w}^{-1} )$ applied to the subspace $\ringring{\sQ} ( \beta, w, u ) \subset \sQ ( \beta )$, and the action of $( N \cap \dot{v} N \dot{v}^{-1})$ applied to the subspace $\ringring{\sQ} ( \beta, u, v ) \subset \sQ ( \beta )$ also have trivial stabilizers for each $u \in W$. In view of the fact that $\bO_\sB ( w )$ is $N$-stable, the action of $(N \cap \dot{v} N \dot{v}^{-1})$ on $\sQ ( \beta )$ preserves $\sqcup_{u \in W} \ringring{\sQ} ( \beta, w, u )$. Therefore, we deduce an embedding
$$( N^- \cap \dot{w} N \dot{w}^{-1} ) \times( N \cap \dot{v} N \dot{v}^{-1} ) \times \ringring{\sQ} ( \beta, w, v ) \ni (n_1,n_2,x) \mapsto n_1 n_2 x \in \sQ ( \beta ).$$
By the dimension comparison using (\ref{dimQ}), we deduce that this embedding must be open dense in $\sQ ( \beta ) = \sQ (\beta, e,w_0)$. The locus $Y$ on which the singularity of $\ringring{\sQ} ( \beta, w, v )$ is worse than rational singularities gives rise to the locus
$$( N^- \cap \dot{w} N \dot{w}^{-1} ) \times( N \cap \dot{v} N \dot{v}^{-1} ) \times Y \subset \sQ ( \beta )$$
on which the singularity of $\sQ ( \beta )$ is worse than rational singularities.

However, the variety $\sQ ( \beta )$ has only rational singularities (Theorem \ref{rat-res}). Thus, the locus of $Y \subset \sQ ( \beta,w,v )$ on which $\sQ ( \beta, w,v )$ has worse than rational singularity must be empty. Hence $\sQ ( \beta, w, v )$ must have rational singularities. The latter assertion follows from \cite[Theorem 5.10]{KM98}.
\end{proof}

\begin{cor}\label{vanish}
Let $\la \in P_{+}$. For each $w \in W$ and $\beta \in Q^{\vee}$, we have
$$H^{>0} ( \sX ( \beta,w ), \cO_{\sX ( \beta,w )} ( \la ) ) = \{0\}.$$
\end{cor}

\begin{proof}
In view of Theorem \ref{Q-rat-sing}, we apply \cite[Theorem 5.10]{KM98} and the Leray spectral sequence to reduce the assertion to $H^{>0} ( \sQ ( \beta,w ), \cO_{\sQ ( \beta,w )} ( \la ) ) = \{0\}$. This is Theorem \ref{Qnorm} 2).
\end{proof}

\begin{prop}\label{w-calc}
Let $w \in W$ and $\la \in P_{+}$. We have
$$\lim _{\beta \to \infty} \chi_q (\sX ( \beta,w ), \cO_{\sX ( \beta,w )} ( \la ) ) = \gch \, H ^0 ( \bQ_G ( w ), \cO_{\bQ_G ( w )} ( \la ) ).$$
\end{prop}

\begin{proof}
By Corollary \ref{vanish}, we have
$$\chi_q (\sX ( \beta,w ), \cO_{\sX ( \beta,w )} ( \la ) ) = \gch \, H ^0 ( \sX ( \beta,w ), \cO_{\sX ( \beta,w )} ( \la ) )$$
for every $\beta \in Q^{\vee}_+$.

By Theorem \ref{Q-rat-sing}, we deduce
$$H ^0 ( \sX ( \beta,w ), \cO_{\sX ( \beta,w )} ( \la ) ) = H ^0 ( \sQ ( \beta,w ), \cO_{\sQ ( \beta,w )} ( \la ) )$$
for every $\la \in P_+$ and $\beta \in Q^{\vee}_+$.

By Theorem \ref{Qnorm} 2), we have
$$\lim _{\beta \to \infty} \chi_q (\sX ( \beta,w ), \cO_{\sX ( \beta,w )} ( \la ) ) = \lim _{\beta \to \infty} \chi_q (\sQ ( \beta,w ), \cO_{\sQ ( \beta,w )} ( \la ) ) \hskip 5mm \la \in P_{+}$$
and it is uniquely determined by Theorem \ref{Qnorm} 3). In addition, the comparison of Theorem \ref{Qnorm} 3) with \cite[Theorem 4.12]{Kat18} implies
$$\lim _{\beta \to \infty} \chi_q (\sQ ( \beta,w ), \cO_{\sQ ( \beta,w )} ( \la ) ) = \gch \, H ^0 ( \bQ_G ( w ), \cO_{\bQ_G ( w )} ( \la ) ) \hskip 5mm \la \in P_{+}.$$

Combining these implies the desired equality.
\end{proof}

\subsection{The image of Schubert classes under $\Psi$}\label{subsec:pbt}

\begin{thm}\label{b-trans}
The map $\Psi$ constructed in Theorem \ref{wJcomp} satisfies
$$\Psi ( [\cO_{\sB ( w )}] ) = [\cO_{\bQ ( w )}] \hskip 5mm w \in W.$$
\end{thm}

The rest of this subsection is devoted to the proof of Theorem \ref{b-trans}. In this subsection, $\otimes$ is understood to be $\otimes_{\cO_Z}$, where $Z$ is the variety we are considering.

We consider the $\C [q^{\pm 1}] P$-valued functional $F_{\beta}^{\la} ( \bullet )$ on $( \C [q^{\pm 1}] \otimes_{\C} K_H ( \sB ) ) [\![Q^{\vee}_+]\!]$ with parameters $\beta \in Q^{\vee}_+$ and $\la \in P_{+}$:
\begin{align*}
\sum_{\beta \in Q^{\vee}_+} F_{\beta}^{\la} ( \bullet ) Q^{\beta} & : = \sum_{\gamma \in Q^{\vee}_+} \chi_q ( \sX (\gamma), \cO _{\sX (\gamma)} ( \la ) \otimes \mathtt{ev}_1^* ( \bullet ) \otimes \mathtt{ev}_2^* ( \cO _{\sB} ) ) Q^{\gamma}\\
 & = \chi ( \sB, T ( \prod_{i \in \tI} A_i ( q )^{- \left< \al_i^{\vee}, \la \right>} ( [ \cO_{\sB} ] ) ) \cdot \overline{T} ( \bullet )),
\end{align*}
where the second equality is a reformulation of \cite[Proposition 2.13]{IMT15} and the last term is connected to the quantum $K$-theoretic product by Theorem \ref{GLK} and Theorem \ref{graph-to-map}. Note that this collection of functionals $\{F_{\beta}^{\la} ( \bullet )\}_{\beta,\la}$ is uniquely determined by the calculations from \S \ref{subsec:ccalc} and Theorem \ref{zas}
\begin{align*}
\sum_{\beta \in Q^{\vee}_+} F _{\beta}^{\la} ( [\cO_{\sB}] ) Q^{\beta} & = \sum_{\beta \in Q^{\vee}_+} \chi_q ( \sX (\beta), \cO _{\sX (\beta)} ( \la ) \otimes \mathtt{ev}_1^* ( \cO_{\sB} ) \otimes \mathtt{ev}_2^* ( \cO _{\sB} ) ) Q^{\beta}\\
= & \sum_{\beta \in Q^{\vee}_+} \chi_q ( \sQ ( \beta ), \cO_{\sQ( \beta)} ( \la ) ) Q^{\beta} = D_{w_0} ( J ' ( Qq^{\la}, q ) e^{w_0 \la} J' ( Q, q^{-1} ) ),
\end{align*}
as $\sum _{\beta} F_{\beta}^{\la} ( \bullet ) Q^{\beta}$ commutes with the $\C [q^{\pm 1}] P$-action and the right $\C Q^{\vee}$-action, and intertwines the shift operator $A_i ( q )$ with the line bundle twist by $\cO_{\sX (\beta)} ( - \varpi_i )$ for each $i \in \tI$. The last two expressions assert that our functional $F^\la_{\beta}$ is the same one as that employed in the proof of Theorem \ref{wJcomp}.

For each $a \in ( \C [q^{\pm 1}] \otimes_{\C} K_H ( \sB ) ) [\![Q^{\vee}_+]\!]$, we have the class $\Psi ( a\MID _{q = 1} ) \in K_{H} ( \bQ_G )$ written as
$$\Psi ( a\MID _{q = 1} ) = \sum_{w \in W_\af} c^w (a) [\cO_{\bQ_G ( w )}] \hskip 5mm w \in W_\af, c^w ( a ) \in \C P.$$
In view of the proof of Theorem \ref{wJcomp}, the coefficients $c^w ( a ) \in \C P$ are characterized as
$$c^w ( a ) = c^w_q ( a ) \MID _{q = 1}$$
if we have elements $c_q^w ( a ) \in \C [q^{\pm 1}] P$ ($e \ge_\si w \in W_\af$) determined by
$$\lim_{\beta \to \infty} F_{\beta}^{\la} ( a ) = \sum_{w} c^w_q ( a ) \, \gch \, \Gamma ( \bQ_G ( w ), \cO_{\bQ_G ( w )} ( \la ) ) \hskip 5mm \la \in P_+.$$

Thus, we have
\begin{align}\nonumber
\lim_{\beta \to \infty} F_{\beta}^{\la} ( [\cO _{\sB(w)}] ) & = \lim_{\beta \to \infty} \chi_q ( \sX (\beta), \cO _{\sX (\beta)} ( \la ) \otimes \mathtt{ev}_1^* ( \cO _{\sB(w)} ) \otimes \mathtt{ev}_2^* ( \cO _{\sB} ) ) \\ \label{rewrite}
& = \lim_{\beta \to \infty} \chi_q ( \sX (\beta), \cO _{\sX (\beta)} ( \la ) \otimes \mathtt{ev}_1^* ( \cO _{\sB(w)} ) )\\\nonumber
&  = \lim_{\beta \to \infty} \chi_q ( \sX ( \beta, w ), \cO _{\sX ( \beta, w )} ( \la )) = \gch \, H^0 ( \bQ_G ( w ), \cO_{\bQ_G ( w )} ( \la ) )
\end{align}
for each $\la \in P_{+}$ and $w \in W$, where the last equality is Proposition \ref{w-calc}.

Therefore, we conclude
$$\Psi ( [\cO_{\sB ( w )}]) = [\cO_{\bQ_G ( w )}] \hskip 5mm w \in W$$
as required.

\subsection{Consequences}\label{subsec:conseq}

Since Theorem \ref{sline} is used only when we reformulate Theorem \ref{wJcomp} into Corollary \ref{Jcomp} (the last result in \S \ref{subsec:ide}), we obtain an alternative proof of the following:

\begin{thm}[$=$ Theorem \ref{sline} due to Anderson-Chen-Tseng]
For each $i \in \tI$, we have $A_i ( q ) ( [\cO_{\sB}] ) = [\cO_\sB ( - \varpi_i )]$.
\end{thm}

\begin{proof}
By (the proof of) Theorem \ref{wJcomp} and Remark \ref{qwJ}, we know that $\Psi_q ( A_i ( q ) ( [\cO_{\sB}] ) ) = [\cO_{\bQ_G (e)} ( - \varpi_i )]$. Now we argue as:
\begin{align*}
A_i ( q ) ( [\cO_{\sB}] ) & =  \Psi^{-1}_q ( [\cO_{\bQ_G ( e )} ( - \varpi_i )] ) & & \\
& =  e^{-\varpi_i} \Psi^{-1}_q ( [\cO_{\bQ_G ( e )}] - [\cO_{\bQ_G ( s_i )}] ) & & \text{by Lemma \ref{divc}}&\\
& =  e^{-\varpi_i} ( [\cO_{\sB} ( e )] - [\cO_{\sB( s_i )} ] ) & & \text{by Theorem \ref{b-trans}}\\
& =  [\cO_{\sB} (- \varpi_i)] & & \text{by (\ref{LSrem})}.
\end{align*}
These imply the result.
\end{proof}

\begin{cor}[Finiteness of the shift operators]\label{qfin}
For each $i \in \tI$ and $w \in W$, the element $A_i ( q ) ( [\cO_{\sB ( w )}] )$ is a finite $\C [q^{\pm 1}]P$-linear combination of $\{[\cO_{\sB ( w )}]Q^{\beta}\}_{w \in W, \beta \in Q^{\vee}_+}$.
\end{cor}

\begin{proof}
By Remark \ref{qwJ} and Theorem \ref{b-trans}, the problem reduces to the corresponding problem in $K_{\Gm \ltimes \bI} ( \bQ_G^{\ra})$. The latter is explained in either \cite[Theorem 3.7]{Kat20} (as in Corollary \ref{LLMSfin}) or \cite[Theorem 1]{NOS18} (as an explicit formula), though author's original reasoning is by the finiteness of the (global version of the) decomposition procedure in \cite{FMO18} (as their global generalized Weyl modules are exactly $\Gamma ( \bQ_G ( w ), \cO_{\bQ_G (w)} ( \la ) )^*$; see e.g. \cite[\S 5]{Kat18}).
\end{proof}

\begin{cor}[Conjectured by Lam-Li-Mihalcea-Shimozono \cite{LLMS17}]\label{LLMSc}
We have a natural $\C P$-algebra dense embedding
$$\Psi^{-1} \circ \Phi : K_H ( \Gr )_\lo \hookrightarrow qK _H ( \mathscr B )_\lo,$$
such that
\begin{equation}
\Psi^{-1} \circ \Phi ( [\cO_{\Gr_{wt_{\beta}}}] \odot [\cO_{\Gr_{\gamma}}]^{-1} ) = [\cO_{\sB ( w )}]Q^{\beta - \gamma} \hskip 5mm w \in W\label{image}
\end{equation}
holds for every $\beta, \gamma \in Q^{\vee}_<$.
\end{cor}

\begin{proof}
For the first assertion, combine Theorem \ref{main} and Corollary \ref{Jcomp} to obtain the map $\Psi^{-1} \circ \Phi$, that have dense image. Note that the both sides are rings and the identity $[\cO_{\Gr_0}]$ goes to the identity $[\cO_{\sB}]$. The map $\Psi^{-1} \circ \Phi$ commutes with the natural $Q^{\vee}$-actions given by $\mathtt t_{\gamma}$ and $Q^{\gamma}$ for each $\gamma \in Q^{\vee}$ by Theorem \ref{comm} and Corollary \ref{Jcomp}. Moreover, the action of $\Theta_i$ (see \S \ref{pmain}) and the quantum multiplication by $[\cO _{\sB} ( -\varpi_i )]$ corresponds for each $i \in \tI$ (by Theorem \ref{main} and Corollary \ref{Jcomp}). Therefore, the $\odot$-multiplication by the element $\bh_i$ and $\star$-multiplication by $[\cO_{\sB ( s_i )}] = ( [\cO_{\sB}] - e^{\varpi_i} [\cO _{\sB} ( -\varpi_i )] )$ coincide for each $i \in \tI$. Since the ring $K_H ( \Gr )_\lo$ is generated by $\{\bh_i\}_{i \in \tI}$ up to the $\C P$-action and $\{ \mathtt t _{\gamma} \}_{\gamma}$-action (Remark \ref{gen}), we conclude that $\Psi^{-1} \circ \Phi$ is a ring embedding.

For the second assertion, note that Theorem \ref{comm} asserts that
$$\Phi ( [\cO_{\Gr_{wt_{\beta}}}] \odot [\cO_{\Gr_{\gamma}}]^{\pm 1} ) = [\cO_{\bQ_G ( w t_{\beta \pm \gamma} )}] \hskip 5mm w \in W$$
for each $\beta, \gamma \in Q^{\vee}_<$ (cf. Lemma \ref{mt-op}). From this, we derive
\begin{align*}
\Psi^{-1} \circ \Phi ( [\cO_{\Gr_{wt_{\beta}}}] \odot [\cO_{\Gr_{\gamma}}]^{\pm 1} ) & = \Psi^{-1} ( [ \cO_{\bQ_G ( w )} ]) Q^{\beta \pm \gamma} & \text{by Corollary \ref{Jcomp}}\\
& = [\cO_{\sB ( w )}] Q^{\beta \pm \gamma}& \text{by Theorem \ref{b-trans}}.
\end{align*}
This yields the desired equality.
\end{proof}

\begin{cor}\label{triangle}
We have a commutative diagram, whose bottom arrow is an embedding of rings:
$$
\xymatrix{
& K_{H} ( \bQ_G^{\ra} ) & \\
K _H ( \Gr ) _\lo \ar@{^{(}->}[rr] \ar@{^{(}->}[ru]^{\Phi} & & qK_H ( \sB )_\lo \ar[lu]_{\Psi}^{\cong}
}.$$
This induces an isomorphism
$$K _H ( \Gr ) _\lo \stackrel{\cong}{\longrightarrow} \bigoplus _{w \in W, \beta \in Q^{\vee}} \C P [\cO_{\sB (w)}] Q^{\beta} \subset qK_H ( \sB )_\lo.$$
In addition, the map $\Phi$ is an injective $K_H ( \mathrm{pt} ) \otimes \C Q^{\vee}$-module homomorphism, and $K_{H} ( \bQ_G^{\ra} )$ acquires the structure of a ring from $K_H ( \Gr )$ or $qK_H ( \sB )$ $($cf. Remark \ref{Kra}$)$.
\end{cor}

\begin{proof}
The first assertion combines Corollary \ref{LLMSc} and its proof. From (\ref{image}), we conclude the second assertion. The last assertion follows as both $qK_H ( \sB )$ and $K_H ( \Gr )$ are rings.
\end{proof}

In view of \cite{LLMS17}, we obtain another proof of the finiteness of quantum $K$-theory of $\sB$ originally proved in Anderson-Chen-Tseng \cite{ACT17,ACT18}. We reproduce the reasoning here for the sake of reference:

\begin{cor}[Anderson-Chen-Tseng \cite{ACT17,ACT18}]\label{LLMSfin}
For each $w,v \in W$, we have
$$[\cO_{\sB ( w )}] \star [\cO_{\sB ( v )}] \in \bigoplus_{\beta \in Q^{\vee}_+, u \in W} \C P [\cO_{\sB ( u )}] Q^{\beta}.$$
In other words, the multiplication rule of $qK_H ( \sB )$ is finite.
\end{cor}


\begin{proof}[Proof of Corollary \ref{LLMSfin} due to Lam-Li-Mihalcea-Shimozono $\cite{LLMS17}$]
By Corollary \ref{LLMSc} (cf. Theorem \ref{LSS-formula}), the assertion follows from
\begin{equation}
[\cO_{\Gr _{\beta}}] \odot [\cO_{\Gr _{\gamma}}] \in \bigoplus_{\kappa \in Q^{\vee}} \C P [\cO_{\Gr _{\kappa}}] \hskip 5mm \forall \beta, \gamma \in Q^{\vee}.\label{ACTodotGr}
\end{equation}
By definition, the LHS of (\ref{ACTodotGr}) is a product inside the ring $\sC$ that has $\{[\cO_{\Gr _{\kappa}}]\}_{\kappa}$ as its $\C P$-basis (Theorem \ref{LSS10}). Hence, the assertion follows.
\end{proof}

\medskip

{\small
\hskip -5.25mm {\bf Acknowledgement:} The author would like to thank Michael Finkelberg and Tatsuyuki Hikita for discussions, David Anderson and Hiroshi Iritani for helpful correspondence, and Thomas Lam for pointing out inaccuracies in a previous version of this paper. He is also grateful to the reviewers for their thoughtful comments and suggestions. The author also thanks Daisuke Sagaki and Daniel Orr for their collaborations. This research was supported in part by JSPS KAKENHI Grant Numbers JP26287004 and JP19H01782, and by a JSPS Grant-in-Aid for Challenging Research (Exploratory) 24K21192.}


\begin{table}[hbtp]
  \caption{List of main global symbols (ordered by their first appearances)}
  \centering
  \begin{tabular}{lc}
    \hline
\S \ref{sec:prelim} & $\gdim, f \mapsto \overline{f}$\\
\S \ref{subsec:prelim} & $G,B,H,N,N^-,B^-,E[z],E[\![z]\!],E(\!(z)\!),\bI,\bI^-,P,\Delta,\Pi,\left< \bullet,\bullet\right>,Q^{\vee},Q^{\vee}_+$\\
& $\le,P_+,\tI,\al_i,W,\rho,\Delta_\af,\varpi_i,\vartheta,\tI_\af,W_{\af},\ell,t_{\beta},\mathop{SL}(2,i),B_i,P_i,W_{\af}^-,Q^{\vee}_<$\\
& $\le,\le_\si,L(\la),\ch,\gch,\sB,\bO_\sB (w),\sB(w),\bO^\op_\sB (w),\sB^\op(w),\cO_\sB ( \la ),K_H ( \sB )$\\
\S \ref{subsec:nDAHA} & $\sH,D_i,\sS,\sA,D_w$\\
\S \ref{subsec:Gr} & $\Gr,\Fl,\bO_w^\Fl,\bO_\beta^\Gr,\Gr_\beta,\Fl_w,\sC,\odot,\Omega K,\mathbf h_i,\mathtt t_\gamma$\\
\S \ref{subsec:sif} & $\bQ_G^\ra,\Upsilon,\bO(w),\bQ_G(w),p_w,\cO_{\bQ_G^\ra} ( \la ),\bQ_G^-(w),\bI^\flat,\g[z],\bI',\bW( \la )$\\
& $P_\la,\phi_{u,i},S(u),S^-(u), \chi_q$\\
\S \ref{subsec:eK} & $K_{\Gm \ltimes \bI}(\bQ_G^\ra), K_{H}(\bQ_G), K_{H}(\bQ_G^\ra),\mathrm{Fun}_P,\mathrm{Fun}_P^{\mathrm{neg}},\Theta,\Xi(\la),H_i$\\
\S \ref{subsec:GQL} & $\sGB_{n,\beta},\sB_{n,\beta},\widetilde{\mathtt{ev}}_{j},\mathtt{ev}_{j},\chi_q, Q^{\beta},Q_i,\left< \bullet \right>_\GW^q,\left< \bullet \right>_\GW$\\
\S \ref{subsec:eqK} & $qK_H(\sB),\star,p_i,q^{Q_i\partial_{Q_i}},T,\mathbb L,A_i (q),J(Q,q),a_i$\\
    \hline
\S \ref{sec:Gr} & $D_i^\sharp$ \hskip 2mm (\S \ref{subsec:trans}), \hskip 5mm $\Phi$ \hskip 2mm (Theorem \ref{pmain})\\
    \hline
\S \ref{subsec:QM} & $|D|,|D|_x,\sQ ( \beta ), \sQ ( \beta, w ),\mathring{\sQ} ( \beta ),\mathring{\sQ} ( \beta,w ),\imath_\gamma,\cO_{\sQ ( \beta,w )} ( \la )$\\
\S \ref{subsec:fact} & $\sZ( \beta ), C^{(\beta)},\mathbb A_x^1,u_i ( f,D ),\phi_i(f,D;z),\mathfrak f^{\beta},\mathbf 0, \mathcal S_y$\\
\S \ref{subsec:ide} & $\Psi, F^\la_\beta$ \hskip 2mm (Theorem \ref{wJcomp} and its proof), \hskip 5mm $\Psi_q$ \hskip 2mm (Remark \ref{qwJ})\\
\hline
\S \ref{subsec:GQL2} & $\pi_{n,\beta},\pi_{\beta},\cO_{\sGB_{n,\beta}}(\la),\sX(\beta),\sQ(\beta,w,v),\sX(\beta,w,v),\pi_{\beta,w,v},\mathring{\sQ} (\beta,w,v)$\\
\S \ref{subsec:ccalc} & $\ringring{\sQ} (\beta,w,v)$\\
  \end{tabular}
\end{table}

{\footnotesize
\bibliography{kmref}

\begin{thebibliography}{10}

\bibitem{ACT17}
David Anderson, Linda Chen, and Hsian-Hua Tseng.
\newblock On the quantum {$K$}-ring of the flag manifold.
\newblock arXiv:1711.08414, 2017.

\bibitem{ACT18}
David Anderson, Linda Chen, Hsian-Hua Tseng, and Hiroshi Iritani.
\newblock The quantum {$K$}-theory of a homogeneous space is finite.
\newblock {\em International Mathematics Research Notices}, published online
  \texttt{https://doi.org/10.1093/imrn/rnaa108}, 2020.

\bibitem{AK06}
Sergey Arkhipov and Mikhail Kapranov.
\newblock Toric arc schemes and quantum cohomology of toric varieties.
\newblock {\em Math. Ann.}, 4:953--964, 2006.

\bibitem{BB73}
Andrzej Bia{\l}ynicki-Birula.
\newblock Some theorems on actions of algebraic groups.
\newblock {\em Ann. of Math. (2)}, 98:480--497, 1973.

\bibitem{BB05}
Anders Bjorner and Francesco Brenti.
\newblock {\em Combinatorics of {C}oxeter Groups}, volume 231 of {\em Graduate
  Text in Mathematics}.
\newblock Springer Science+Business Media, Inc., 2005.

\bibitem{Bou87}
J.-F. Boutot.
\newblock Singularites rationnelles et quotients par les groupes reductifs.
\newblock {\em Invent. Math.}, 88:65--68, 1987.

\bibitem{BFGM}
A.~Braverman, M.~Finkelberg, D.~Gaitsgory, and I.~Mirkovi\'c.
\newblock Intersection cohomology of {D}rinfeld's compactifications.
\newblock {\em Selecta Math. (N.S.)}, 8(3):381--418, 2002.

\bibitem{BF14a}
Alexander Braverman and Michael Finkelberg.
\newblock Semi-infinite {S}chubert varieties and quantum {$K$}-theory of flag
  manifolds.
\newblock {\em J. Amer. Math. Soc.}, 27(4):1147--1168, 2014.

\bibitem{BF14b}
Alexander Braverman and Michael Finkelberg.
\newblock Weyl modules and {$q$}-{W}hittaker functions.
\newblock {\em Math. Ann.}, 359(1-2):45--59, 2014.

\bibitem{BF14c}
Alexander Braverman and Michael Finkelberg.
\newblock Twisted zastava and $q$-{W}hittaker functions.
\newblock {\em J. London Math. Soc.}, 96(2):309--325, 2017, arXiv:1410.2365.

\bibitem{BFG06}
Alexander Braverman, Michael Finkelberg, and Dennis Gaitsgory.
\newblock {Uhlenbeck spaces via affine Lie algebras}.
\newblock In {\em The unity of mathematics}, pages 17--135. 2006.

\bibitem{BCMP}
Anders~S. Buch, Pierre-Emmanuel Chaput, Leonardo~C. Mihalcea, and Nicolas
  Perrin.
\newblock Finiteness of cominuscule quantum {$K$}-theory.
\newblock {\em Ann. Sci. \'Ec. Norm. Sup\'er. (4)}, 46(3):477--494, 2013.

\bibitem{CI15}
Vyjayanthi Chari and Bogdan Ion.
\newblock B{GG} reciprocity for current algebras.
\newblock {\em Compos. Math.}, 151(7):1265--1287, 2015.

\bibitem{CG97}
Neil Chriss and Victor Ginzburg.
\newblock {\em Representation theory and complex geometry}.
\newblock Modern Birkh\"auser Classics. Birkh\"auser Boston, Inc., Boston, MA,
  2010.
\newblock Reprint of the 1997 edition.

\bibitem{FFKM}
Boris Feigin, Michael Finkelberg, Alexander Kuznetsov, and Ivan Mirkovi{\'c}.
\newblock Semi-infinite flags. {II}. {L}ocal and global intersection cohomology
  of quasimaps' spaces.
\newblock In {\em Differential topology, infinite-dimensional {L}ie algebras,
  and applications}, volume 194 of {\em Amer. Math. Soc. Transl. Ser. 2}, pages
  113--148. Amer. Math. Soc., Providence, RI, 1999.

\bibitem{FF}
Boris Feigin and Edward Frenkel.
\newblock Affine {K}ac-{M}oody algebras and semi-infinite flag manifold.
\newblock {\em Comm. Math. Phys.}, 128:161--189, 1990.

\bibitem{FMO18}
Evgeny Feigin, Ievgen Makedonskyi, and Daniel Orr.
\newblock Generalized {W}eyl modules and nonsymmetric $q$-{W}hittaker
  functions.
\newblock {\em Adv. Math.}, 339:997--1033, 2018.

\bibitem{FM99}
Michael Finkelberg and Ivan Mirkovi{\'c}.
\newblock Semi-infinite flags. {I}. {C}ase of global curve {$\bold P^1$}.
\newblock In {\em Differential topology, infinite-dimensional {L}ie algebras,
  and applications}, volume 194 of {\em Amer. Math. Soc. Transl. Ser. 2}, pages
  81--112. Amer. Math. Soc., Providence, RI, 1999.

\bibitem{FP95}
William Fulton and Rahul Pandharipande.
\newblock Notes on stable maps and quantum cohomology.
\newblock In J\'anos Koll\'ar, Robert Lazarsfeld, and David Morrison, editors,
  {\em Algebraic Geometry: Santa Cruz 1995}. Amer Mathematical Society, 1995.

\bibitem{Giv94}
Alexander Givental.
\newblock {Homological geometry and mirror symmetry}.
\newblock {\em Proceedings of the International Congress of Mathematicians
  1994}, pages 472--480, 1995.

\bibitem{Giv00}
Alexander Givental.
\newblock On the {WDVV} equation in quantum {$K$}-theory.
\newblock {\em Michigan Math. J.}, 48:295--304, 2000.

\bibitem{GL03}
Alexander Givental and Yuan-Pin Lee.
\newblock Quantum {$K$}-theory on flag manifolds, finite-difference {T}oda
  lattices and quantum groups.
\newblock {\em Invent. Math.}, 151(1):193--219, 2003.

\bibitem{Giv96}
Alexander~B. Givental.
\newblock Equivariant {G}romov-{W}itten invariants.
\newblock {\em Internat. Math. Res. Notices}, 13:613--663, 1996.

\bibitem{EGAI}
A.~Grothendieck.
\newblock \'{E}l\'ements de g\'eom\'etrie alg\'ebrique. {I}. {L}e langage des
  sch\'emas.
\newblock {\em Inst. Hautes \'Etudes Sci. Publ. Math.}, (4):228, 1960.

\bibitem{EGAIII-1}
A.~Grothendieck.
\newblock \'{E}l\'ements de g\'eom\'etrie alg\'ebrique. {III}. \'{E}tude
  cohomologique des faisceaux coh\'erents. {I}.
\newblock {\em Inst. Hautes \'Etudes Sci. Publ. Math.}, (11):167, 1961.

\bibitem{Ha98}
Nobuo Hara.
\newblock A characterization of rational singularities in terms of injectivity
  of {F}robenius map.
\newblock {\em Amer. J. Math.}, 120:981--996, 1998.

\bibitem{Har77}
Robin {Hartshorne}.
\newblock {\em {Algebraic geometry.}}, volume~52 of {\em Graduate Text in
  Mathematics}.
\newblock Springer- Verlag, 1977.

\bibitem{Iri06}
Hiroshi Iritani.
\newblock Quantum {$D$}-modules and equivariant {F}loer theory for free loop
  spaces.
\newblock {\em Math. Z.}, 252(3):577--622, 2006.

\bibitem{IMT15}
Hiroshi Iritani, Todor Milanov, and Valentin Tonita.
\newblock {Reconstruction and convergence in quantum K-Theory via Difference
  Equations}.
\newblock {\em International Mathematics Research Notices},
  2015(11):2887--2937, 2015, 1309.3750.

\bibitem{Jos85}
Anthony Joseph.
\newblock On the {D}emazure character formula.
\newblock {\em Ann. Sci. \'Ecole Norm. Sup. (4)}, 18(3):389--419, 1985.

\bibitem{Kat18}
Syu Kato.
\newblock Demazure character formula for semi-infinite flag varieties.
\newblock {\em Math. Ann.}, 371(3):1769--1801, 2018, arXiv:1605.0279.

\bibitem{Kat19a}
Syu Kato.
\newblock On quantum {$K$}-groups of partial flag manifolds.
\newblock arXiv:1906.09343, 2019.

\bibitem{Kat20}
Syu Kato.
\newblock Darboux coordinates on the {BFM} spaces.
\newblock arXiv:2008.01310, 2020.

\bibitem{Kat18d}
Syu Kato.
\newblock Frobenius splitting of {S}chubert varieties of semi-infinite flag
  manifolds.
\newblock {\em Forum of Mathematics, Pi}, 9:e5, 2021.

\bibitem{Kat21}
Syu Kato.
\newblock The formal model of semi-infinite flag manifolds.
\newblock In {\em Proc. Int. Cong. Math. 2022}, number~3, pages 1600--1622. EMS
  Press, Berlin, 2023.

\bibitem{KL17}
Syu Kato and Sergey Loktev.
\newblock A {W}eyl module stratification of integrable representations.
\newblock {\em Comm. Math. Phys.}, 368:113--141, 2019.
\newblock arXiv:1712.03508.

\bibitem{KNS17}
Syu Kato, Satoshi Naito, and Daisuke Sagaki.
\newblock Equivariant {$K$}-theory of semi-infinite flag manifolds and the
  {P}ieri-{C}hevalley formula.
\newblock {\em Duke Math. J.}, 169(13):2421--2500, 2020.

\bibitem{Kaw78}
Tetsuro Kawasaki.
\newblock The signature theorem for {$V$}-manifolds.
\newblock {\em Topology}, 17:75--83, 1978.

\bibitem{KP01}
B.~Kim and R.~Pandharipande.
\newblock The connectedness of the moduli space of maps to homo- geneous
  spaces.
\newblock In {\em Symplectic geometry and mirror symmetry (Seoul, 2000)}, pages
  187--201. World Sci. Publ., River Edge, NJ, 2001.

\bibitem{Kol86}
J{\'a}nos Koll{\'a}r.
\newblock Higher direct images of dualizing sheaves, {I}.
\newblock {\em Ann. of Math. (2)}, 123(1):11--42, 1986.

\bibitem{KM98}
J{\'a}nos Koll{\'a}r and Shigefumi Mori.
\newblock {\em Birational geometry of algebraic varieties}, volume 134 of {\em
  Cambridge Tracts in Mathematics}.
\newblock Cambridge University Press, Cambridge, 1998.
\newblock With the collaboration of C. H. Clemens and A. Corti, Translated from
  the 1998 Japanese original.

\bibitem{KM94}
M.~Kontsevich and Yu. Manin.
\newblock Gromov-{W}itten classes, quantum cohomology, and enumerative
  geometry.
\newblock {\em Comm. Math. Phys.}, 164(3):525--562, 1994.

\bibitem{KK90}
Bertram Kostant and Shrawan Kumar.
\newblock {$T$}-equivariant {$K$}-theory of generalized flag varieties.
\newblock {\em J. Differential Geom.}, 32(2):549--603, 1990.

\bibitem{Kum02}
Shrawan Kumar.
\newblock {\em Kac-{M}oody groups, their flag varieties and representation
  theory}, volume 204 of {\em Progress in Mathematics}.
\newblock Birkh\"auser Boston, Inc., Boston, MA, 2002.

\bibitem{Kum17}
Shrawan Kumar.
\newblock Positivity in {$T$}-equivariant {$K$}-theory of flag varieties
  associated to {K}ac--{M}oody groups.
\newblock {\em Journal of the European Mathematical Society}, 19(8):2469--2519,
  2017.

\bibitem{LLMS17}
Thomas Lam, Changzheng Li, Leonardo~C. Mihalcea, and Mark Shimozono.
\newblock A conjectural {P}eterson isomorphism in {$K$}-theory.
\newblock {\em J. Algebra}, 513:326--343, 2018.

\bibitem{LSS10}
Thomas Lam, Anne Schilling, and Mark Shimozono.
\newblock {$K$}-theory {S}chubert calculus of the affine {G}rassmannian.
\newblock {\em Compositio Mathematica}, 146(4):811--852, 2010, 0901.1506.

\bibitem{LS10}
Thomas Lam and Mark Shimozono.
\newblock Quantum cohomology of {$G/P$} and homology of affine {G}rassmannian.
\newblock {\em Acta Mathematica}, 204(1):49--90, 2010, arXiv:0705.1386v1.

\bibitem{Lee04}
Yuan-Pin Lee.
\newblock Quantum {$K$}-theory, {I}: {F}oundations.
\newblock {\em Duke Math. J.}, 121(3):389--424, 2004.

\bibitem{LP04}
Yuan-Pin Lee and Rahul Pandharipande.
\newblock A reconstruction theorem in quantum cohomology and quantum
  {$K$}-theory.
\newblock {\em Amer. J. Math.}, 126(6):1367--1379, 2004.

\bibitem{Lus80}
George Lusztig.
\newblock Hecke algebras and {J}antzen's generic decomposition patterns.
\newblock {\em Adv. in Math.}, 37(2):121--164, 1980.

\bibitem{LusICM}
George Lusztig.
\newblock Intersection cohomology methods in representation theory.
\newblock In {\em Proceedings Proceedings of the ICM 1990 (Kyoto)}, volume I,
  II, pages 155--174, 1991.

\bibitem{Mac03}
I.~G. Macdonald.
\newblock {\em Affine {H}ecke algebras and orthogonal polynomials}, volume 157
  of {\em Cambridge Tracts in Mathematics}.
\newblock Cambridge University Press, Cambridge, 2003.

\bibitem{Mat86}
Hideyuki Matsumura.
\newblock {\em Commutative ring theory}, volume~8 of {\em Cambridge Studies in
  Advanced Mathematics}.
\newblock Cambridge University Press, Cambridge, 1986.
\newblock Translated from the Japanese by M. Reid.

\bibitem{MS97}
V.~B. Mehta and V.~Srinivas.
\newblock A characterization of rational singularities.
\newblock {\em Asian Journal of Mathematics}, 1(2):249--271, 1997.

\bibitem{NOS18}
Satoshi Naito, Daniel Orr, and Daisuke Sagaki.
\newblock Pieri-{C}hevalley formula for anti-dominant weights in the
  equivariant {$K$}-theory of semi-infinite flag manifolds.
\newblock arXiv:1808.01468.

\bibitem{Pet97}
Dale Peterson.
\newblock Quantum cohomology of {$G/P$}.
\newblock Lecture at MIT, 1997.

\bibitem{Pon39}
Lev Pontryagin.
\newblock Homologies in compact {L}ie groups.
\newblock {\em Recueil Math\'ematique (Matematicheskii Sbornik). New Series.},
  6(48):389--422, 1939.

\bibitem{PS86}
Andrew Pressley and Graeme Segal.
\newblock {\em Loop groups}.
\newblock Oxford Mathematical Monographs. The Clarendon Press, Oxford
  University Press, New York, 1986.

\bibitem{Ser55}
Jean-Pierre Serre.
\newblock G\'eom\'etrie alg\'ebrique et g\'eom\'etrie analytique.
\newblock {\em Annales de l'Institut Fourier}, 6:1--42, 1955.

\bibitem{stacks-project}
The {Stacks project authors}.
\newblock The stacks project.
\newblock \url{https://stacks.math.columbia.edu}, 2018.

\bibitem{Tho98}
Jasper~F. Thomsen.
\newblock Irreducibility of {$M_{0,n} ( G/P, \beta)$}.
\newblock {\em Internat. J. Math.}, 9(3):367--376, 1998.

\bibitem{Vil92}
Orlando~E. Villamayor~U.
\newblock Patching local uniformizations.
\newblock {\em Ann. Sci. \'Ec. Norm. Sup\'er. (4)}, 25(6):629--677, 1992.

\end{thebibliography}
\bibliographystyle{hplain}}
\end{document}